\documentclass[twoside]{article}
\usepackage[letterpaper, left=3.5cm, right=3.5cm, top=4cm, bottom=2.5cm]{geometry}
\usepackage{graphicx}
\usepackage{amsmath,amssymb,amsthm}
\usepackage{authblk}
\usepackage[utf8]{inputenc}
\usepackage{microtype}
\usepackage{mathrsfs}  
\usepackage{enumitem}
\usepackage{calc}
\usepackage{esint}
\usepackage{fancyhdr}
\usepackage{xcolor}
\usepackage[labelfont = bf]{caption}
\usepackage{doi}
\usepackage{subfigure}
\usepackage[numbers]{natbib}
\usepackage{float}
\usepackage{stmaryrd}
\usepackage{mathtools}
\usepackage{booktabs}
\usepackage{colortbl}
\usepackage{tabulary}
\usepackage{diagbox}
\usepackage{caption}

\newtheorem{theorem}{Theorem}[section]
\numberwithin{equation}{section}
\newtheorem{proposition}[theorem]{Proposition}

\newtheorem{corollary}[theorem]{Corollary}
\newtheorem{remark}[theorem]{Remark}
\newtheorem{lemma}[theorem]{Lemma}

\newtheorem{algorithm}[theorem]{Algorithm}
\newtheorem{assumption}[theorem]{Assumption}

\usepackage{titlesec}
\titleformat{\section}{\normalfont\scshape\centering}{\thesection.}{0.5em}{}
\titleformat*{\subsection}{\itshape}
\titleformat*{\subsubsection}{\itshape}

\providecommand{\keywords}[1]
{
	{\small\textit{Keywords:} #1}
}

\providecommand{\MSC}[1]
{
	{\small\textit{AMS MSC (2020):~~} #1}
}

\usepackage{hyperref}
\hypersetup{
	colorlinks=true,
	linkcolor=blue,
	citecolor = blue,
	filecolor=magenta,      
	urlcolor=cyan,
}

\providecommand{\jumptmp}[2]{#1\llbracket{#2}#1\rrbracket}
\providecommand{\jump}[1]{\jumptmp{}{#1}}

\AtBeginDocument{%
	\def\MR#1{}
}

\newcommand{\AAA}{\boldsymbol{\mathcal{A}}}

\begin{document}
	\setlength{\abovedisplayskip}{5.5pt}
	\setlength{\belowdisplayskip}{5.5pt}
	\setlength{\abovedisplayshortskip}{5.5pt}
	\setlength{\belowdisplayshortskip}{5.5pt}

	\title{Error analysis for a Crouzeix--Raviart approximation\\ of the $p$-Dirichlet problem}
	\author[1]{Alex Kaltenbach\thanks{Email: \texttt{alex.kaltenbach@mathematik.uni-freiburg.de\vspace*{-15mm}}}}
	\date{October 21, 2022}
	\affil[1]{\small{Department of Applied Mathematics, University of Freiburg, Ernst--Zermelo--Stra\ss e~1, 79104 Freiburg im Breisgau, Germany}}
	\clearpage\maketitle
    \thispagestyle{empty}

	\pagestyle{fancy}
	\fancyhf{}
	\fancyheadoffset{0cm}
	\addtolength{\headheight}{-0.25cm}
	\renewcommand{\headrulewidth}{0pt} 
	\renewcommand{\footrulewidth}{0pt}
	\fancyhead[CO]{\textsc{Error analysis for a CR approximation of the $p$-Dirichlet problem}}
	\fancyhead[CE]{\textsc{A. Kaltenbach}}
	\fancyhead[R]{\thepage}
	\fancyfoot[R]{}
	
	\begin{abstract}
		In the present paper, we examine a Crouzeix--Raviart approximation for non-linear~partial differential equations having a $(p,\delta)$-structure for some $p\in (1,\infty)$ and $\delta\ge 0$.~We~establish a~\mbox{priori} error estimates, which are optimal for all $p\in (1,\infty)$ and $\delta\ge 0$, medius error estimates, i.e., best-approximation results, and a primal-dual a posteriori error estimate, which~is~both reliable and efficient. 
  The theoretical findings are supported by numerical~experiments.
	\end{abstract}

	\keywords{\hspace{-0.15mm}$p$-Dirichlet \hspace{-0.15mm}problem;  \hspace{-0.15mm}Crouzeix--Raviart \hspace{-0.15mm}element; \hspace{-0.15mm}a \hspace{-0.15mm}priori \hspace{-0.15mm}error \hspace{-0.15mm}analysis;~\hspace{-0.15mm}medius~\hspace{-0.15mm}error~\hspace{-0.15mm}\mbox{analysis}; a posteriori error analysis.}
	
	\MSC{49M29; 65N15; 65N50.}

	\section{Introduction}
   	
   	\qquad We examine the numerical approximation of a non-linear system of $p$-Dirichlet type, i.e.,
   	\begin{align}\label{eq:pDirichlet}
   	    \begin{aligned}
   	        -\mathrm{div}\,\AAA(\nabla u)&=f&&\quad\textup{ in }\Omega\,,\\
   	        u&= 0&&\quad\textup{ on }\Gamma_D\,,\\
   	       \AAA(\nabla u)\cdot n &= 0&&\quad\textup{ on }\Gamma_N\,,
   	    \end{aligned}
   	\end{align}
   	using \hspace*{-0.15mm}the \hspace*{-0.15mm}Crouzeix--Raviart \hspace*{-0.15mm}element, \hspace*{-0.15mm}cf.\ \cite{CR73}. \hspace*{-0.15mm}More \hspace*{-0.15mm}precisely, \hspace*{-0.15mm}for \hspace*{-0.15mm}a \hspace*{-0.15mm}given \hspace*{-0.15mm}right-hand~\hspace*{-0.15mm}side~\hspace*{-0.15mm}${f\!\in\! L^{p'}(\Omega)}$, $p'\coloneqq \frac{p}{p-1}$, $p\in (1,\infty)$, we seek $\smash{u\in W^{1,p}_D(\Omega)\coloneqq \{v\in W^{1,p}(\Omega)\mid \textup{tr}\,v=0\textup{ in }L^p(\Gamma_D)\}}$ solving~\eqref{eq:pDirichlet}. Here, $\smash{\Omega\subseteq \mathbb{R}^d}$, ${d\in \mathbb{N}}$, is a bounded Lipschitz domain, whose topological boundary $\partial\Omega$ is disjointly divided into a Dirichlet part $\Gamma_D$ and a Neumann part $\Gamma_N$, and the non-linear~\mbox{operator}~${\AAA\colon\mathbb{R}^d\to \mathbb{R}^d}$  has a $(p,\delta)$-structure for some $p\in (1,\infty)$ and $\delta\ge 0$. The relevant~example,~falling~into~this~class, for every $a\in \mathbb{R}^d$, is defined by 
   	\begin{align}\label{example}
   	    \AAA(a)\coloneqq (\delta+\vert a\vert)^{p-2}a\,.
   	\end{align}
   	\qquad Problems \hspace{-0.1mm}of \hspace{-0.1mm}type \hspace{-0.1mm}\eqref{eq:pDirichlet} \hspace{-0.1mm}arise \hspace{-0.1mm}in \hspace{-0.1mm}various~\hspace{-0.1mm}mathematical~\hspace{-0.1mm}models~\hspace{-0.1mm}describing~\hspace{-0.1mm}physical~\hspace{-0.1mm}processes,~\hspace{-0.1mm}e.g., in plasticity, bimaterial problems in elastic-plastic mechanics, non-Newtonian~fluid~mechanics, blood rheology, and glaciology, cf. \cite{MNRR96,Liu99,HA18}. Most \hspace*{-0.15mm}of \hspace*{-0.15mm}these \hspace*{-0.15mm}models \hspace*{-0.15mm}admit \hspace*{-0.15mm}equivalent \hspace*{-0.15mm}formulations \hspace*{-0.15mm}as \hspace*{-0.15mm}convex~\hspace*{-0.15mm}minimization~\hspace*{-0.15mm}problems, e.g., for the  non-linear system \eqref{eq:pDirichlet}, if the non-linear operator $\AAA\colon\mathbb{R}^d\to \mathbb{R}^d$ possesses~a~potential, i.e., there is a strictly convex function $\varphi\colon \mathbb{R}_{\ge 0}\to \mathbb{R}_{\ge 0}$~such~that $D(\varphi\circ \vert \cdot\vert)(a)=\AAA(a)$~for~all~${a\in \mathbb{R}^d}$, e.g., \eqref{example}, then each solution  $u\in  \smash{W^{1,p}_D(\Omega)}$ of \eqref{eq:pDirichlet}~is~unique minimizer of the energy functional ${I\colon \smash{W^{1,p}_D(\Omega)}\to \mathbb{R}_{\ge 0}}$, for every $v\in\smash{ W^{1,p}_D(\Omega)}$ defined by
        \begin{align}
            I(v)\coloneqq \int_{\Omega}{\varphi(\vert \nabla v\vert)\,\mathrm{d}x}-\int_{\Omega}{f\,v\,\mathrm{d}x}\,,\label{eq:pDirichletMin}
        \end{align}
        and vice-versa, leading to a primal and a dual formulation of \eqref{eq:pDirichlet}, as well as to convex duality relations.

    \subsection{Related contributions}\vspace{-0.5mm}\enlargethispage{7mm}
    
    \qquad The finite element approximation of \eqref{eq:pDirichlet} 
    has been intensively analyzed by numerous authors:
    The first contributions addressing a priori error estimation as well as a posteriori estimation,  measured in \hspace*{-0.1mm}the \hspace*{-0.1mm}conventional \hspace*{-0.1mm}$W^{1,p}(\Omega)$-semi-norm,
     \hspace*{-0.1mm}can \hspace*{-0.1mm}be \hspace*{-0.1mm}found \hspace*{-0.1mm}in \hspace*{-0.1mm}\cite{Cia78,BL1993a,Ver95,Pad97}.~\hspace*{-0.1mm}Sharper~\hspace*{-0.1mm}(optimal) a priori error estimates for the conforming Lagrange finite element method  applied to \eqref{eq:pDirichlet}, measured in the so-called quasi-norm or natural distance, resp.,
     were established in \cite{baliu94,eb-liu,DR07}. 
     Furthermore, residual a posteriori error estimates for the conforming Lagrange finite element method and the non-conforming Crouzeix--Raviart finite~element~method applied \hspace*{-0.1mm}to \hspace*{-0.1mm}\eqref{eq:pDirichlet}, \hspace*{-0.1mm}each \hspace*{-0.1mm}measured \hspace*{-0.1mm}in \hspace*{-0.1mm}the \hspace*{-0.1mm}quasi-norm \hspace*{-0.1mm}or \hspace*{-0.1mm}natural \hspace*{-0.1mm}distance, \hspace*{-0.1mm}resp., \hspace*{-0.1mm}were \hspace*{-0.1mm}established~\hspace*{-0.1mm}in~\hspace*{-0.1mm}\mbox{\cite{LY01A,LY01B,CLY06,DK08,BDK12,BM20}}. In addition, there \hspace*{-0.1mm}exist \hspace*{-0.1mm}optimal \hspace*{-0.1mm}a \hspace*{-0.1mm}priori \hspace*{-0.1mm}error \hspace*{-0.1mm}estimates \hspace*{-0.1mm}for \hspace*{-0.1mm}Discontinuous \hspace*{-0.1mm}Galerkin \hspace*{-0.1mm}(DG) \hspace*{-0.1mm}methods, \hspace*{-0.1mm}cf.~\hspace*{-0.1mm}\cite{dkrt-ldg,mret18,kr-phi-ldg}. In \cite{LLC18}, if $p\ge  2$ and $\delta=0$ in \eqref{example}, a priori and a posteriori error estimates~for~the Crouzeix--Raviart \hspace*{-0.1mm}finite \hspace*{-0.1mm}element \hspace*{-0.1mm}method \hspace*{-0.1mm}applied \hspace*{-0.1mm}to \hspace*{-0.1mm}\eqref{eq:pDirichlet}, \hspace*{-0.1mm}measured~\hspace*{-0.1mm}in~\hspace*{-0.1mm}the~\hspace*{-0.1mm}\mbox{quasi-norm},~\hspace*{-0.1mm}were~\hspace*{-0.1mm}\mbox{derived}.
    However, in \cite{LLC18}, the optimality of the a priori error estimates and the efficiency of the a posteriori error estimates remain unclear. 
    In \cite{Bre15}, if $p=2$ and $\delta=0$ in \eqref{example}, by means of a so-called medius error analysis, i.e., a best-approximation result, for the Crouzeix--Raviart finite element method applied to \eqref{eq:pDirichlet}, 
    an optimal a priori error estimate was derived. In particular, this medius error analysis reveals that the performances of the 
    conforming Lagrange finite element method and the non-conforming Crouzeix--Raviart \hspace*{-0.1mm}finite \hspace*{-0.1mm}element \hspace*{-0.1mm}method \hspace*{-0.1mm}applied~\hspace*{-0.1mm}to~\hspace*{-0.1mm}\eqref{eq:pDirichlet}~\hspace*{-0.1mm}are~\hspace*{-0.1mm}comparable. However,~\hspace*{-0.1mm}for~\hspace*{-0.1mm}the~\hspace*{-0.1mm}case~\hspace*{-0.1mm}${p\neq 2}$, to the best of the author's knowledge, such results~are~still~pending. More~precisely, there is neither a medius error analysis, i.e., a best-approximation result, available, nor an optimal a priori error estimate, measured in the quasi-norm or natural distance, resp.~It~is the purpose of this paper to fill this lacuna.\vspace{-0.5mm}

    \subsection{New contribution}\vspace{-0.5mm}
    
    \qquad  Deriving local efficiency estimates in terms of shifted $N$-functions and deploying the so-called node-averaging quasi-interpolation operator, cf. \cite{Osw93,Sus96}, we generalize the medius error analysis in 
    \cite{Bre15} from $p=2$ and $\delta=0$ in \eqref{example}, i.e., $\AAA=\textup{id}_{\mathbb{R}^d}\colon \mathbb{R}^d\to \mathbb{R}^d$, to general non-linear operators $\AAA\colon\hspace*{-0.1em} \mathbb{R}^d\hspace*{-0.1em}\to\hspace*{-0.1em} \mathbb{R}^d$ \hspace*{-0.1mm}having \hspace*{-0.1mm}a \hspace*{-0.1mm}$(p,\delta)$-structure \hspace*{-0.1mm}for \hspace*{-0.1mm}$p\hspace*{-0.1em}\in\hspace*{-0.1em} (1,\infty)$ \hspace*{-0.1mm}and \hspace*{-0.1mm}$\delta\hspace*{-0.1em}\ge\hspace*{-0.1em} 0$, \hspace*{-0.1mm}e.g.,~\hspace*{-0.1mm}\eqref{example}.~\hspace*{-0.1mm}This~\hspace*{-0.1mm}medius~\hspace*{-0.1mm}error~\hspace*{-0.1mm}analysis,  reveals that the performances of the 
    conforming Lagrange finite element method applied to \eqref{eq:pDirichlet} and the non-conforming Crouzeix--Raviart finite element method applied~to~\eqref{eq:pDirichlet}~are~comparable. As \hspace*{-0.1mm}a \hspace*{-0.1mm}result, \hspace*{-0.1mm}we \hspace*{-0.1mm}get 
    \hspace*{-0.1mm}a \hspace*{-0.1mm}priori \hspace*{-0.1mm}error \hspace*{-0.1mm}estimates \hspace*{-0.1mm}for \hspace*{-0.1mm}the \hspace*{-0.1mm}Crouzeix--Raviart \hspace*{-0.1mm}finite \hspace*{-0.1mm}element~\hspace*{-0.1mm}method~\hspace*{-0.1mm}applied to \eqref{eq:pDirichlet}, which are optimal for all $p\hspace*{-0.1em}\in\hspace*{-0.1em} (1,\infty)$ and $\delta\hspace*{-0.1em}\ge\hspace*{-0.1em} 0$. If $\AAA\colon\mathbb{R}^d\hspace*{-0.1em}\to \hspace*{-0.1em}\mathbb{R}^d$ has a potential~and,~thus, \eqref{eq:pDirichlet} admits an equivalent formulation as a convex minimization problem, cf. \eqref{eq:pDirichletMin}, then~we~have~access to a (discrete) convex duality theory, and \eqref{eq:pDirichlet} as well as the Crouzeix--Raviart approximation of \eqref{eq:pDirichlet} admit dual formulations with a dual solution and a discrete dual solution,~resp.,~cf.~\cite{LLC18,Bar21,BKAFEM}. We establish a priori error estimates for the error between the dual solution and the discrete dual solution, measured in the so-called conjugate natural distance, which are optimal for all $p\in (1,\infty)$ and $\delta\ge 0$.
    One further by-product of the medius error analysis consists in an efficiency~type~result, which allows to establish the efficiency of a so-called primal-dual a posteriori error estimator, which was recently derived in \cite{BKAFEM} and is also applicable~if~$\AAA\colon\mathbb{R}^d\to \mathbb{R}^d$~has~a~potential.\vspace{-0.5mm}

   	\subsection{Outline}\vspace{-0.5mm}
   	
   \qquad \textit{This article is organized as follows:} \!In Section \ref{sec:preliminaries}, we introduce the employed notation,~the~\mbox{basic} assumptions on the non-linear operator $\AAA\colon\mathbb{R}^d\to \mathbb{R}^d$ and its corresponding properties, the relevant finite element spaces, and give brief review of the continuous and the discrete $p$-Dirichlet~problem. 
   In Section \ref{sec:medius}, we establish a medius error analysis, i.e., best-approximation result, for the Crouzeix--Raviart finite element method applied to \eqref{eq:pDirichlet}.
   In Section \ref{sec:a_priori}, by~means~of~this~medius~error~analysis, we derive a priori error estimates for the  Crouzeix--Raviart finite element method applied~to~\eqref{eq:pDirichlet}, which are optimal for all  $p\hspace{-0.1em}\in\hspace{-0.1em} (1,\infty)$ and $\delta\hspace{-0.1em}\ge\hspace{-0.1em} 0$. In Section \ref{sec:a_posteriori}, we establish the efficiency~of~a~\mbox{so-called} primal-dual a posteriori error estimator.
   In Section \ref{sec:experiments}, we confirm our theoretical findings via numerical experiments.
    
    \newpage
   	 
	\section{Preliminaries}\label{sec:preliminaries}\vspace{-0.5mm}

	\qquad Throughout the entire article, if not otherwise specified, we always denote by ${\Omega\subseteq \mathbb{R}^d}$,~${d\in\mathbb{N}}$, a bounded polyhedral Lipschitz domain, whose topological boundary $\partial\Omega$ is disjointly divided into a closed Dirichlet part $\Gamma_D$, for which we always assume that $\vert \Gamma_D\vert>0$\footnote{For a (Lebesgue) measurable set $M\subseteq \mathbb{R}^d$, $d\in \mathbb{N}$, we denote by $\vert M\vert $ its $d$-dimensional Lebesgue measure. For a $(d-1)$-dimensional submanifold $M\subseteq \mathbb{R}^d$, $d\in \mathbb{N}$, we denote by $\vert M\vert $ its $(d-1)$-dimensional~Hausdorff~measure.},~and~a~Neumann~part~$\Gamma_N$,~i.e.,  ${\partial\Omega=\Gamma_D\cup\Gamma_N}$~and~${\emptyset=\Gamma_D\cap\Gamma_N}$. We employ $c, C>0$ to denote generic constants, that may change from line
to line, but are not depending on the crucial quantities. Moreover,
we~write~${f\sim g}$ if and only if there exist constants $c,C>0$ such
that $c\, f \le g\le C\, f$.
	\subsection{Standard function spaces}\vspace{-0.5mm}

	\qquad For $p\in \left[1,\infty\right]$ and $l\in \mathbb{N}$, we employ the standard notations\vspace*{-0.5mm}\footnote{\textcolor{black}{Here, $W^{\smash{-\frac{1}{p},p}}(\partial\Omega)\coloneqq (W^{\smash{1-\frac{1}{p'},p'}}(\partial\Omega))^*$.}\vspace{-20mm}}
	\begin{align*}
		\begin{aligned}
		W^{1,p}_D(\Omega;\mathbb{R}^l)&\coloneqq \big\{v\in L^p(\Omega;\mathbb{R}^l)&&\hspace*{-3.25mm}\mid \nabla v\in L^p(\Omega;\mathbb{R}^{l\times d}),\, \textup{tr}\,v=0\text{ in }L^p(\Gamma_D;\mathbb{R}^l)\big\}\,,\\
		\textcolor{black}{W^{p}_N(\textup{div};\Omega)}&\textcolor{black}{\coloneqq \big\{y\in L^p(\Omega;\mathbb{R}^d)}&&\textcolor{black}{\hspace*{-3.25mm}\mid \textup{div}\,y\in L^p(\Omega)\,,\langle \textup{tr}\,y\cdot n,v\rangle_{W^{\smash{1-\frac{1}{p'}},p'}(\partial\Omega)}=0\text{ for all }v\in W^{1,p'}_D(\Omega)\big\}}\,,  
	\end{aligned}
	\end{align*}
	$\smash{W^{1,p}(\Omega;\mathbb{R}^l)\coloneqq W^{1,p}_D(\Omega;\mathbb{R}^l)}$ if $\smash{\Gamma_D=\emptyset}$, and $\smash{W^{p}(\textup{div};\Omega)\coloneqq W^{p}_N(\textup{div};\Omega)}$ if $\smash{\Gamma_N=\emptyset}$,
	where~we~\mbox{denote} by $\textup{tr}\colon\smash{W^{1,p}(\Omega;\mathbb{R}^l)}\hspace*{-0.1em}\to\hspace*{-0.1em}\smash{L^p(\partial\Omega;\mathbb{R}^l)}$ and by $
	\textup{tr}(\cdot)\cdot n\colon\smash{W^p(\textup{div};\Omega)}\hspace*{-0.1em}\to\hspace*{-0.1em} \smash{W^{-\frac{1}{p},p}(\partial\Omega)}$, the trace and normal trace operator, resp. In particular, we  \mbox{predominantly}~\mbox{omit} $\textup{tr}(\cdot)$~in~this~context.~In~addition,~we~employ the abbreviations $L^p(\Omega) \coloneqq  L^p(\Omega;\mathbb{R}^1)$,~${W^{1,p}(\Omega)\coloneqq W^{1,p}(\Omega;\mathbb{R}^1)}$~and~${W^{1,p}_D(\Omega)\coloneqq W^{1,p}_D(\Omega;\mathbb{R}^1)}$.
	
	\subsection{$N$-functions}\vspace{-0.5mm}
	
	\qquad A \hspace{-0.1em}(real) \hspace{-0.1em}convex function
    $\psi\colon\hspace{-0.1em}\mathbb{R}_{\geq 0} \to \mathbb{R}_{\geq 0}$ is called
    \textit{$N$-function},~if~${\psi(0)=0}$,~${\psi(t)>0}$~for~all~${t>0}$,
    $\lim_{t\rightarrow0} \psi(t)/t=0$, and
    $\lim_{t\rightarrow\infty} \psi(t)/t=\infty$. If, in addition,  $\psi\in C^1(\mathbb{R}_{\geq 0})\cap C^2(\mathbb{R}_{> 0})$~and~${\psi''(t)\hspace{-0.1em}>\hspace{-0.1em}0}$ for \hspace{-0.1mm}all \hspace{-0.1mm}$t>0$, \hspace{-0.1mm}we \hspace{-0.1mm}call \hspace{-0.1mm}$\psi$ \hspace{-0.1mm}a \hspace{-0.1mm}\textit{regular
      \hspace{-0.1mm}$N$-function}. \hspace{-0.1mm}For \hspace{-0.1mm}a \hspace{-0.1mm}regular \hspace{-0.1mm}$N$-function \hspace{-0.1mm}${\psi \colon \hspace{-0.1em}\mathbb{R}_{\geq 0}\to \mathbb{R}_{\geq 0}}$,~\hspace{-0.1mm}we~\hspace{-0.1mm}have~\hspace{-0.1mm}that $\psi (0)=\psi'(0)=0$,
    $\psi'\colon\hspace{-0.1em}\mathbb{R}_{\geq 0} \to \mathbb{R}_{\geq 0}$ is increasing and $\lim _{t\to \infty} \psi'(t)=\infty$.~For~a~given~\mbox{$N$-function} ${\psi \colon\mathbb{R}_{\geq 0} \to \mathbb{R}_{\geq 0}}$, we define the \textit{(Fenchel) conjugate \mbox{$N$-function}} $\psi^*\colon\mathbb{R}_{\geq 0} \to \mathbb{R}_{\geq 0}$,~for~every~$t\ge 0$,~by
    ${\psi^*(t)\coloneqq  \sup_{s \geq 0} (st
      -\psi(s))}$, which satisfies $(\psi^*)' =
    (\psi')^{-1}$ in $\mathbb{R}_{\ge 0}$. An  $N$-function $\psi$ satisfies the \textit{$\Delta_2$-condition}
    (in short, $\psi \hspace*{-0.1em}\in\hspace*{-0.1em} \Delta_2$), if there exists $K\hspace*{-0.1em}>\hspace*{-0.1em} 2$ such that~for~all~${t \hspace*{-0.1em}\ge\hspace*{-0.1em} 0}$,~it~holds~${\psi(2\,t) \hspace*{-0.1em}\leq\hspace*{-0.1em} K\, \psi(t)}$. Then, \hspace*{-0.15mm}we \hspace*{-0.15mm}denote \hspace*{-0.15mm}the
    \hspace*{-0.1mm}smallest \hspace*{-0.15mm}such \hspace*{-0.15mm}constant \hspace*{-0.15mm}by \hspace*{-0.15mm}$\Delta_2(\psi)\hspace*{-0.15em}>\hspace*{-0.15em}0$.  \hspace*{-0.15mm}We \hspace*{-0.15mm}say \hspace*{-0.15mm}that \hspace*{-0.15mm}an \hspace*{-0.15mm}$N$-function~\hspace*{-0.15mm}${\psi\colon\hspace*{-0.1em}\mathbb{R}_{\ge 0}\hspace*{-0.15em}\to \hspace*{-0.15em}\mathbb{R}_{\ge 0}}$ satisfies the \textit{$\nabla_2$-condition} (in short, $\psi\in \nabla_2$), if its (Fenchel) conjugate $\psi^*\colon\mathbb{R}_{\ge 0}\to \mathbb{R}_{\ge 0}$ is an $N$-function satisfying the $\Delta_2$-condition. 
    If $\psi\colon\mathbb{R}_{\ge 0}\to \mathbb{R}_{\ge 0}$ satisfies the $\Delta_2$- and the $\nabla_2$-condition (in~short, $\psi\in \Delta_2\cap \nabla_2$), then, there holds 
    the following refined version of the \textit{$\varepsilon$-Young~inequality}: for every
    $\varepsilon> 0$, there exists a constant $c_\varepsilon>0 $, depending only on
    $\Delta_2(\psi),\Delta_2( \psi ^*)<\infty$, such that for every $ s,t\geq0 $, it holds\vspace{-1mm}
    \begin{align}
      \label{ineq:young}
        t\,s&\leq \varepsilon \, \psi(t)+ c_\varepsilon \,\psi^*(s)\,.
    \end{align}
    The mean value of a locally integrable function $f\colon\Omega\to\mathbb{R}$ over a (Lebesgue)~measurable~set~${M\subseteq \Omega}$ is~denoted~by $\fint_M{ f \,\textup{d}x} \coloneqq \smash{\frac 1 {|M|}\int_M f \,\textup{d}x}$. Furthermore, we employ the 
    notations~$(f,g)_M\coloneqq \int_M f g\,\textup{d}x$ and $\rho_{\psi,M}(f)\coloneqq \int_M \psi(\cdot,\vert f\vert )\,\textup{d}x $, 
    for (Lebesgue) measurable functions $f,g\colon\Omega\to \mathbb{R}$, a (Lebesgue) measurable set $M\subseteq \Omega$ and a generalized $N$-function $\psi:M\times\mathbb{R}_{\ge 0}\to \mathbb{R}_{\ge 0}$, i.e.,  $\psi$ is a Carath\'eodory function and $\psi(x,\cdot)$ an $N$-function for a.e.\ $x\in M$,
    whenever the right-hand side is \mbox{well-defined}.
    
    \subsection{Basic properties of the non-linear operator}\vspace{-0.5mm}
    
    \qquad  Throughout the entire paper, we assume that the non-linear operator $\boldsymbol{\mathcal{A}}$ has a $(p,\delta)$-structure, which will be defined now. A detailed discussion and full proofs can be found,~e.g.,~in~\cite{die-ett,dr-nafsa}.~~~~~~~~~~~ 

    For $p \in (1,\infty)$ and $\delta\ge 0$, we define a special $N$-function
$\varphi\vcentcolon =\varphi_{p,\delta}\colon\mathbb{R}_{\ge 0}\to \mathbb{R}_{\ge 0}$ by
\begin{align} 
  \label{eq:def_phi} 
  \varphi(t)\coloneqq  \int _0^t \varphi'(s)\, \mathrm ds,\quad\text{where}\quad
  \varphi'(t) \coloneqq  (\delta +t)^{p-2} t\,,\quad\textup{ for all }t\ge 0\,.
\end{align}
Then, $\varphi\colon \mathbb{R}_{\ge 0}\to  \mathbb{R}_{\ge 0}$ satisfies, independent of $\delta\ge   0$, the
$\Delta_2$-condition with ${\Delta_2(\phi)\leq c\, 2^{\max \{2,p\}}}$. In addition, 
the (Fenchel) conjugate function $\varphi^*\colon\mathbb{R}_{\ge 0}\to \mathbb{R}_{\ge 0}$ satisfies, uniformly~in~${t\hspace{-0.1em}\ge\hspace{-0.1em} 0}$~and~${\delta\hspace{-0.1em}\ge\hspace{-0.1em} 0}$,  $\varphi^*(t) \sim
(\delta^{p-1} + t)^{\smash{p'-2}} t^2$ as well as  the $\Delta_2$-condition with
$\Delta_2(\varphi^*) \leq c\,2^{\smash{\max \{2,p'\}}}$.\newpage

For an $N$-function $\psi\colon\mathbb{R}_{\ge 0}\to \mathbb{R}_{\ge 0}$, we define  {\rm shifted $N$-functions} ${\psi_a\colon\mathbb{R}_{\ge 0}\to \mathbb{R}_{\ge 0}}$,~${a\ge 0}$,~by~~~~
\begin{align}
  \label{eq:phi_shifted}
  \psi_a(t)\coloneqq  \int _0^t \psi_a'(s)\, \mathrm ds\,,\quad\text{where }\quad
  \psi'_a(t)\coloneqq \psi'(a+t)\frac {t}{a+t}\,,\quad\textup{ for all }a,t\ge 0\,.
\end{align}

\begin{remark} 
    For the above defined $N$-function $\varphi\colon \mathbb{R}_{\ge 0}\to \mathbb{R}_{\ge 0}$, cf. \eqref{eq:def_phi}, uniformly in ${a,t\ge 0}$, we have that
    $\varphi_a(t) \sim (\delta+a+t)^{p-2} t^2$ and $(\varphi_a)^*(t)
    \sim ((\delta+a)^{p-1} + t)^{\smash{p'-2}} t^2$. Apart from that, the families
    $\{\varphi_a\}_{\smash{a \ge 0}},\{(\varphi_a)^*\}_{\smash{a \ge 0}}\colon\mathbb{R}_{\ge 0}\to \mathbb{R}_{\ge 0}$ satisfy, uniformly in $a \ge 0$,
    the $\Delta_2$-condition, i.e., for every $a\ge 0$, it holds
    $\Delta_2(\varphi_a) \leq c\, 2^{\smash{\max \{2,p\}}}$ and
    $\Delta_2((\varphi_a)^*) \leq c\, 2^{\smash{\max \{2,p'\}}}$,~respectively.
\end{remark}

\begin{assumption}\label{assum:extra_stress} 
We assume that $\AAA\in C^0(\mathbb{R}^d; \mathbb{R}^d) \cap C^1(\mathbb{R}^d\setminus\{0\}; \mathbb{R}^d) $  satisfies $\AAA(0)=0$ and has a $(p,\delta)$-structure, i.e., there exist $p\in (1,\infty)$, $\delta\ge 0$, and constant $C_0,C_1>0$ such that
\begin{align*}
    ((\nabla \AAA)(a)b )\cdot b
    &\ge C_0(\delta +\vert a\vert)^{p-2}\vert b\vert^2\,,\\
    \vert (\nabla \AAA)(a)\vert&\leq C_1(\delta +\vert a\vert)^{p-2}\,,
\end{align*}
are satisfied for all $a, b\in \mathbb{R}^d$ with $a\neq 0$ and $i,j=1,\dots,d$. The constants $C_0,C_1>0$ and $p\in (1,\infty)$ are called the \textit{characteristics} of $\AAA$.
\end{assumption}

\begin{remark}\label{rem:assum}
    An example of a non-linear operator $\AAA\colon\mathbb{R}^d\to \mathbb{R}^d$ satisfying
    Assumption~\ref{assum:extra_stress} for some $p\in (1,\infty)$ and $\delta\ge 0$, for every $a\in \mathbb{R}^d$, is given via
    \begin{align}
         \AAA (a) = 
   \frac{\varphi'(\vert a\vert )}{\vert a\vert} a=(\delta+\vert a\vert)^{p-2}a\,,\label{special_case}
    \end{align}
    where the characteristics~of~$\AAA\colon\mathbb{R}^d\to \mathbb{R}^d$ 
    depend only on $p\in (1,\infty)$ and are independent~of~${\delta \geq 0}$.
\end{remark}

Closely related to the non-linear operator $\AAA\colon\mathbb{R}^d\to \mathbb{R}^d$ with $(p,\delta)$-structure, where $p\in (1,\infty)$ and $\delta\ge 0 $, are the non-linear operators $F,F^*\colon\mathbb{R}^d\to \mathbb{R}^d$, for every $a\in \mathbb{R}^d$ defined by
\begin{align}
    F(a)\coloneqq (\delta+\vert a\vert)^{\smash{\frac{p-2}{2}}}a\,,\qquad F^*(a)\coloneqq (\delta^{p-1}+\vert a\vert)^{\smash{\frac{p'-2}{2}}}a
    \,.\label{eq:def_F}
\end{align}

The connections between
$\AAA,F,F^*\colon\mathbb{R}^d
\to \mathbb{R}^d$ and
$\varphi_a,(\varphi^*)_a,(\varphi_a)^*\colon\mathbb{R}^{\ge
  0}\to \mathbb{R}^{\ge
  0}$,~${a\ge 0}$, are best explained
by the following proposition.

\begin{proposition}
  Let $\AAA\colon\mathbb{R}^d \to \mathbb{R}^d$ satisfy Assumption~\ref{assum:extra_stress} for $p\in (1,\infty)$ and $\delta \ge 0$. Moreover, let $\varphi\colon\mathbb{R}_{\ge 0}\to \mathbb{R}_{\ge 0}$ be~defined~by~\eqref{eq:def_phi} and let $F,F^*\colon\mathbb{R}^d \to \mathbb{R}^d$ be defined by \eqref{eq:def_F}, each for the same $p\in (1,\infty)$ and $\delta \ge 0$. Then, uniformly with respect to 
  $a, b \in \mathbb{R}^d$,~we~have~that
    \begin{align}
      (\AAA(a) - \AAA(b))
      \cdot(a-b ) &\sim  \smash{\vert F(a) - F(b)\vert^2}
      \sim \varphi_{\vert a\vert }(\vert a - b\vert )\notag
      \\
      &\sim (\varphi_{\vert a\vert })^*(\vert \AAA(a) - \AAA(b)\vert )
      \sim (\varphi^*)_{\vert \AAA(a)\vert }(\vert \AAA(a) - \AAA(b)\vert )\label{eq:hammera}
      \\&\sim \vert F^*(\AAA(a)) - F^*(\AAA(b))\vert^2\,,\notag
       \\[0.5mm]
      \smash{\vert F^*(a) - F^*(b)\vert^2}
      \label{eq:hammerf}
      &\sim  \smash{\smash{(\varphi^*)}_{\smash{\vert a\vert }}(\vert a - b\vert )}\,.
    \end{align}
  The constants in \eqref{eq:hammera} and \eqref{eq:hammerf}
  depend only on the characteristics of ${\AAA}$.
\end{proposition} 

\begin{proof}
    \!For \hspace{-0.2mm}the \hspace{-0.2mm}first \hspace{-0.2mm}two \hspace{-0.2mm}equivalences \hspace{-0.2mm}in \hspace{-0.2mm}\eqref{eq:hammera} \hspace{-0.2mm}and \hspace{-0.2mm}the \hspace{-0.2mm}equivalence \hspace{-0.2mm}\eqref{eq:hammerf}, \hspace{-0.2mm}we \hspace{-0.2mm}refer \hspace{-0.2mm}to \hspace{-0.2mm}\cite[Lemma~\hspace{-0.2mm}6.16]{dr-nafsa}. For the last three equivalences in \eqref{eq:hammera}, we refer to \cite[Lemma 2.8]{dkrt-ldg} and \cite[Lemma 26]{die-ett}.
\end{proof}

In addition, we need the following auxiliary result.

\begin{lemma}[Change of shift]
    Let  $\varphi\colon\mathbb{R}_{\ge 0}\to\mathbb{R}_{\ge 0}$ be defined by \eqref{eq:def_phi} for $p\in (1,\infty)$ and $\delta \ge 0$ and let $F\colon\mathbb{R}^d\to\mathbb{R}^d$ be defined by \eqref{eq:def_F} for the same $p\in (1,\infty)$ and $\delta \ge 0$. Then,
  for~every~${\varepsilon>0}$, there exists $c_\varepsilon \geq 1$, depending only
  on $\varepsilon>0$, $p\in (1,\infty)$, and $\delta\ge 0$, such that for every $a,b\in\mathbb{R}^d$ and $t\geq 0$, it holds\vspace{-1mm}
  \begin{align}
    \varphi_{\vert a\vert}(t)&\leq c_\varepsilon\, \varphi_{\vert b\vert }(t)
    +\varepsilon\, \vert F(a) - F(b)\vert^2\,,\label{lem:shift-change.1}
    \\
     (\varphi_{\vert a\vert})^*(t)&\leq c_\varepsilon\, (\varphi_{\vert b\vert })^*(t)
    +\varepsilon\, \vert F(a) - F(b)\vert^2\,.\label{lem:shift-change.3}
  \end{align}
\end{lemma}

\begin{proof}
    See \cite[Corollary 26, (5.5) \& Corollary 28, (5.8)]{DK08}.\vspace*{-2mm}
\end{proof}\newpage

\begin{remark}[Natural distance]
  \label{rem:natural_dist}
    If \hspace{-0.15mm}$\AAA\colon\mathbb{R}^d\to \mathbb{R}^d$ \hspace{-0.15mm}satisfies \hspace{-0.15mm}Assumption~\hspace{-0.15mm}\ref{assum:extra_stress} \hspace{-0.15mm}for \hspace{-0.15mm}${p\in (1,\infty)}$~\hspace{-0.15mm}and~\hspace{-0.15mm}${\delta\ge 0}$, then, due to \eqref{eq:hammera}, uniformly in  $u, v \in W^{1,p}(\Omega)$, it holds
    \begin{align*}
      \smash{(\AAA(\nabla u) -
      \AAA(\nabla v),\nabla u - \nabla v)_\Omega
      \sim
        \|F(\nabla u)-F(\nabla v)\|_{L^2(\Omega;\mathbb{R}^d)}^2 \,\sim \rho_{\varphi_{\vert \nabla u\vert},\Omega}(\nabla u -
        \nabla v)\,.}
    \end{align*}
    In the context~of the $p$-Dirichlet problem, the quantity $F\colon\hspace*{-0.1em}\mathbb{R}^d\hspace*{-0.1em}\to\hspace*{-0.1em} \mathbb{R}^d$
    was first introduced~in~\cite{acerbi-fusco},~while the last expression
    equals the quasi-norm introduced in~\cite{barliu}
   if raised~to~the~power~of~${\rho = \max \{p,2\}}$.  We refer to all
     three equivalent quantities as the \textit{natural distance}. 
\end{remark}

\begin{remark}[Conjugate natural distance]
  \label{rem:conjugate_natural_dist}
    If \hspace{-0.15mm}$\AAA\colon\mathbb{R}^d\to \mathbb{R}^d$ \hspace{-0.15mm}satisfies \hspace{-0.15mm}Assumption~\hspace{-0.15mm}\ref{assum:extra_stress}~\hspace{-0.15mm}for~\hspace{-0.15mm}${p\in (1,\infty)}$ and $\delta\ge 0$, then, it is readily seen that $\AAA\colon\mathbb{R}^d\hspace*{-0.1em}\to\hspace*{-0.1em} \mathbb{R}^d$ 
    is continuous, strictly~monotone,~and~coercive, so that from the theory of monotone operators, cf. \cite{Zei90B}, it follows that $\AAA\colon\mathbb{R}^d\to \mathbb{R}^d$  is bijective and $\smash{\AAA^{-1}}\colon\mathbb{R}^d\to \mathbb{R}^d$ continuous. In addition,  due to \eqref{eq:hammera}, uniformly in ${z, y \in L^{p'}(\Omega;\mathbb{R}^d)}$,~it~holds
    \begin{align*}
     \smash{ (\smash{\AAA^{-1}}(z)-\smash{\AAA^{-1}}(y),z-y)_\Omega
      \sim
        \|F^*(z)-F^*(y)\|_{L^2(\Omega;\mathbb{R}^d)}^2 \,\sim \rho_{(\varphi^*)_{\vert z\vert},\Omega}( z -
        y)\,.}
    \end{align*}
    We refer to all
     three equivalent quantities as the \textit{conjugate natural distance}. 
\end{remark}
    	
	\subsection{Triangulations and standard finite element spaces}
    \enlargethispage{5mm}
	
	\qquad Throughout the entire paper, we denote by  $\mathcal{T}_h$, $h>0$,  a family~of~regular, i.e., uniformly shape regular and conforming, triangulations of $\Omega\subseteq \mathbb{R}^d$, $d\in\mathbb{N}$, cf. \!\cite{EG21}. 
	\textcolor{black}{Here,~${h>0}$~refers to the average mesh-size, i.e., ${h\coloneqq  (\vert\Omega\vert/\textup{card}(\mathcal{T}_h))^{\frac{1}{d}} }$.}
	For every element $T \in \mathcal{T}_h$,
    we denote by $\rho_T>0$, the supremum of diameters of~inscribed~balls. We assume that there exists a constant $\omega_0>0$, independent of $h>0$, such that $\max_{T\in \mathcal{T}_h}{h_T}{\rho_T^{-1}}\le
    \omega_0$. The smallest such constant is called the chunkiness of $(\mathcal{T}_h)_{h>0}$. Also note that, in what follows, all constants may depend on the chunkiness, but are independent~of~${h>0}$.  For~every~${T \in \mathcal{T}_h}$, 
    let $\omega_T$ denote the patch of~$T$, i.e., the union of all elements of~$\mathcal{T}_h$ touching~$T$.
    We assume that $\textup{int}(\omega_T)$ is connected for all $T\in \mathcal{T}_h$. 
    Under~these~assumptions,  $\vert T\vert \sim
    \vert \omega_T\vert$ uniformly in  $T\in \mathcal{T}_h$ and $h>0$, and the
    number of elements in $\omega_T$ and patches to which an element $T$ belongs
    to are uniformly bounded with respect to $T \in \mathcal{T}_h$~and~$h>0$. 
    We define~the~sides~of~$\mathcal{T}_h$~in~the~\mbox{following}~way: an interior side is the closure of the 
    non-empty relative interior of $\partial T \cap \partial T'$, where $T, T'\in \mathcal{T}_h$~are~adjacent~elements.
    For an interior side $S\coloneqq 
    \partial T \cap \partial T'\in \mathcal{S}_h$, where $T,T'\in \mathcal{T}_h$, we employ the notation $\omega_S\coloneqq  T \cup
    T'$. A boundary side is the closure of the non-empty relative interior of
    $\partial T \cap \partial \Omega$, where $T\in \mathcal{T}_h$ denotes a boundary~element~of~$\mathcal{T}_h$.  For a boundary side $S\coloneqq  \partial T \cap \partial
    \Omega$, we employ the notation $\omega_S\coloneqq  T $. By $\mathcal{S}_h^{i}$ and $\mathcal{S}_h$, we denote the sets of 
    all interior sides and the set~of~all~sides,~respectively.~\textcolor{black}{Eventually,   we define ${h_S\coloneqq \textup{diam}(S)}$ for all $S\in \mathcal{S}_h$ and $h_T\coloneqq \textup{diam}(T)$ for all $T\in \mathcal{T}_h$.}
	
	For $k\in \mathbb{N}\cup\{0\}$ and $T\in \mathcal{T}_h$, let $\mathcal{P}_k(T)$ denote the set of polynomials of maximal~degree~$k$~on~$T$. Then, for $k\in \mathbb{N}\cup\{0\}$~and $l\in \mathbb{N}$,  the sets of continuous and~\mbox{element-wise}~polynomial functions or vector~fields,~respectively, are defined by
	\begin{align*}
	\begin{aligned}
	\mathcal{S}^k(\mathcal{T}_h)^l&\coloneqq 	\big\{v_h\in C^0(\overline{\Omega};\mathbb{R}^l)\hspace*{-3mm}&&\mid v_h|_T\in\mathcal{P}_k(T)^l\text{ for all }T\in \mathcal{T}_h\big\}\,,\\
	\mathcal{L}^k(\mathcal{T}_h)^l&\coloneqq    \big\{v_h\in L^\infty(\Omega;\mathbb{R}^l)\hspace*{-3mm}&&\mid v_h|_T\in\mathcal{P}_k(T)^l\text{ for all }T\in \mathcal{T}_h\big\}\,.
	\end{aligned}
	\end{align*}
	The element-wise constant mesh-size function $h_\mathcal{T}\in \mathcal{L}^0(\mathcal{T}_h)$ is defined~by~${h_\mathcal{T}|_T\coloneqq h_T}$~for~all~${T\in \mathcal{T}_h}$.
	The side-wise constant mesh-size function $h_\mathcal{S}\in \mathcal{L}^0(\mathcal{S}_h)$ is defined~by~${h_\mathcal{S}|_S\coloneqq h_S}$~for~all~${S\in \mathcal{S}_h}$.
	\textcolor{black}{For every $T\in \mathcal{T}_h$ and $S\in \mathcal{S}_h$,~we~denote~by $\smash{x_T\!\coloneqq \!\frac{1}{d+1}\sum_{z\in \mathcal{N}_h\cap T}{z}}$ \hspace{-0.2mm}and \hspace{-0.2mm}$\smash{x_S\!\coloneqq \!\frac{1}{d}\sum_{z\in \mathcal{N}_h\cap S}{z}}$,  \hspace{-0.2mm}the \hspace{-0.2mm}midpoints \hspace{-0.2mm}(barycenters) \hspace{-0.2mm}of~\hspace{-0.2mm}$T$~\hspace{-0.2mm}and~\hspace{-0.2mm}$S$,~\hspace{-0.2mm}\mbox{respectively}.} The \hspace{-0.2mm}(local) \hspace{-0.4mm}$L^2$\hspace{-0.2mm}-projection \hspace{-0.2mm}operator \hspace{-0.2mm}onto \hspace{-0.2mm}\mbox{element-wise} \hspace{-0.2mm}constant \hspace{-0.2mm}functions \hspace{-0.2mm}or \hspace{-0.2mm}vector \hspace{-0.2mm}fields,~\hspace{-0.2mm}\mbox{respectively},  is denoted by\vspace{-0.5mm}
	\begin{align*}
	\smash{\Pi_h\colon L^1(\Omega;\mathbb{R}^l)\to \mathcal{L}^0(\mathcal{T}_h)^l\,.}
	\end{align*}
	For every $v_h\in \smash{\mathcal{L}^1(\mathcal{T}_h)^l}$, it holds $\Pi_hv_h|_T=v_h(x_T)$ in $T$ for all $T\in \mathcal{T}_h$. 
    The element-wise~gradient operator 
    $\nabla_{\!h}\colon \mathcal{L}^1(\mathcal{T}_h)^l\to \mathcal{L}^0(\mathcal{T}_h)^{l\times d}$, for every $v_h\in \mathcal{L}^1(\mathcal{T}_h)^l$, is defined by $\nabla_{\!h}v_h|_T\coloneqq \nabla(v_h|_T)$~in~$T$ for all $T\in \mathcal{T}_h$.
 	
	\subsubsection{Crouzeix--Raviart element}
	
	\qquad The Crouzeix--Raviart finite element space, introduced~in~\cite{CR73}, consists of \mbox{element-wise} affine functions that are continuous at the midpoints of inner element sides, i.e.,\footnote{Here, for every $S\in\mathcal{S}_h^{i}$, $\jump{v_h}_S\coloneqq v_h|_{T_+}-v_h|_{T_-}$ on $S$, where $T_+, T_-\in \mathcal{T}_h$ satisfy $\partial T_+\cap\partial  T_-=S$, and for every $S\in\mathcal{S}_h\cap\partial \Omega$, $\jump{v_h}_S\coloneqq v_h|_T$ on $S$, where $T\in \mathcal{T}_h$ satisfies $S\subseteq \partial T$.}
	\begin{align*}\mathcal{S}^{1,\textit{cr}}(\mathcal{T}_h)\coloneqq \big\{v_h\in \mathcal{L}^1(\mathcal{T}_h)\mid \jump{v_h}_S(x_S)=0\text{ for all }S\in \mathcal{S}_h^{i}\big\}\,.
	\end{align*}
	Crouzeix--Raviart finite element functions that vanish at the midpoints of boundary~element~sides that correspond to the  Dirichlet boundary $\Gamma_D$ are contained~in~the~space
	\begin{align*}
			\smash{\mathcal{S}^{1,\textit{cr}}_D(\mathcal{T}_h)}\coloneqq \big\{v_h\in\smash{\mathcal{S}^{1,\textit{cr}}(\mathcal{T}_h)}\mid v_h(x_S)=0\text{ for all }S\in \mathcal{S}_h\cap \Gamma_D\big\}\,.
	\end{align*}
	In particular, we have that $	\smash{\mathcal{S}^{1,\textit{cr}}_D(\mathcal{T}_h)}=	\smash{\mathcal{S}^{1,\textit{cr}}(\mathcal{T}_h)}$ if $\Gamma_D=\emptyset$.
	A basis of  $\smash{\mathcal{S}^{1,\textit{cr}}(\mathcal{T}_h)}$~is~given~by~functions $\varphi_S\in \smash{\mathcal{S}^{1,\textit{cr}}(\mathcal{T}_h)}$, $S\in \mathcal{S}_h$, satisfying the Kronecker~property $\varphi_S(x_{S'})=\delta_{S,S'}$ for all $S,S'\in \mathcal{S}_h$. A basis of  $\smash{\smash{\mathcal{S}^{1,\textit{cr}}_D(\mathcal{T}_h)}}$~is~given~by~$\varphi_S\in \smash{\mathcal{S}^{1,\textit{cr}}_D(\mathcal{T}_h)}$, $S\in \mathcal{S}_h\setminus\Gamma_D$. 
	
	\subsubsection{Raviart--Thomas element}
 
	\qquad The lowest order Raviart--Thomas finite element space, introduced~in~\cite{RT75},~consists~of~element-wise \hspace{-0.2mm}affine \hspace{-0.2mm}vector \hspace{-0.2mm}fields \hspace{-0.2mm}that \hspace{-0.2mm}have \hspace{-0.2mm}continuous \hspace{-0.2mm}constant \hspace{-0.2mm}normal \hspace{-0.2mm}components \hspace{-0.2mm}on \hspace{-0.2mm}inner~\hspace{-0.2mm}elements~\hspace{-0.2mm}sides,~\hspace{-0.2mm}i.e.,\!\footnote{Here, for every $S\in\mathcal{S}_h^{i}$, $\jump{y_h\cdot n}_S\coloneqq \smash{y_h|_{T_+}\cdot n_{T_+}+y_h|_{T_-}\cdot n_{T_-}}$ on $S$, where $T_+, T_-\in \mathcal{T}_h$ satisfy $\smash{\partial T_+\cap\partial  T_-=S}$,  and for every $T\in \mathcal{T}_h$, $\smash{n_T\colon\partial T\to \mathbb{S}^{d-1}}$ denotes the outward unit normal vector field~to~$ T$, 
	and  for every $\smash{S\in\mathcal{S}_h\cap\partial \Omega}$, $\smash{\jump{y_h\cdot n}_S\coloneqq \smash{y_h|_T\cdot n}}$ on $S$, where $T\in \mathcal{T}_h$ satisfies $S\subseteq \partial T$ and $\smash{n\colon\partial\Omega\to \mathbb{S}^{d-1}}$ denotes the outward unit normal vector field to $\Omega$.}
	\begin{align*}
        \mathcal{R}T^0(\mathcal{T}_h)\coloneqq \big\{y_h\in \mathcal{L}^1(\mathcal{T}_h)^d\mid &\,\smash{y_h|_T\cdot         n_T=\textup{const}\text{ on }\partial T\text{ for  all }T\in \mathcal{T}_h\,,}\\ 
        &\smash{	\jump{y_h\cdot n}_S=0\text{ on }S\text{ for all }S\in \mathcal{S}_h^{i}\big\}\,.}
	\end{align*}
	Raviart--Thomas finite element functions that have vanishing normal components~on~the Neumann boundary $\Gamma_N$ are contained in the space 
	\begin{align*}
		\smash{\mathcal{R}T^{0}_N(\mathcal{T}_h)}\coloneqq \big\{y_h\in	\mathcal{R}T^0(\mathcal{T}_h)\mid y_h\cdot n=0\text{ on }\Gamma_N\big\}\,.
	\end{align*}
	In particular, we have that $\smash{\mathcal{R}T^{0}_N(\mathcal{T}_h)}=\mathcal{R}T^0(\mathcal{T}_h)$ if $\Gamma_N=\emptyset$. 
	A basis of  $\mathcal{R}T^0(\mathcal{T}_h)$ is given~by~vector fields \hspace{-0.2mm}$\psi_S\hspace{-0.15em}\in\hspace{-0.15em} \mathcal{R}T^0(\mathcal{T}_h)$, \hspace{-0.2mm}$S\hspace{-0.15em}\in\hspace{-0.15em} \mathcal{S}_h$, \hspace{-0.2mm}satisfying \hspace{-0.2mm}the \hspace{-0.2mm}Kronecker \hspace{-0.2mm}property \hspace{-0.2mm}$\psi_S|_{S'}\cdot n_{S'}\hspace{-0.15em}=\hspace{-0.15em}\delta_{S,S'}$~\hspace{-0.2mm}on~\hspace{-0.2mm}$S'$~\hspace{-0.2mm}for~\hspace{-0.2mm}all~\hspace{-0.2mm}${S'\hspace{-0.2em}\in\hspace{-0.15em} \mathcal{S}_h}$, where $n_S$ for all $S\in \mathcal{S}_h$ is the unit normal vector on $S$ pointing from~$T_-$~to~$T_+$~if~${T_+\cap T_-=S\in \mathcal{S}_h}$. A basis of $\smash{\mathcal{R}T^{0}_N(\mathcal{T}_h)}$ is given by $\psi_S\in \smash{\mathcal{R}T^{0}_N(\mathcal{T}_h)}$, ${S\in \mathcal{S}_h\setminus\Gamma_N}$. 
	
	\subsubsection{Discrete integration-by-parts formula}
 
    	\qquad An element-wise integration-by-parts implies that for every $v_h\in \mathcal{S}^{1,\textit{cr}}(\mathcal{T}_h)$ and ${y_h\in \mathcal{R}T^0(\mathcal{T}_h)}$, we have the \textit{discrete integration-by-parts
	formula}
	\begin{align}
	(\nabla_hv_h,\Pi_h y_h)_{\Omega}+(\Pi_h v_h,\,\textup{div}\,y_h)_{\Omega}=(v_h,y_h\cdot n)_{\partial\Omega}\,.\label{eq:pi}
	\end{align}
	Here, \hspace{-0.15mm}we  \hspace{-0.15mm}used \hspace{-0.15mm}that \hspace{-0.15mm}$y_h\hspace{-0.18em}\in \hspace{-0.18em}\mathcal{R}T^0(\mathcal{T}_h)$ \hspace{-0.15mm}has \hspace{-0.15mm}continuous \hspace{-0.15mm}constant  \hspace{-0.15mm}normal \hspace{-0.15mm}components~\hspace{-0.15mm}on~\hspace{-0.15mm}inner~\hspace{-0.15mm}\mbox{element}~\hspace{-0.15mm}sides, i.e., ${y_h|_T\cdot n_T}=\textrm{const}$ on $\partial T$ for every $T\in \mathcal{T}_h$ and	$\jump{y_h\cdot n}_S=0$ on $S$ for every $S\in \mathcal{S}_h^{i}$,~as~well~as that the jumps~of ${v_h\in \mathcal{S}^{1,\textit{cr}}(\mathcal{T}_h)}$  across inner element sides have vanishing integral~mean,~i.e., $\int_{S}{\jump{v_h}_S\,\textup{d}s}\hspace*{-0.1em}=\hspace*{-0.1em}\jump{v_h}_S(x_S)
 \hspace*{-0.1em}=\hspace*{-0.1em}0$ for all $S\hspace*{-0.1em}\in \hspace*{-0.1em}\mathcal{S}_h^{i}$.
	In particular, for any $v_h\hspace*{-0.1em}\in \hspace*{-0.1em}\smash{\mathcal{S}^{1,\textit{cr}}_D(\mathcal{T}_h)}$~and~${y_h\hspace*{-0.1em}\in\hspace*{-0.1em} \smash{\mathcal{R}T^0_N(\mathcal{T}_h)}}$, \eqref{eq:pi} reads 
	\begin{align}
		(\nabla_hv_h,\Pi_h y_h)_{\Omega}=-(\Pi_h v_h,\,\textup{div}\,y_h)_{\Omega}\,.\label{eq:pi0}
	\end{align}
	In \cite{Bar21}, the discrete integration-by-parts formula \eqref{eq:pi0} formed a cornerstone in the derivation~of~discrete convex duality theory and, as such, also plays a central 
	role 
	in the hereinafter~analysis.

	\subsection{$p$-Dirichlet problem}
	
	\qquad In this section, we briefly review the variational, the primal, and the dual formulation of the $p$-Dirichlet problem \eqref{eq:pDirichlet}. In addition, we examine a natural regularity assumption  on the solution $u\in \smash{W^{1,p}_D(\Omega)}$ of \eqref{eq:pDirichlet} and its consequences, in particular,  for the flux $z\coloneqq \AAA(\nabla u)\in \smash{W^{p'}_N(\textup{div};\Omega)}$.
	
	\subsubsection{Variational problem}
		
		\qquad Given a right-hand side $f\in L^{p'}(\Omega)$, $p\in (1,\infty)$, and given a non-linear operator ${\AAA\colon\mathbb{R}^d\to \mathbb{R}^d}$ that satisfies Assumption \ref{assum:extra_stress} for $p\!\in\! (1,\infty)$ and  $\delta\!\ge\! 0$, the $p$-Dirichlet~problem~seeks~for~${u\!\in\! W^{1,p}_D(\Omega)}$ such that for every $v\in \smash{W^{1,p}_D(\Omega)}$, it holds\footnote{Note that, by Assumption \ref{assum:extra_stress} and \cite[(2.13)]{dkrt-ldg}, 
 it holds $\vert \AAA(a)\vert \sim\varphi'(\vert a\vert) \sim \delta^{p-1}+\vert a\vert^{p-1}$ for all $a\in \mathbb{R}^d$. Thus, by the theory of Nemytskii operators, for every $v\in W^{1,p}_D(\Omega)$, it holds $\AAA(\nabla v)\in L^{p'}(\Omega;\mathbb{R}^d)$.}
		\begin{align}
		    (\AAA(\nabla u),\nabla v)_{\Omega}=(f,v)_{\Omega}\,.\label{eq:pDirichletW1p}
		\end{align}
		Resorting to the celebrated  theory of monotone operators, cf. \cite{Zei90B},~it~is~readily~apparent~that~\eqref{eq:pDirichletW1p} admits a unique solution. In what follows, we~reserve~the~notation~${u\in \smash{W^{1,p}_D(\Omega)}}$~for~this~solution.~~~~~
		
	\subsubsection{Minimization problem and convex duality relations}\label{subsec:convex_duality}
	
	\qquad In the case  \eqref{special_case}, i.e., $\AAA\colon\mathbb{R}^d\to \mathbb{R}^d$, has a potential, cf. Remark \ref{rem:assum}, 
	the variational problem \eqref{eq:pDirichletW1p} arises as an optimality condition of an equivalent convex minimization problem, leading to a primal and a dual formulation of \eqref{eq:pDirichletW1p}, as well as to convex duality relations.
	
	\qquad \textit{Primal problem.} In the case  \eqref{special_case}, a problem equivalent to \eqref{eq:pDirichletW1p} is given~by~the~minimiza-tion of the \textit{$p$-Dirichlet energy}, i.e., the energy functional $I\colon \smash{W^{1,p}_D(\Omega)}\to \mathbb{R}$, for every $v\in \smash{W^{1,p}_D(\Omega)}$ defined by\enlargethispage{1mm}
	\begin{align}\label{eq:pDirichletPrimal}
	    I(v)\coloneqq \rho_{\varphi,\Omega}( \nabla v)-(f,v)_{\Omega}\,.
	\end{align}
	In what follows, we refer the minimization of the $p$-Dirichlet energy \eqref{eq:pDirichletPrimal}~to~as~the~\textit{primal~problem}.
	Since the $p$-Dirichlet energy is proper, 
	strictly convex, weakly coercive,  
	and~lower~semi-continuous, the direct method in the calculus of variations, cf. \cite{Dac08}, 
 implies the existence~of~a~unique~minimizer, called the \textit{primal solution}. In particular,
	since the $p$-Dirichlet energy is Fr\'echet differentiable~and for every $v,w\in W^{1,p}_D(\Omega)$, it holds
	\begin{align*}
	    \langle DI(v),w\rangle_{\smash{W^{1,p}_D(\Omega)}}=(\AAA(\nabla v),\nabla w)_{\Omega}\,,
	\end{align*}
	the optimality conditions of the primal~problem~and the convexity of the $p$-Dirichlet energy imply that $\smash{u\in W^{1,p}_D(\Omega)}$ solves the primal problem, i.e., is the unique minimizer of the~$p$-Dirichlet~energy.
	
	\qquad \textit{Dual problem.} In the case  \eqref{special_case}, proceeding as,~e.g.,~in~\mbox{\cite[p.\ 113~ff.]{ET99}}, one  finds that
		the \textit{dual problem} consists in the maximization of the energy functional ${D\colon\smash{W^{p'}_N(\textup{div};\Omega)}\to \mathbb{R}\cup\{-\infty\}}$, for every $y\in \smash{W^{p'}_N(\textup{div};\Omega)}$ defined by
		\begin{align}
			D(y)\coloneqq -\rho_{\varphi^*,\Omega}(y) -I_{\{-f\}}(\textup{div}\,y)\,,\label{eq:pDirichletDual}
		\end{align}
		where $I_{\{-f\}}\colon\smash{L^{p'}(\Omega)}\to \mathbb{R}\cup\{+\infty\}$ is defined by $I_{\{-f\}}(g)\coloneqq 0$~if~$g=-f$ and  ${I_{\{-f\}}(g)\coloneqq +\infty}$~else. 
		Appealing \hspace*{-0.1mm}to \hspace*{-0.1mm}\cite[Proposition \hspace*{-0.1mm}5.1, \hspace*{-0.1mm}p. \hspace*{-0.35em}115]{ET99}, \hspace*{-0.1mm}the \hspace*{-0.1mm}dual \hspace*{-0.1mm}problem \hspace*{-0.1mm}admits \hspace*{-0.1mm}a \hspace*{-0.1mm}unique \hspace*{-0.1mm}solution \hspace*{-0.1mm}$\smash{z\!\in\! W^{p'}_N(\textup{div};\Omega)}$, i.e., a maximizer  of \eqref{eq:pDirichletDual}, called the \textit{dual solution}, and a \textit{strong duality relation}, i.e., $I(u)=D(z)$, applies. In addition, there~hold~the \textit{convex optimality relations}
		\begin{alignat}{2}
                z\cdot\nabla u&=\varphi^*(\vert z\vert )+\varphi(\vert \nabla u\vert )&&\quad\textup{ in }L^1(\Omega)\,,\label{eq:pDirichletOptimality1.2}\\
				\textup{div}\,z&=-f&&\quad\text{ in }L^{p'}(\Omega)
                \,.
		\end{alignat}
		Note that,  by the Fenchel--Young identity, cf. \cite[Proposition 5.1, p.\ 21]{ET99},  \eqref{eq:pDirichletOptimality1.2} is equivalent to
		\begin{align}
			z=\AAA(\nabla u)\quad\text{ in }\smash{W^{p'}_N(\textup{div};\Omega)}\,.\label{eq:pDirichletOptimality2}
		\end{align}
		
	\subsubsection{Natural regularity assumption on the solution to the $p$-Dirichlet problem} 
	
	\qquad In this section, we briefly collect  important consequences of the natural regularity assumption 
    \begin{align}\label{natural_regularity}
        F(\nabla u)\in W^{1,2}(\Omega;\mathbb{R}^d)\,,
    \end{align}   
    on the solution $u\in W^{1,p}_D(\Omega)$ of \eqref{eq:pDirichletW1p}, which is satisfied under mild assumptions on the bounded domain $\Omega\subseteq\mathbb{R}^d$, $d\in \mathbb{N}$, and the right-hand side $f\in \smash{L^{p'}(\Omega)}$, cf. \cite[Remark 5.11]{DR07}. For a detailed discussion addressing this regularity assumption, please refer~to~the~contributions~\cite{acerbi-fusco,BL1993a,giu1,ELS04,Ebmeyer05}. The following lemma relates 
    \eqref{natural_regularity} 
    with the weighted integrability of the Hessian of $u\in W^{1,p}_D(\Omega)$.
	
	\begin{lemma}\label{lem:reg_primal_source}
            Let $F\colon \mathbb{R}^d\to \mathbb{R}^d$ be defined by \eqref{eq:def_F} for $p\in (1,\infty)$ and $\delta\ge 0$. Then, there exists a constant ${c>0}$, depending only on $p\in (1,\infty)$, such that for every  $v\in W^{1,p}(\Omega)$~with~$F(\nabla v)\in W^{1,2}(\Omega;\mathbb{R}^d)$, it holds
                \begin{align*}
                  \smash{ c^{-1}\,(\delta+\vert \nabla v\vert )^{p-2}\vert \nabla^2 v\vert^2\leq \vert \nabla F(\nabla v)\vert^2\leq c\,(\delta+\vert \nabla v\vert )^{p-2}\vert \nabla^2 v\vert^2\quad\text{ a.e.\ in }\Omega\,.}
                \end{align*}
		\end{lemma}
		
		\begin{proof}
		 See \cite[Proposition 2.14, (2.15)]{br-plasticity} (where $\delta\ge 0$) or \cite[Lemma 3.8]{BDR10} (where $\delta>0$).
		\end{proof}

            Distinguishing between the cases $p\in [2,\infty)$ and $p\in (1,2)$, using Lemma \ref{lem:reg_primal_source}, the unweighted integrability of the Hessian of $u\in W^{1,p}_D(\Omega)$ can be derived from 
            \eqref{natural_regularity}.
		
		\begin{lemma}\label{lem:reg_primal}
            Let $F\colon \mathbb{R}^d\to \mathbb{R}^d$ be defined by \eqref{eq:def_F} for $p\in (1,\infty)$ and $\delta\ge 0$. Then, there exists a constant ${c>0}$, depending only on $p\in (1,\infty)$, such that for every  $v\in W^{1,p}(\Omega)$~with~$F(\nabla v)\in W^{1,2}(\Omega;\mathbb{R}^d)$, the following statements apply:
		    \begin{itemize}[noitemsep,topsep=2.0pt,labelwidth=\widthof{\textit{(ii)}},leftmargin=!]
                \item[(i)] If $p\ge 2$ and $\delta>0$, then, it holds $ v\in W^{2,2}(\Omega)$ with $\vert\nabla^2 v\vert^2\leq c\, \delta^{2-p}\vert\nabla F(\nabla v)\vert^2$ a.e.\ in $\Omega$.
                
                \item[(ii)] If $p\leq 2$, then, it holds $ v\in W^{2,p}(\Omega)$ with $\vert\nabla^2 v\vert^p\leq c\,\vert\nabla F(\nabla v)\vert^2+(\delta+\vert \nabla v\vert)^p$ a.e.\ in $\Omega$.
            \end{itemize}
		\end{lemma}
		
		\begin{proof}
		
		         \textit{ad (i).} Immediate consequence of Lemma \ref{lem:reg_primal_source}.
		         
		         \textit{ad (ii).} We proceed as in the proof of \cite[Lemma 4.4]{BDR10}, where the case $\delta>0$ is considered, but instead of \cite[Lemma 3.8]{BDR10}, we refer to Lemma \ref{lem:reg_primal_source}, to cover also the case $\delta=0$.
		\end{proof}

  The following lemma is of crucial importance for the derivation of optimal a priori estimates, since it translates the natural regularity assumption \eqref{natural_regularity} to the flux $z\coloneqq \AAA(\nabla u)\in \smash{W^{p'}_N(\textup{div};\Omega)}$. This enables us later to estimate oscillation terms optimally.
		
		\begin{lemma}\label{lem:reg_equiv} Let $\AAA\colon\hspace{-0.1em}\mathbb{R}^d\hspace{-0.15em}\to\hspace{-0.15em} \mathbb{R}^d$ satisfy Assumption \ref{assum:extra_stress} for $p\hspace{-0.1em}\in\hspace{-0.1em} (1,\infty)$ and $\delta\hspace{-0.1em}\ge \hspace{-0.1em} 0$ and~let~${F\colon\hspace{-0.1em}\mathbb{R}^d\hspace{-0.15em}\to\hspace{-0.15em} \mathbb{R}^d}$ be defined by \eqref{eq:def_F} for the same $p\in (1,\infty)$ and $\delta\ge 0$. Then, for every function $v\in W^{1,p}(\Omega)$ and $y\vcentcolon =\AAA(\nabla v)\in L^{p'}(\Omega;\mathbb{R}^d)$, it holds  $F(\nabla v)\in W^{1,2}(\Omega;\mathbb{R}^d)$ if and only if $F^*(y)\in W^{1,2}(\Omega;\mathbb{R}^d)$, $\vert F(\nabla v)\vert \sim \vert  F^*(y)\vert$ a.e.\ in $\Omega$ and $\vert \nabla F(\nabla v)\vert \sim \vert \nabla F^*(y)\vert$ a.e.\ in $\Omega$, 
        where the constants in the equivalences only depend on the characteristics of $\AAA$.
		\end{lemma}
		
		\begin{proof}
		        See \cite[Lemma 2.10]{dkrt-ldg}.
		\end{proof}

        Lemma \hspace*{-0.1mm}\ref{lem:reg_equiv}, \hspace*{-0.1mm}in \hspace*{-0.1mm}turn, \hspace*{-0.1mm}motivates \hspace*{-0.1mm}to \hspace*{-0.1mm}establish \hspace*{-0.1mm}the \hspace*{-0.1mm}following \hspace*{-0.1mm}analogs \hspace*{-0.1mm}of \hspace*{-0.1mm}Lemma \hspace*{-0.1mm}\ref{lem:reg_primal_source} \hspace*{-0.1mm}and \hspace*{-0.1mm}Lemma~\hspace*{-0.1mm}\ref{lem:reg_primal}.
		
		\begin{lemma}\label{lem:reg_dual_source}
        Let $F^*\colon \mathbb{R}^d\to \mathbb{R}^d$ be defined by \eqref{eq:def_F} for $p\in (1,\infty)$ and $\delta\ge 0$. Then, there exists a constant $c>0$, depending only on $p\in (1,\infty)$, such that for every  $y\in L^{p'}(\Omega;\mathbb{R}^d)$~with~$F^*(y)\in W^{1,2}(\Omega;\mathbb{R}^d)$, it holds
                \begin{align*}
                  \smash{ c^{-1}\,(\delta^{p-1}+\vert y\vert )^{p'-2}\vert \nabla y\vert^2\leq \vert \nabla F^*(y)\vert^2\leq c\,(\delta^{p-1}+\vert y\vert )^{p'-2}\vert \nabla y\vert^2\quad\text{ a.e.\ in }\Omega\,.}
                \end{align*}
		\end{lemma}
		
		\begin{proof}\let\qed\relax
		  We only give a proof for the case $\delta>0$ and proceed similar to \cite[Lemma 3.8]{BDR10}.~For~${\delta=0}$, one \hspace{-0.1mm}proceeds \hspace{-0.1mm}as \hspace{-0.1mm}in \hspace{-0.1mm}\cite[\hspace{-0.1mm}Propositon \hspace{-0.1mm}2.14]{br-plasticity}. 
   \hspace{-0.1mm}First, \hspace{-0.1mm}we \hspace{-0.1mm}observe \hspace{-0.1mm}that \hspace{-0.1mm}$\vert \nabla y\vert \hspace{-0.1em}=\hspace{-0.1em}\vert \nabla F^*(y)\vert\hspace{-0.1em}=\hspace{-0.1em}0$~\hspace{-0.1mm}a.e.~\hspace{-0.1mm}in~\hspace{-0.1mm}${\{\vert y\vert \hspace{-0.1em}=\hspace{-0.1em}0\}}$, i.e., the claimed equivalence applies in $\{\vert y\vert =0\}$.
   As a consequence, for the remainder~of~the~proof, it suffices to consider the case  $\vert y\vert>0$. For this case, we compute that
		          \begin{align}
		              \smash{\nabla F^*(y)=\tfrac{p'-2}{2}(\delta^{p-1}+\vert y\vert)^{\smash{\frac{p'-4}{2}}} y\otimes\nabla \vert y\vert+(\delta^{p-1}+\vert y\vert)^{\smash{\frac{p'-2}{2}}}\nabla y=\vcentcolon a+b\quad\textup{ a.e. in }\Omega\,.}\label{lem:reg_dual.1}
		          \end{align}
		          Then, from \eqref{lem:reg_dual.1}, in turn, we deduce that $\vert \nabla F^*(y)\vert^2=\vert a\vert^2+2\,a\cdot b +\vert b\vert^2 $ with
		          \begin{align}
		          \left.\begin{aligned}
		             \vert a\vert^2&=\big(\tfrac{p'-2}{2}\big)^2\,(\delta^{p-1}+\vert y\vert)^{\smash{p'-4}}\vert y\vert^2\vert \nabla \vert y\vert\vert^2\,,\\[-0.5mm]
		             2\,a\cdot b &=(p'-2)\,(\delta^{p-1}+\vert y\vert)^{\smash{p'-3}}\vert y\vert\vert \nabla\vert y\vert\vert^2\,,\\[-0.5mm]
		             \vert b\vert^2&=(\delta^{p-1}+\vert y\vert)^{p'-2}\vert \nabla y\vert^2\,,\\[-0.5mm] \end{aligned}\quad\right\}\quad\textup{ a.e. in }\Omega\,.\label{lem:reg_dual.2}
		          \end{align}
		          
		          Next, we need to distinguish  between the cases  $\smash{p'\ge 2}$ and  $\smash{p'< 2}$:
		          
		          \textit{Case $\smash{p'\ge 2}$.} In this case, we have that $ 2\,a\cdot b\ge 0$, cf. \eqref{lem:reg_dual.2}. Therefore, using $\vert a\vert^2\leq (\frac{p'-2}{2})^2\vert b\vert^2$, we~deduce~that
		          \begin{align*}
		              \smash{\vert b\vert^2\leq \vert \nabla F^*(y)\vert^2=\vert a\vert^2+2\,\vert a\vert \vert b\vert +\vert b\vert^2  \leq \big(\big(\tfrac{p'-2}{2}\big)^2+(p'-2)+1\big)\vert b\vert^2=\tfrac{(p')^2}{4}\vert b\vert^2\,.}
		          \end{align*}
		          
		          \textit{Case $p'\leq 2$.} In this case, using that $\smash{\vert a\vert\leq \tfrac{2-p'}{2}\vert b\vert \leq\tfrac{p'+2}{2}\vert b\vert \leq 2\vert b\vert}$, we find that
		          \begin{align*}
		             \smash{\tfrac{(p')^2}{4}\vert b\vert^2\hspace{-0.1em}\leq\hspace{-0.1em} \big(\vert a\vert\hspace{-0.1em}-\hspace{-0.1em}\tfrac{p'+2}{2}\vert b\vert \big)\big(\vert a\vert\hspace{-0.1em}-\hspace{-0.1em}\tfrac{2-p'}{2}\vert b\vert \big)\hspace{-0.1em}+\hspace{-0.1em}\tfrac{(p')^2}{4}\vert b\vert^2\hspace{-0.1em}=\hspace{-0.1em}\vert \nabla F^*(y)\vert^2
		            \hspace{-0.1em}=\hspace{-0.1em}\vert a\vert\big(\vert a\vert\hspace{-0.1em}-\hspace{-0.1em}2\vert b\vert \big)\hspace{-0.1em}+\hspace{-0.1em}\vert b\vert^2\hspace{-0.1em}\le \vert b\vert^2\,.}\tag*{$\qedsymbol$}\hspace{-0.1em}
		          \end{align*}
		\end{proof}
		
		\begin{lemma}\label{lem:reg_dual}
             Let $F^*\colon \mathbb{R}^d\to \mathbb{R}^d$ be defined by \eqref{eq:def_F} for  $p\in (1,\infty)$ and $\delta\ge 0$. Then, there exists a constant $c>0$, depending only on $p\in (1,\infty)$, such that for every  $y\in L^{p'}(\Omega;\mathbb{R}^d)$~with~$F^*(y)\in W^{1,2}(\Omega;\mathbb{R}^d)$, the following statements apply:
		    \begin{itemize}[noitemsep,topsep=2.0pt,labelwidth=\widthof{\textit{(ii)}},leftmargin=!]
                \item[(i)] If $p\leq 2$ and $\delta>0$, then, it holds $ y\in W^{1,2}(\Omega;\mathbb{R}^d)$ with $\vert\nabla y\vert^2\leq c\, \delta^{2-p'}\vert\nabla F^*(y)\vert^2$ a.e.\ $\Omega$.
                
                \item[(ii)] If $p\geq 2$, then, it holds $ y\in W^{1,\smash{p'}}(\Omega;\mathbb{R}^d)$ with $\vert\nabla y\vert^{p'}\leq c\,\vert\nabla F^*(y)\vert^2+(\delta^{p-1}+\vert y\vert)^{p'}$  a.e.\ $\Omega$.
            \end{itemize}
		\end{lemma}
		
		\begin{proof}\let\qed\relax
		
		          \textit{ad (i).} Immediate consequence of Lemma \ref{lem:reg_dual_source}.
		        
		          \textit{ad (ii).} Using that for $p\geq 2$, i.e., $p'\leq 2$, it holds $a^{p'}\leq a^2b^{p'-2}+b^{p'}$ for all $a\ge 0$ and $b>0$, and Lemma \ref{lem:reg_dual_source}, we find that
		          \begin{align}\tag*{$\qedsymbol$}
		             \smash{ \vert \nabla y\vert^{p'}\leq (\delta^{p-1}+\vert y\vert)^{\smash{p'}-2}\vert \nabla y\vert^2+(\delta^{p-1}+\vert y\vert)^{\smash{p'}}
		              \leq c\,\vert \nabla F^*(y)\vert^2+(\delta^{p-1}+\vert y\vert)^{\smash{p'}}\quad\text{ a.e.\ in }\Omega\,.}
		          \end{align}
		\end{proof}

    \subsection{$\mathcal{S}^1_D(\mathcal{T}_h)$-approximation of the $p$-Dirichlet problem}
		
		\qquad Given a right-hand side $f\in L^{p'}(\Omega)$, $p\in (1,\infty)$, and given a non-linear operator ${\AAA\colon\mathbb{R}^d\to \mathbb{R}^d}$ that satisfies Assumption \ref{assum:extra_stress} for $p\hspace{-0.1em}\in\hspace{-0.1em} (1,\infty)$ and $\delta\hspace{-0.1em}\ge \hspace{-0.1em}0$,  the $\smash{\mathcal{S}^1_D(\mathcal{T}_h)}$-approximation,~where~$\smash{\mathcal{S}^1_D(\mathcal{T}_h)}\coloneqq \mathcal{S}^1(\mathcal{T}_h)\cap \smash{W^{1,p}_D(\Omega)}$,  of \eqref{eq:pDirichletW1p} seeks for  $u_h^{\textit{c}}\in \mathcal{S}^1_D(\mathcal{T}_h)$ such that for every $v_h\in \mathcal{S}^1_D(\mathcal{T}_h)$,~it~holds
		\begin{align}
		    (\AAA(\nabla u_h^{\textit{c}}),\nabla v_h)_{\Omega}=(f,v_h)_{\Omega}\,.\label{eq:pDirichletS1D}
		\end{align}
			Resorting to the celebrated  theory of monotone operators, cf. \cite{Zei90B},~it~is~readily~apparent~that~\eqref{eq:pDirichletS1D} admits a unique solution. In what follows, we reserve~the~notation~$\smash{u_h^{\textit{c}}\in \mathcal{S}^1_D(\mathcal{T}_h)}$~for~this~solution. 
        The following best-approximation result applies:
		
		\begin{theorem}[Best-approximation]\label{P1_best-approx}
		    Let $\AAA\colon\mathbb{R}^d\to \mathbb{R}^d$ satisfy Assumption~\ref{assum:extra_stress}~for~${p\in (1,\infty)}$ and $\delta\ge 0$ and let $F\colon\mathbb{R}^d\to \mathbb{R}^d$ be defined by \eqref{eq:def_F} for the same ${p\in (1,\infty)}$ and $\delta\ge 0$. Then, there exists a constant $c>0$, depending only on the characteristics of $\AAA$, such that
		    \begin{align*}
		        \|F(\nabla u_h^{\textit{c}})-F(\nabla u)\|_{L^2(\Omega;\mathbb{R}^d)}^2\leq c\,\inf_{v_h\in \mathcal{S}^1_D(\mathcal{T}_h)}{\|F(\nabla v_h)-F(\nabla u)\|_{L^2(\Omega;\mathbb{R}^d)}^2}\,.
		    \end{align*}
		\end{theorem}

		\begin{proof} See \cite[Lemma 5.2]{DR07}.
		\end{proof}
		
		The combination of Theorem \ref{P1_best-approx} with the approximation properties of the Scott--Zhang~quasi-interpolation \hspace*{-0.15mm}operator \hspace*{-0.15mm}$I_h^{\textit{sz}}\colon \hspace*{-0.1em}W^{1,p}_D(\Omega)\hspace*{-0.15em}\to\hspace*{-0.15em}\mathcal{S}^1_D(\mathcal{T}_h)$, \hspace*{-0.15mm}cf.\ \hspace*{-0.15mm}\cite{SZ90}, \hspace*{-0.15mm}leads \hspace*{-0.15mm}to \hspace*{-0.15mm}the \hspace*{-0.15mm}following~\hspace*{-0.15mm}a~\hspace*{-0.15mm}priori~\hspace*{-0.15mm}error~\hspace*{-0.15mm}\mbox{estimate}.
		
		\begin{theorem}[A priori error estimate]\label{P1_apriori}
		    Let $\AAA\colon\mathbb{R}^d\to \mathbb{R}^d$ satisfy Assumption~\ref{assum:extra_stress}~for $p\in (1,\infty)$ and $\delta\ge 0$ and let $F\colon\mathbb{R}^d\to \mathbb{R}^d$ be defined by \eqref{eq:def_F} for the same $p\in (1,\infty)$ and $\delta\ge 0$. Moreover, assume that 
            \eqref{natural_regularity} is satisfied. Then, there exists a constant $c>0$, depending only~on~the~charac-teristics of $\AAA$ and the chunkiness $\omega_0>0$, such that
		    \begin{align*}
		        \|F(\nabla u_h^{\textit{c}})-F(\nabla u)\|_{L^2(\Omega;\mathbb{R}^d)}^2\leq c\,\|h_{\mathcal{T}}\nabla F(\nabla u)\|_{L^2(\Omega;\mathbb{R}^{d\times d})}^2\,.
		    \end{align*}
		\end{theorem}
		
		\begin{proof}
		        See \cite[Lemma 5.2]{DR07}.
		\end{proof}
	    
	\subsection{$\smash{\mathcal{S}^{1,\textit{cr}}_D(\mathcal{T}_h)}$-approximation of the $p$-Dirichlet problem}
	
	    \subsubsection{Discrete variational problem}
		
		\qquad Given a right-hand side $f\in L^{p'}(\Omega)$, $p\in (1,\infty)$, and given a non-linear operator ${\AAA\colon\mathbb{R}^d\to \mathbb{R}^d}$ that satisfies Assumption \ref{assum:extra_stress} for $p\in (1,\infty)$ and $\delta\ge 0$,  setting $f_h\coloneqq \Pi_h f\in \mathcal{L}^0(\mathcal{T}_h)$,~the~$\smash{\mathcal{S}^{1,\textit{cr}}_D(\mathcal{T}_h)}$-approximation \hspace*{-0.1mm}of  \hspace*{-0.1mm}\eqref{eq:pDirichletW1p} \hspace*{-0.1mm}seeks \hspace*{-0.1mm}for 
  \hspace*{-0.1mm}$u_h^{\textit{cr}}\hspace*{-0.15em}\in\hspace*{-0.15em} \smash{\mathcal{S}^{1,\textit{cr}}_D(\mathcal{T}_h)}$ \hspace*{-0.1mm}such \hspace*{-0.1mm}that~\hspace*{-0.1mm}for~\hspace*{-0.1mm}every~\hspace*{-0.1mm}${v_h\hspace*{-0.15em}\in\hspace*{-0.15em} \smash{\mathcal{S}^{1,\textit{cr}}_D(\mathcal{T}_h)}}$, it holds
		\begin{align}
		    (\AAA(\nabla_{\!h}u_h^{\textit{cr}}),\nabla_{\! h} v_h)_{\Omega}=(f_h,\Pi_h v_h)_{\Omega}\,.\label{eq:pDirichletS1crD}
		\end{align}
		Resorting to the celebrated  theory of monotone operators, cf. \cite{Zei90B},~it~is~readily~apparent~that~\eqref{eq:pDirichletS1crD} admits a unique solution. In what follows, we reserve~the~notation~$u_h^{\textit{cr}}\in \smash{\mathcal{S}^{1,\textit{cr}}_D(\mathcal{T}_h)}$~for~this~solution. 
		
		\subsubsection{Discrete minimization problem and discrete convex duality relations}\label{subsec:discrete_convex_duality}
		
		\qquad In the case  \eqref{special_case}, i.e., $\AAA\colon\mathbb{R}^d\to \mathbb{R}^d$, has a potential, cf. Remark \ref{rem:assum}, 
	    the variational~problem \eqref{eq:pDirichletS1crD} arises as an optimality condition of an equivalent convex minimization problem.
		
		\qquad \textit{Discrete primal problem.} In the case  \eqref{special_case}, a problem equivalent to \eqref{eq:pDirichletS1crD} is given by the minimization \hspace*{-0.15mm}of \hspace*{-0.15mm}the \hspace*{-0.15mm}\textit{discrete \hspace*{-0.15mm}$p$-Dirichlet \hspace*{-0.15mm}energy}, \hspace*{-0.15mm}i.e., \hspace*{-0.15mm}the \hspace*{-0.1mm}discrete \hspace*{-0.15mm}energy~\hspace*{-0.15mm}functional~\hspace*{-0.15mm}${I_h^{\textit{cr}}\hspace{-0.18em}:\hspace{-0.17em}\smash{\mathcal{S}^{1,\textit{cr}}_D(\mathcal{T}_h)}\hspace{-0.18em}\to\hspace{-0.18em} \mathbb{R}}$, for every $v_h\in \smash{\mathcal{S}^{1,\textit{cr}}_D(\mathcal{T}_h)}$ defined by
	    \begin{align}\label{eq:pDirichletPrimalCR}
		    I_h^{\textit{cr}}(v)\coloneqq \rho_{\varphi,\Omega}(\nabla_{\!h} v_h)-(f_h,\Pi_hv_h)_{\Omega}\,.
		\end{align}
		In what follows, we refer the minimization of the discrete $p$-Dirichlet energy \eqref{eq:pDirichletPrimalCR}~to~as~the~\textit{\mbox{discrete} primal problem}.
		Since the discrete $p$-Dirichlet energy is proper, 
		strictly~convex,~weakly~coercive, and lower semi-continuous,  the direct method in the calculus of variations, cf. \cite{Dac08}, 
		implies the existence of a unique minimizer, called the \textit{discrete primal solution}.  More precisely,
	since~the~discrete $p$-Dirichlet energy \eqref{eq:pDirichletPrimalCR}~is~Fr\'echet differentiable and for every $v_h,w_h\in \smash{\mathcal{S}^{1,\textit{cr}}_D(\mathcal{T}_h)}$,~it~holds
	\begin{align*}
	    \langle DI_h^{\textit{cr}}(v_h),w_h\rangle_{\smash{\mathcal{S}^{1,\textit{cr}}_D(\mathcal{T}_h)}}=(\AAA(\nabla_{\!h} v_h),\nabla_{\! h} w_h)_{\Omega}\,,
	\end{align*}
	the \hspace{-0.2mm}optimality \hspace{-0.2mm}conditions \hspace{-0.2mm}of \hspace{-0.2mm}the \hspace{-0.2mm}discrete \hspace{-0.2mm}primal \hspace{-0.2mm}problem \hspace{-0.2mm}and \hspace{-0.2mm}the \hspace{-0.2mm}convexity \hspace{-0.2mm}of \hspace{-0.2mm}the \hspace{-0.2mm}discrete~\hspace{-0.2mm}\mbox{$p$-Dirichlet} energy~\eqref{eq:pDirichletPrimalCR} imply that $\smash{u_h^{\textit{cr}}\in \smash{\mathcal{S}^{1,\textit{cr}}_D(\mathcal{T}_h)}}$ solves the discrete primal problem, i.e., is the unique minimizer of the discrete $p$-Dirichlet energy.
	
		\qquad \textit{Discrete dual problem.} Appealing to \cite[Section 5]{BKAFEM}, the \textit{discrete dual problem} consists in the maximization of the discrete energy functional $D_h^{\textit{rt}}\colon\mathcal{R}T^0_N(\mathcal{T}_h)\to \mathbb{R}\cup\{-\infty\}$, for every $y_h\in \mathcal{R}T^0_N(\mathcal{T}_h)$ defined by
	\begin{align}
		D_h^{\textit{rt}}(y_h)\coloneqq -\rho_{\varphi^*,\Omega}(\Pi_hy_h)-I_{\{-f_h\}}(\textup{div}\,y_h)\,.\label{eq:pDirichletDualCR}
	\end{align}
	Appealing \hspace*{-0.1mm}to \hspace*{-0.1mm}\cite[Proposition \hspace*{-0.1mm}3.1]{Bar21}, \hspace*{-0.1mm}the \hspace*{-0.1mm}discrete \hspace*{-0.1mm}dual \hspace*{-0.1mm}problem \hspace*{-0.1mm}admits \hspace*{-0.1mm}a \hspace*{-0.1mm}unique~\hspace*{-0.1mm}solution~\hspace*{-0.1mm}${\smash{z_h^{\textit{rt}}}\!\in\! \mathcal{R}T^0_N(\mathcal{T}_h)}$, i.e., a maximizer of \eqref{eq:pDirichletDualCR}, called the \textit{discrete dual solution}, and a \textit{discrete strong duality relation}, i.e., \hspace*{-0.1mm}$I_h^{\textit{cr}}(u_h^{\textit{cr}})\hspace{-0.1em}=\hspace{-0.1em}D_h^{\textit{rt}}(\smash{z_h^{\textit{rt}}})$, \hspace*{-0.1mm}applies.\ \hspace*{-0.1mm}In addition, cf.\ \hspace*{-0.1mm}\cite[\hspace*{-0.1mm}Proposition \hspace*{-0.1mm}2.1]{BKAFEM},
	\hspace*{-0.1mm}there~\hspace*{-0.1mm}hold~\hspace*{-0.1mm}the~\hspace*{-0.1mm}\textit{discrete~\hspace*{-0.1mm}\mbox{convex} optimality relations}
	\begin{alignat}{2}
 \Pi_h\smash{z_h^{\textit{rt}}}\cdot\nabla_{\! h} u_h^{\textit{cr}}&=\varphi^*(\vert\Pi_h\smash{z_h^{\textit{rt}}}\vert)+\varphi(\vert\nabla_{\! h}u_h^{\textit{cr}}\vert)&&\quad\text{ in }\mathcal{L}^0(\mathcal{T}_h)\,,\label{eq:pDirichletOptimalityCR1.2}\\
		\textup{div}\,\smash{z_h^{\textit{rt}}}&=-f_h&&\quad\text{ in }\mathcal{L}^0(\mathcal{T}_h)\,.\label{eq:pDirichletOptimalityCR1.1}
	\end{alignat}
	Note that, by the Fenchel--Young identity, cf. \cite[Proposition 5.1, p.\ 21]{ET99}, \eqref{eq:pDirichletOptimalityCR1.2} is equivalent to
	\begin{align}\label{eq:pDirichletOptimalityCR2}
	    \Pi_h\smash{z_h^{\textit{rt}}}=\AAA(\nabla_{\! h}u_h^{\textit{cr}})\quad\textup{ in }\mathcal{L}^0(\mathcal{T}_h)^d\,.
	\end{align}
	Moreover, cf.\ \cite[Proposition 3.1]{Bar21}, the unique solution $\smash{z_h^{\textit{rt}}}\in \mathcal{R}T^0_N(\mathcal{T}_h)$ of the discrete~dual~problem is given via the \textit{generalized Marini formula}
	\begin{align}
		\smash{z_h^{\textit{rt}}}=\AAA(\nabla_{\! h}u_h^{\textit{cr}})-{\frac{f_h}{d}}\big(\textup{id}_{\mathbb{R}^d}-\Pi_h\textup{id}_{\mathbb{R}^d}\big)\quad\textup{ in } \mathcal{R}T^0_N(\mathcal{T}_h)\,.\label{eq:gen_marini}
	\end{align}

	\section{Medius error analysis}\label{sec:medius}
	
	\qquad In this section, we establish a best-approximation result similar to the best-approximation result for the $\mathcal{S}^1_D(\mathcal{T}_h)$-approximation \eqref{eq:pDirichletS1D} of \eqref{eq:pDirichletW1p}, cf. \mbox{Theorem} \ref{P1_best-approx}, but now for the $\mathcal{S}^{1,\textit{cr}}_D(\mathcal{T}_h)$-approximation \eqref{eq:pDirichletS1crD}.
    
    \begin{theorem}\label{thm:best-approx} 
        Let $\AAA\colon\mathbb{R}^d\to \mathbb{R}^d$ satisfy  Assumption \ref{assum:extra_stress} for $p\in (1,\infty)$ and  $\delta\ge 0$. Moreover, let $\varphi\colon \mathbb{R}_{\ge 0}\to \mathbb{R}_{\ge 0}$ be defined by \eqref{eq:def_phi} and let $F\colon\mathbb{R}^d\to \mathbb{R}^d$ be defined by \eqref{eq:def_F}, each~for~the~same $p\in (1,\infty)$ and  $\delta\ge 0$. Then, there exists a constant $c>0$,
        depending only on the characteristics~of~$\AAA$ and the chunkiness $\omega_0>0$, such that
		\begin{align*}
		\|F(\nabla_{\!h} u_h^{\textit{cr}})-F(\nabla u)\|_{L^2(\Omega;\mathbb{R}^d)}^2&\leq  c\,  \inf_{v_h\in\mathcal{S}^1_D(\mathcal{T}_h)}{\big[\|F(\nabla  v_h)-F(\nabla u)\|_{L^2(\Omega;\mathbb{R}^d)}^2}+\mathrm{osc}_h(f,v_h)\big]\,,
		\end{align*}
		where for every $v_h\in \smash{\mathcal{L}^1(\mathcal{T}_h)}$ and $\mathcal{M}_h\subseteq \mathcal{T}_h $, we define $\mathrm{osc}_h(f,v_h,\mathcal{M}_h)\coloneqq  \sum_{T\in \mathcal{M}_h}{\mathrm{osc}_h(f,v_h,T)}$, 
		where $\mathrm{osc}_h(f,v_h,T)\coloneqq \rho_{(\varphi_{\vert \nabla v_h\vert})^*,T}(h_T(f-f_h))$ for all $T\in \mathcal{T}_h$, and $\mathrm{osc}_h(f,v_h)\vcentcolon =\mathrm{osc}_h(f,v_h,\mathcal{T}_h)$.
	\end{theorem}
		
		  Before we prove Theorem \ref{thm:best-approx}, we will first introduce some technical tools.
		  
		  \subsection{Node-averaging quasi-interpolation operator}
		  
		  \qquad The \hspace{-0.1mm}first \hspace{-0.1mm}tool \hspace{-0.1mm}is \hspace{-0.1mm}the \hspace{-0.1mm}node-averaging \hspace{-0.1mm}quasi-interpolation \hspace{-0.1mm}operator \hspace{-0.1mm}and~\hspace{-0.1mm}its~\hspace{-0.1mm}uniform~\hspace{-0.1mm}\mbox{approximation} and stability properties with respect to shifted $N$-functions, cf. \cite{Osw93,Sus96,EG21}.
		  
		  The node-averaging quasi-interpolation operator $I_h^{\textit{av}}\colon\mathcal{L}^1(\mathcal{T}_h)\hspace{-0.075em}\to\hspace{-0.075em} \mathcal{S}^1_D(\mathcal{T}_h)$, 
		  denoting~for~${z\hspace{-0.05em}\in\hspace{-0.05em} \mathcal{N}_h}$, by ${\mathcal{T}_h(z)\coloneqq \{T\in \mathcal{T}_h\mid z\in T\}}$, the set of elements sharing~$z$, for every ${v_h\in \mathcal{L}^1(\mathcal{T}_h)}$,~is~defined~by
	\begin{align*}
		I_h^{\textit{av}}v_h\coloneqq \sum_{z\in \smash{\mathcal{N}_h}}{\langle v_h\rangle_z\varphi_z}\,,\qquad \langle v_h\rangle_z\coloneqq \begin{cases}
			\frac{1}{\textup{card}(\mathcal{T}_h(z))}\sum_{T\in \mathcal{T}_h(z)}{(v_h|_T)(z)}&\;\text{ if }z\in \Omega\cup \Gamma_N\\
			0&\;\text{ if }z\in \Gamma_D
		\end{cases}\,,
	\end{align*}
	where we denote by $(\varphi_z)_{\smash{z\in \mathcal{N}_h}}$, the nodal basis of $\mathcal{S}^1(\mathcal{T}_h)$.
	If $\psi\colon\mathbb{R}_{\ge 0}\to \mathbb{R}_{\ge 0}$ is an $N$-function~with $\psi\hspace{-0.1em}\in \hspace{-0.1em}\Delta_2\cap \nabla_2$, then, there exists a constant $c\hspace{-0.1em}>\hspace{-0.1em}0$, depending on $\Delta_2(\psi)\hspace{-0.1em}>\hspace{-0.1em}0$ and the chunkiness~${\omega_0\hspace{-0.1em}>\hspace{-0.1em}0}$, such that for every $a\hspace{-0.1em}\ge\hspace{-0.1em} 0$, $v_h\hspace{-0.1em}\in\hspace{-0.1em} \smash{\smash{\mathcal{S}^{1,\textit{cr}}_D(\mathcal{T}_h)}}$,~${T\hspace{-0.1em}\in\hspace{-0.1em} \mathcal{T}_h}$, and $m\hspace{-0.1em}\in \hspace{-0.1em}\{0,1\}$, cf. Appendix \ref{sec:appendix}, we have that\footnote{Here, $\nabla_{\!h}^m\colon\mathcal{L}^1(\mathcal{T}_h)\to \mathcal{L}^{1-m}(\mathcal{T}_h)^{d^m}$, for every $v_h\in\mathcal{L}^1(\mathcal{T}_h)$ defined by $(\nabla_{\!h}^mv_h)|_T\coloneqq \nabla^m(v_h|_T)$ for all $T\in \mathcal{T}_h$, denotes the element-wise $m$-th gradient operator.\vspace{-1cm}}
	\begin{align}\label{eq:a6}
	    \hspace{-2mm}
	    \fint_T{\psi_a(h_T^m\vert\nabla^m_{\!h}(v_h-I_h^{\textit{av}}v_h)\vert)\,\mathrm{d}x}&\leq c\sum_{S\in \mathcal{S}_h(T)\setminus\Gamma_N}{\!\fint_S{\psi_a(\vert\jump{v_h}_S\vert)\,\mathrm{d}s}}
	    \leq c\fint_{\omega_T}{\!\psi_a(h_T\vert\nabla_{\!h}v_h\vert)\,\mathrm{d}x}\,.
      \hspace{-2mm}
	\end{align}
	where $\mathcal{S}_h(T)\coloneqq \{S\in \mathcal{S}_h\mid S\cap T\neq \emptyset\}$. 
	
	\subsection{Local efficiency estimates}
		  
		  \qquad The second tool involves the following local efficiency estimates.
		
		\begin{lemma}\label{lem:efficiency}
		Let $\AAA\colon\mathbb{R}^d\to \mathbb{R}^d$ satisfy  Assumption \ref{assum:extra_stress} for $p\in (1,\infty)$ and  $\delta\ge 0$. Moreover, let $\varphi\colon \mathbb{R}_{\ge 0}\to \mathbb{R}_{\ge 0}$ be defined by \eqref{eq:def_phi} and let $F\colon\mathbb{R}^d\to \mathbb{R}^d$ be defined by \eqref{eq:def_F}, each for the same $p\in (1,\infty)$ and  $\delta\ge 0$. Then, there exists a constant $c>0$, depending only on the characteristics of $\AAA$ and the chunkiness $\omega_0>0$, such that the following statements apply:
     \begin{itemize}[noitemsep,topsep=4pt,labelwidth=\widthof{(ii)},leftmargin=!]
      \item[(i)] For every $v_h\in \mathcal{L}^1(\mathcal{T}_h)$ and $T\in \mathcal{T}_h$, it holds
        \begin{align}
		        \begin{split}
		            \rho_{(\varphi_{\vert \nabla v_h\vert})^*,T}( h_T f )
		            \leq c\,\|F(\nabla  v_h)-F(\nabla u)\|_{L^2(T;\mathbb{R}^d)}^2+c\,\mathrm{osc}_h(f,v_h,T)\,,
		        \end{split}\label{lem:efficiency.1}
		\end{align}
      \item[(ii)] For every $v_h\in \mathcal{S}^1_D(\mathcal{T}_h)$ and $S\in \mathcal{S}_h^{i}$, it holds
      \begin{align}
		        h_S\|\jump{F(\nabla v_h)}_S\|_{L^2(S;\mathbb{R}^d)}^2&\leq c\,\|F(\nabla v_h)-F(\nabla u)\|_{L^2(\omega_S;\mathbb{R}^d)}^2+c\,\mathrm{osc}_h(f,v_h,\omega_S)\label{lem:efficiency.3}\,.
	  \end{align}
  \end{itemize}
		\end{lemma}
		
		The local efficiency estimate \eqref{lem:efficiency.3} can  be extended to arbitrary functions $v_h\in \mathcal{L}^1(\mathcal{T}_h)$.~For~this, however, one has to pay with a term quantifying the natural~distance,~cf.~Remark~\ref{rem:natural_dist},~to~$\mathcal{S}^1_D(\mathcal{T}_h)$.
		
		 \begin{corollary}\label{cor:efficiency}
	        Let $\AAA\colon\mathbb{R}^d\to \mathbb{R}^d$ satisfy  Assumption \ref{assum:extra_stress} for $p\in (1,\infty)$ and  $\delta\ge 0$, and let $F\colon\mathbb{R}^d\to \mathbb{R}^d$ be defined by \eqref{eq:def_F} for the same $p\in (1,\infty)$     and  $\delta\ge 0$. Then, there exists a constant $c>0$,
            depending only on the characteristics of $\AAA$ and the chunkiness $\omega_0>0$, such that for every $v_h\in \mathcal{L}^1(\mathcal{T}_h)$ and $S\in \mathcal{S}_h^{i}$, it holds
            \begin{align*}
                h_S\|\jump{F(\nabla_{\!h}v_h)}\|_{L^2(S;\mathbb{R}^d)}^2&\leq \smash{c\,\|F(\nabla_{\!h}  v_h)-F(\nabla     u)\|_{L^2(\omega_S;\mathbb{R}^d)}^2+c\,\textup{osc}_h(f,v_h,\omega_S)}\\&\quad+\smash{c\,\textup{dist}_F^2(v_h,\mathcal{S}^{1}_D(\mathcal{T}_h),\omega_S)}\,,
            \end{align*}
            where for every $v_h\in \mathcal{L}^1(\mathcal{T}_h)$ and $\mathcal{M}_h\subseteq \mathcal{T}_h$, we define\vspace{-0.5mm}
		    \begin{align*}
		        \textup{dist}_F^2(v_h,\mathcal{S}^{1}_D(\mathcal{T}_h),\mathcal{M}_h)\vcentcolon =\inf_{\tilde{v}_h\in\mathcal{S}^{1}_D(\mathcal{T}_h)}{\|F(\nabla_{\!h}  v_h)-F(\nabla     \tilde{v}_h)\|_{L^2(\mathcal{M}_h;\mathbb{R}^d)}^2}\,.
		    \end{align*}
	    \end{corollary}
		
		\begin{proof}[Proof (of Lemma \ref{lem:efficiency}).]
		    We extend the proofs of \cite[Lemma 9 \& Lemma 10]{DK08}.
		    
		    \textit{ad \eqref{lem:efficiency.1}.} Let $T\in\mathcal{T}_h$ be fixed, but arbitrary. Then, there exists a bubble~function~${b_T\in W^{1,p}_0(T)}$ such that $0\leq b_T\leq c $ in $T$, $\vert \nabla b_T\vert\leq c\,\smash{h^{-1}_T}$ in $T$ and  $\smash{\fint_T{b_T\,\mathrm{d}x}}=1$, where the constant~${c>0}$~\mbox{depends} only on the chunkiness $\omega_0>0$.
		    Using \eqref{eq:pDirichletW1p} and integration-by-parts, taking into account that $\nabla_{\!h}  v_h\in  \mathcal{L}^0(\mathcal{T}_h)^d$ and ${b_T\in W^{1,p}_0(T)}$ 
		    in doing so, for every $\lambda\in \mathbb{R}$, we find that
		    \begin{align}
		        (\AAA(\nabla u)-\AAA(\nabla v_h), \nabla(\lambda b_T))_T=(f,\lambda b_T)_T\,.\label{lem:efficiency.5}
		    \end{align}
		    For $\lambda_T\coloneqq \textrm{sgn}(f_h)((\varphi_{\vert\nabla v_h\vert})^*)'(h_T\vert f_h\vert )\in \mathbb{R}$, 
		    by the Fenchel--Young identity, cf.\ \cite[Prop.~5.1, p.\ 21]{ET99}, it holds
		    \begin{align}
		        \lambda_T (h_T f_h)=(\varphi_{\vert\nabla  v_h\vert})^*(h_T\vert f_h\vert)+\varphi_{\vert\nabla  v_h\vert}(\vert\lambda_T\vert)\quad\textup{ in }T \,.\label{lem:efficiency.6}
		    \end{align}
		    Then, for  the particular choice $\lambda=h_T\lambda_T\in \mathbb{R}$, cf.\ \eqref{lem:efficiency.6}, in \eqref{lem:efficiency.5}, we observe that
		    \begin{align}\label{lem:efficiency.7}
		        \begin{aligned}
		            \hspace{-2mm}\rho_{(\varphi_{\vert\nabla  v_h\vert})^*,T}(h_T  f_h)+\rho_{\varphi_{\vert\nabla v_h\vert},T}( \lambda_T)
		            &=(f,h_T\lambda_T b_T)_T+(f_h-f,h_T\lambda_T b_T)_T
		            \\&= (\AAA(\nabla u)-\AAA(\nabla v_h), \nabla(h_T\lambda_T b_T))_T\\&\quad +(f_h-f,h_T\lambda_T b_T)_T\,.
		         \end{aligned}
		    \end{align}
		    Applying element-wise the $\varepsilon$-Young inequality \eqref{ineq:young} with $\psi=\varphi_{\vert\nabla v_h\vert}$ in conjunction~with~\eqref{eq:hammera}, also using that $\vert b_T\vert+h_T\vert \nabla b_T\vert\leq c$ in $T$, we obtain
		    \begin{align}\label{lem:efficiency.8}
		        \begin{aligned}
		            (\AAA(\nabla u)-\AAA(\nabla v_h), \nabla(h_T\lambda_T b_T))_T&\leq c_\varepsilon\,\|F(\nabla v_h)-F(\nabla     u)\|_{L^2(T;\mathbb{R}^d)}^2+\varepsilon\,\rho_{\varphi_{\vert\nabla v_h\vert},T}(\lambda_T)\,,\\
		             (f_h-f,h_T\lambda_T b_T)_T&\leq c_\varepsilon\,\mathrm{osc}_h(f,v_h,T)+\varepsilon\,\rho_{\varphi_{\vert\nabla v_h\vert},T}(\lambda_T)\,.
		        \end{aligned}
		    \end{align}
		    Taking into account \eqref{lem:efficiency.8} in \eqref{lem:efficiency.7}, for sufficiently small $\varepsilon>0$, we deduce that
		    \begin{align}\label{lem:efficiency.9}
		        \begin{aligned}
		            \rho_{(\varphi_{\vert\nabla  v_h\vert})^*,T}( h_Tf_h)\leq c_\varepsilon\,\|F(\nabla v_h)-F(\nabla u)\|_{L^2(T;\mathbb{R}^d)}^2+c_\varepsilon\,\mathrm{osc}_h(f,v_h,T)\,.
		         \end{aligned}
		    \end{align}
		    Due to the convexity of $(\varphi_{\vert\nabla v_h(x)\vert})^*\colon\mathbb{R}_{\ge 0}\hspace{-0.1em}\to\hspace{-0.1em} \mathbb{R}_{\ge 0}$ for a.e. $\!x\in \Omega$ and ${\sup_{a\ge 0}{\Delta_2((\varphi_a)^*)}\hspace{-0.1em}<\hspace{-0.1em}\infty}$,~it~holds
		    \begin{align*}
		         \rho_{(\varphi_{\vert\nabla  v_h\vert})^*,T}( h_T f)\leq  c_\varepsilon\,\rho_{(\varphi_{\vert\nabla  v_h\vert})^*,T}( h_Tf_h)+c_\varepsilon\,\mathrm{osc}_h(f,v_h,T)\,,
		    \end{align*}
		    which in conjunction with \eqref{lem:efficiency.9} implies \eqref{lem:efficiency.1}.
		    
		    \textit{ad \eqref{lem:efficiency.3}.} Let $S\hspace{-0.15em}\in\hspace{-0.15em} \mathcal{S}_h^{i}$ be fixed, but arbitrary. Then, there exists a bubble~function~${b_S\hspace{-0.15em}\in\hspace{-0.15em} W^{1,p}_0(\omega_S)}$ such \hspace*{-0.1mm}that \hspace*{-0.1mm}$0\hspace{-0.1em}\leq\hspace{-0.1em} b_S\hspace{-0.15em}\leq\hspace{-0.1em} c $ \hspace*{-0.1mm}in \hspace*{-0.1mm}$\omega_S$, \hspace*{-0.1mm}$\vert \nabla b_S\vert\hspace{-0.1em}\leq\hspace{-0.1em}  c\,\smash{h^{-1}_S}$ in $\omega_S$, \hspace*{-0.1mm}and  \hspace*{-0.1mm}$\smash{\fint_S{b_S\,\mathrm{d}s}}\hspace{-0.1em}=\hspace{-0.1em}1$, \hspace*{-0.1mm}where~\hspace*{-0.1mm}the~\hspace*{-0.1mm}constant~\hspace*{-0.1mm}${c\hspace{-0.1em}>\hspace{-0.1em}0}$~\hspace*{-0.1mm}\mbox{depends} only on the chunkiness $\omega_0>0$.
		    Using \eqref{eq:pDirichletW1p} and integration-by-parts, taking into account~that $\nabla  v_h\in  \mathcal{L}^0(\mathcal{T}_h)^d$ and ${b_S\in W^{1,p}_0(\omega_S)}$ with $\smash{\int_S{b_S\,\mathrm{d}s}}=\vert S\vert$  in doing so, for every $\lambda\in \mathbb{R}$,~we~find~that
		    \begin{align}\label{lem:efficiency.10}
		        (\AAA(\nabla u)-\AAA(\nabla v_h), \nabla(\lambda b_S))_{\omega_S}=(f,\lambda b_S)_{\omega_S}- \vert S\vert \jump{\AAA(\nabla v_h)\cdot n}_S\lambda\,.
		    \end{align}
		    Let $T\in \mathcal{T}_h$ with $T\subseteq \omega_S$. For $\lambda_T^S\coloneqq \textrm{sgn}(\jump{\AAA(\nabla  v_h)\cdot n}_S)((\varphi_{\vert\nabla v_h(T)\vert})^*)'(\vert \jump{\AAA(\nabla v_h)\cdot n}_S\vert)\in \mathbb{R}$, where $\vert\nabla v_h(T)\vert\coloneqq \vert\nabla v_h\vert|_T\in \mathbb{R}$, 
		    by  the Fenchel--Young identity, cf. \cite[Prop.~5.1,~p.\ 21]{ET99},~it~holds
		    \begin{align}\label{lem:efficiency.11}
		        \jump{\AAA(\nabla v_h)\cdot n}_S\lambda_T^S=(\varphi_{\vert\nabla v_h(T)\vert})^*(\vert\jump{\AAA(\nabla v_h)\cdot n}_S\vert)+\varphi_{\vert\nabla v_h(T)\vert}(\vert\lambda_T^S\vert)\quad\textup{ in }T\,.
		    \end{align}
		    Then, for the particular choice $\smash{\lambda=\frac{\vert \omega_S\vert}{\vert S\vert}\lambda_T^S\in \mathbb{R}}$,~cf.\ \eqref{lem:efficiency.11},  in \eqref{lem:efficiency.10}, using that $\vert\jump{F(\nabla v_h)}_S\vert^2\sim(\varphi_{\vert \nabla v_h(T)\vert})^*(\jump{\AAA(\nabla v_h)\cdot n}_S)$, cf. \cite[(3.26)]{DK08}, we observe that
		    \begin{align}\label{lem:efficiency.12.1}
		        \begin{aligned}
		           c\,h_S\|\jump{F(\nabla v_h)}_S\|_{L^2(S;\mathbb{R}^d)}^2
		            &+
		            \rho_{\varphi_{\vert\nabla v_h(T)\vert},\omega_S}(\lambda_T^S)
		           \leq  \vert \omega_S\vert \jump{\AAA(\nabla v_h)\cdot n}_S\lambda_T^S
		            \\&\quad= \tfrac{\vert \omega_S\vert}{\vert S\vert}(\AAA(\nabla v_h)-\AAA(\nabla u), \nabla(\lambda_T^S     b_S))_{\omega_S}+\tfrac{\vert \omega_S\vert}{\vert S\vert}(f,\lambda_T^S b_S)_{\omega_S}\,.
		        \end{aligned}
		    \end{align}
		    Applying element-wise the $\varepsilon$-Young inequality \eqref{ineq:young} with $\psi=\varphi_{\vert\nabla v_h\vert }$ in conjunction with \eqref{eq:hammera}, using that $\vert b_S\vert+h_S\vert \nabla b_S\vert\hspace{-0.1em}\leq\hspace{-0.1em} c$ in $\omega_S$ and  $\vert \omega_S\vert\hspace{-0.1em}\sim\hspace{-0.1em} h_S\vert S\vert$ uniformly in $S\hspace{-0.1em}\in\hspace{-0.1em} \mathcal{S}_h$ and $T\hspace{-0.1em}\in \hspace{-0.1em}\mathcal{T}_h$~with~${T\hspace{-0.1em}\subseteq\hspace{-0.1em} \omega_S}$, we obtain
		    \begin{align}\label{lem:efficiency.13}
		        \hspace{-2mm}\begin{aligned}
          \tfrac{\vert \omega_S\vert}{\vert S\vert}(\AAA(\nabla v_h)-\AAA(\nabla u), \nabla(\lambda_T^S     b_S))_{\omega_S}&\leq c_\varepsilon\,\|F(\nabla v_h)-F(\nabla u)\|_{L^2(\omega_S;\mathbb{R}^d)}^2+\varepsilon\,\rho_{\varphi_{\vert\nabla v_h\vert},\omega_S}(\lambda_T^S)\,,\\
		         \tfrac{\vert \omega_S\vert}{\vert S\vert}(f,\lambda_T^S b_S)_{\omega_S}&\leq 
		         c_\varepsilon\,\rho_{(\varphi_{\vert\nabla v_h\vert})^*,\omega_S}( h_{\mathcal{T}} f )+\varepsilon\,\rho_{\varphi_{\vert\nabla v_h\vert},\omega_S}(\lambda_T^S)\,.
		         \end{aligned}\hspace{-2mm}
		    \end{align}
            The shift change \eqref{lem:shift-change.1} on $T'\in \mathcal{T}_h\setminus\{T\}$ with $T'\subseteq \omega_S$ further yields that
            \begin{align}\label{lem:efficiency.14}
                \smash{\rho_{\varphi_{\vert\nabla v_h\vert},\omega_S}(\lambda_T^S)\leq c\,\rho_{\varphi_{\vert\nabla v_h(T)\vert},\omega_S}(\lambda_T^S)+c\,h_S\|\jump{F(\nabla v_h)}_S\|_{L^2(S;\mathbb{R}^d)}^2\,.}
            \end{align}
		    For sufficiently small $\varepsilon>0$,  using \eqref{lem:efficiency.1}, we conclude \eqref{lem:efficiency.3} from \eqref{lem:efficiency.13} and \eqref{lem:efficiency.14} in \eqref{lem:efficiency.12.1}.
		  \end{proof}
		  
		  \begin{proof}[Proof (of Corollary \ref{cor:efficiency}).]
		  For arbitrary $\tilde{v}_h\in \mathcal{S}^{1}_D(\mathcal{T}_h)$, resorting to the discrete trace inequality~\cite[Lemma A.16, (A.18)]{kr-phi-ldg} and \eqref{lem:efficiency.3}, we find that
		    \begin{align}
		        \smash{h_S\,\|\jump{F(\nabla_{\!h}v_h)}\|_{L^2(S;\mathbb{R}^d)}^2}&\leq \smash{2\,h_S\,\|\jump{F(\nabla_{\!h}v_h)-F(\nabla \tilde{v}_h)}\|_{L^2(S;\mathbb{R}^d)}^2 +2\,h_S\,\|\jump{F(\nabla \tilde{v}_h)}\|_{L^2(S;\mathbb{R}^d)}^2}\notag
		        \\&\leq c\,\|\smash{F(\nabla_{\!h}v_h)-F(\nabla \tilde{v}_h)}\|_{L^2(\omega_S;\mathbb{R}^d)}^2\label{cor:efficiency.1}
		        \\&\quad+c\,\|\smash{F(\nabla \tilde{v}_h)-F(\nabla u)}\|_{L^2(\omega_S;\mathbb{R}^d)}^2
		        +c\,\textup{osc}_h(f,\tilde{v}_h,\omega_S)\notag
		        \,.
		    \end{align}
		   The shift change \eqref{lem:shift-change.3} yields that
		   \begin{align}
		       \textup{osc}_h(f,\tilde{v}_h,\omega_S)\leq c\,\textup{osc}_h(f,v_h,\omega_S)+c\,\|\smash{F(\nabla_{\!h}v_h)-F(\nabla \tilde{v}_h)}\|_{L^2(\omega_S;\mathbb{R}^d)}^2\,.\label{cor:efficiency.2}
		   \end{align}
		   Using in \eqref{cor:efficiency.1} both \eqref{cor:efficiency.2} and  
		   \begin{align*}
		        \begin{aligned}
		       \|\smash{F(\nabla \tilde{v}_h)-F(\nabla u)}\|_{L^2(\Omega;\mathbb{R}^d)}^2\leq 2\,\|\smash{F(\nabla_{\!h}v_h)-F(\nabla \tilde{v}_h)}\|_{L^2(\Omega;\mathbb{R}^d)}^2+2\,\|\smash{F(\nabla_{\!h}v_h)-F(\nabla u)}\|_{L^2(\Omega;\mathbb{R}^d)}^2\,,
		        \end{aligned}
		   \end{align*}
		  and, subsequently, taking the infimum with respect to $\tilde{v}_h\in \mathcal{S}^1_D(\mathcal{T}_h)$, we conclude the assertion.
		  \end{proof}
		
	\subsection{Patch-shift-to-element-shift estimate}\enlargethispage{9mm}
	
	    \qquad The third tool involves the following estimate allowing us to pass from element-patch-shifts to element-shifts and, thus,  to deploy quasi-interpolation operators that~are~locally~element-to-patch stable, e.g., the node-averaging quasi-interpolation operator $I_h^{\textit{av}}\colon\mathcal{S}^{1,\textit{cr}}_D(\mathcal{T}_h)\to \mathcal{S}^1_D(\mathcal{T}_h)$, cf. \eqref{eq:a6}.

		\begin{lemma}\label{lem:patch_to_element}
		    Let $\AAA\colon\mathbb{R}^d\to \mathbb{R}^d$ satisfy  Assumption \ref{assum:extra_stress} for $p\in (1,\infty)$ and  $\delta\ge 0$. Moreover, let $\varphi\colon\mathbb{R}_{\ge 0}\to \mathbb{R}_{\ge 0}$ be defined by \eqref{eq:def_phi} and let $F\colon\mathbb{R}^d\to \mathbb{R}^d$ be defined by \eqref{eq:def_F}, each for the same $p\in (1,\infty)$ and  $\delta\ge 0$. Then, there exists a constant $c>0$, depending only on the~characteristics~of $\AAA$ and the chunkiness $\omega_0>0$, such that for every $v_h\in \mathcal{L}^1(\mathcal{T}_h)$, $y\in L^p(\Omega;\mathbb{R}
		   ^d)$, and $T\in \mathcal{T}_h$,~it~holds
		    \begin{align*}
		        \smash{\rho_{\varphi_{\vert \nabla_{\! h} v_h(T)\vert},\omega_T}(y)\leq c\, \rho_{\varphi_{\vert \nabla_{\! h} v_h\vert},\omega_T}(y)+\big\|\smash{h_{\mathcal{S}}^{1/2}}\jump{F(\nabla_{\! h} v_h)}\big\|_{L^2(\mathcal{S}_h^{i}(T);\mathbb{R}^d)}^2\,,}
		    \end{align*}
		    where $\mathcal{S}_h^{i}(T)\coloneqq \mathcal{S}_h(T)\cap \mathcal{S}_h^{i}$ and we write $\vert \nabla_{\! h} v_h(T)\vert\coloneqq \vert \nabla_{\! h} v_h\vert|_T$ to indicate that the shift on the whole patch $\omega_T$ depends on the value of $\vert \nabla_{\! h} v_h\vert$ on the element $T$.
		\end{lemma}
		
		\begin{proof}
		    The proof is based on the argumentation as in \cite[p.\ 9 \& 10]{DK08}. Applying for every $T'\in \mathcal{T}_h$ with $T'\subseteq \omega_T$, the shift change \eqref{lem:shift-change.1}, we arrive at
		    \begin{align}\label{lem:patch_to_element.1}
		            \smash{\rho_{\varphi_{\vert \nabla_{\! h} v_h(T)\vert},\omega_T}(y)\leq c\, \rho_{\varphi_{\vert \nabla_{\! h} v_h\vert},\omega_T}(y)+c\,\|F(\nabla_{\! h} v_h(T))-F(\nabla_{\! h} v_h)\|_{L^2(\omega_T;\mathbb{R}^d)}^2\,.}
		    \end{align}
		    Since one can reach each $T'\in \mathcal{T}_h$ with $T'\subseteq \omega_T$ by passing through finite number~of~sides~${S\in \mathcal{S}_h^{i}(T)}$ (depending on the chunkiness $\omega_0\hspace{-0.1em}>\hspace{-0.1em}0$), for every $T'\hspace{-0.1em}\in\hspace{-0.1em} \mathcal{T}_h$ with $T'\hspace{-0.1em}\subseteq \hspace{-0.1em}\omega_T$, 
		    we deduce that
		    \begin{align}\label{lem:patch_to_element.2}
		        \smash{c\,\|F(\nabla_{\! h} v_h(T))-F(\nabla_{\! h} v_h)\|_{L^2(T';\mathbb{R}^d)}^2\leq c\,\big\|\smash{h_{\mathcal{S}}^{1/2}}\jump{F(\nabla_{\! h} v_h)}\big\|_{L^2(\mathcal{S}_h^{i}(T);\mathbb{R}^d)}^2\,.}
		    \end{align}
		    Eventually, using \eqref{lem:patch_to_element.2} in \eqref{lem:patch_to_element.1}, we conclude the assertion.
		\end{proof}
	    
    \subsection{Proof of Theorem \ref{thm:best-approx}}\vspace*{-1mm}
	    
	   \qquad Eventually, we have everything at our disposal to prove Theorem \ref{thm:best-approx}.\vspace*{-1mm}
	    
        \begin{proof}[Proof (of Theorem \ref{thm:best-approx})] Let $v_h\in \mathcal{S}^1_D(\mathcal{T}_h)$ be arbitrary 
				and introduce $e_h\coloneqq v_h-u_h^{\textit{cr}}\in \smash{\mathcal{S}^{1,\textit{cr}}_D(\mathcal{T}_h)}$. Then, resorting to \eqref{eq:pDirichletW1p}, \eqref{eq:pDirichletS1crD}, and $f-f_h\perp \Pi_h e_h$ in $L^2(\Omega)$, we arrive at
				\begin{align}
				\begin{aligned}
					( \AAA(\nabla v_h)-\AAA(\nabla_{\!h} u_h^{\textit{cr}}),\nabla_{\!h} e_h )_\Omega&=
					(\AAA(\nabla v_h),\nabla_{\!h}( e_h- I_h^{\textit{av}} e_h) )_\Omega	\\&\quad+( f,I_h^{\textit{av}} e_h-e_h)_\Omega
					\\&\quad+( \AAA(\nabla v_h)-\AAA(\nabla u) ,\nabla I_h^{\textit{av}} e_h)_\Omega
					\\&\quad+(f-f_h,e_h-\Pi_h e_h)_\Omega
			        \\&=\vcentcolon I_h^1+I_h^2+I_h^3+I_h^4\,.
			        \end{aligned}\label{thm:best-approx.1}
				\end{align}
				
				\textit{ad $I_h^1$.} Using that $\jump{(\AAA(\nabla v_h)\cdot n)( e_h- I_h^{\textit{av}} e_h)}_S=\jump{\AAA(\nabla v_h)\cdot n}_S\{e_h- I_h^{\textit{av}} e_h\}_S +\{{\AAA(\nabla v_h)\cdot n}\}_S$ $ \jump{e_h-I_h^{\textit{av}} e_h}_S$ on $S$, $\int_S{\jump{e_h-I_h^{\textit{av}} e_h}_S\,\textup{d}s}=0$, and $\{\AAA(\nabla v_h)\cdot n\}_S=\textup{const}$ on $S$ for all ${S\in \mathcal{S}_h^{i}}$, an element-wise integration-by-parts,  the discrete trace inequality~\mbox{\cite[Lemma~12.8]{EG21}},~and~\eqref{eq:a6} with $\psi=\vert\cdot\vert$ and $a=0$, we find that
				\begin{align}
					\begin{aligned}
					I_h^1&=\sum_{S\in \mathcal{S}_h^{i}}{\int_S{\jump{\AAA(\nabla v_h)\cdot n}_S\{e_h- I_h^{\textit{av}} e_h\}_S\,\textup{d}s}} 
						\\[-0.5mm]&\leq c\sum_{S\in \mathcal{S}_h^{i}}{\vert \jump{\AAA(\nabla v_h)\cdot n}_S\vert  \sum_{T\in \mathcal{T}_h;S\subseteq \partial T}{h_T^{-1}\int_T{\vert e_h- I_h^{\textit{av}} e_h\vert \,\textup{d}s}}} \\[-0.5mm]&
					\leq c \sum_{S\in \mathcal{S}_h^{i}}{ \sum_{T\in \mathcal{T}_h;S\subseteq \partial T}{\int_{\omega_T}{\vert \jump{\AAA(\nabla v_h)}_S\vert\vert \nabla_ he_h\vert\,\textup{d}x}}}\,.
				\end{aligned}	\label{thm:best-approx.2}
				\end{align}
				Applying \hspace*{-0.1mm}in \hspace*{-0.1mm}\eqref{thm:best-approx.2} \hspace*{-0.1mm}patch-wise \hspace*{-0.1mm}the \hspace*{-0.1mm}$\varepsilon$-Young \hspace*{-0.1mm}inequality \hspace*{-0.1mm}\eqref{ineq:young} \hspace*{-0.1mm}with \hspace*{-0.1mm}$\psi\hspace*{-0.1em}=\hspace*{-0.1em}\smash{\varphi_{\vert\nabla v_h(T)\vert}}$, \hspace*{-0.1mm}where,~\hspace*{-0.1mm}for~\hspace*{-0.1mm}every~\hspace*{-0.1mm}$T\hspace*{-0.1em}\in\hspace*{-0.1em} \mathcal{T}_h$, we \hspace*{-0.1mm}write \hspace*{-0.1mm}to \hspace*{-0.1mm}indicate \hspace*{-0.1mm}that \hspace*{-0.1mm}the \hspace*{-0.1mm}shift \hspace*{-0.1mm}on \hspace*{-0.1mm}the \hspace*{-0.1mm}whole \hspace*{-0.1mm}patch \hspace*{-0.1mm}$\omega_T$ \hspace*{-0.1mm}depends \hspace*{-0.1mm}on \hspace*{-0.1mm}the \hspace*{-0.1mm}value~\hspace*{-0.1mm}of~\hspace*{-0.1mm}$\vert\nabla v_h\vert$~\hspace*{-0.1mm}on~\hspace*{-0.1mm}$T$, together with
				$(\varphi_{\vert \nabla v_h(T)\vert})^*(\vert \jump{\AAA(\nabla v_h)}_S\vert)\sim\vert \jump{F(\nabla v_h)}_S\vert^2$ in $T$ for~all~${T\in \mathcal{T}_h}$ with $S\subseteq \partial  T$~(cf.~\eqref{eq:hammera}),  
				and the finite overlapping of the element~patches~$\omega_T$,~${T\in \mathcal{T}_h}$, for every $\varepsilon>0$, we conclude that
				\begin{align}\label{thm:best-approx.3}
				    \begin{aligned}
					    I_h^1&\leq  c\sum_{S\in \mathcal{S}_h^{i}}{\sum_{T\in \mathcal{T}_h;S\subseteq \partial T}{\int_{\omega_T}{c_\varepsilon\,(\varphi_{\vert     \nabla v_h(T)\vert})^*(\vert \jump{\AAA(\nabla v_h)}_S\vert)+\varepsilon\,\varphi_{\vert\nabla v_h(T)\vert}(\vert \nabla_     he_h\vert)\,\textup{d}x}}}
					    	\\
					    	&
					    	\leq   c_\varepsilon\,\big\|\smash{h_{\mathcal{S}}^{1/2}}\jump{F(\nabla v_h)}\big\|_{L^2(\mathcal{S}_h^{i};\mathbb{R}^d)}^2+  \varepsilon\,\sum_{T\in \mathcal{T}_h}{      \rho_{\varphi_{\vert\nabla v_h(T)\vert},\omega_T}(\nabla_{\!h} e_h)}\,.
					\end{aligned}
				\end{align}
                    Appealing to Lemma \ref{lem:patch_to_element} with $y= \nabla_{\!h}     e_h\in \smash{L^p(\Omega;\mathbb{R}^d)}$, we have that
				\begin{align}
				\sum_{T\in \mathcal{T}_h}{\rho_{\varphi_{\vert \nabla v_h(T)\vert },\omega_T}(\nabla_{\!h} e_h)}&\leq c\,\rho_{\varphi_{\vert\nabla v_h\vert},\Omega}(\nabla_{\!h} e_h)+c\,\big\|\smash{h_{\mathcal{S}}^{1/2}}\jump{F(\nabla v_h)}\big\|_{L^2(\mathcal{S}_h^{i};\mathbb{R}^d)}^2\,.
					\label{thm:best-approx.6}
				\end{align}
				Thus, resorting in \eqref{thm:best-approx.3} to \eqref{thm:best-approx.6}, \eqref{lem:efficiency.3}, and \eqref{eq:hammera}, for every $\varepsilon>0$, we deduce that
				\begin{align}\label{thm:best-approx.3.1}
					    I_h^1&
					    	\leq   c_\varepsilon\,\big[\|F(\nabla v_h)-F(\nabla u)\|_{L^2(\Omega;\mathbb{R}^d)}^2+\mathrm{osc}_h(f,v_h)\big]
					    	+ \varepsilon\,c\, 
					    	\|F(\nabla v_h)-F(\nabla_{\!h} u_h^{\textit{cr}})\|_{L^2(\Omega;\mathbb{R}^d)}^2
					    	\,.
				\end{align}
				
				\textit{ad $I_h^2$.}
				 Applying element-wise the $\varepsilon$-Young inequality \eqref{ineq:young} with $\psi=\smash{\varphi_{\vert\nabla v_h\vert}}$,~for~every~${\varepsilon>0}$, we obtain\vspace*{-1mm}
				\begin{align}
					\begin{aligned}
						I_h^2&\leq c_\varepsilon\,\rho_{(\varphi_{\vert \nabla v_h\vert})^*,\Omega}(h_{\mathcal{T}} f)+
						\varepsilon\,\rho_{\varphi_{\vert\nabla v_h\vert},\Omega}\big(h_{\mathcal{T}}^{-1}(e_h-I_h^{\textit{av}}e_h)\big)\,.
					\end{aligned}\label{thm:best-approx.4}
				\end{align}
				Then, using element-wise the Orlicz-approximation property of $I_h^{\textit{av}}\!:\!\smash{\mathcal{S}^{1,\textit{cr}}_D(\mathcal{T}_h)}\!\to\! \mathcal{S}^1_D(\mathcal{T}_h)$,~cf.~\eqref{eq:a6}, with $\psi=\varphi$ and $a=\vert \nabla v_h(T)\vert$, 	where, for every $T\in \mathcal{T}_h$, we write $\vert\nabla v_h(T)\vert$ to indicate that the shift on the whole patch $\omega_T$ depends on the value of $\vert\nabla v_h\vert$ on the element $T$, we find that
				\begin{align}\label{thm:best-approx.5}
				    \rho_{\varphi_{\vert\nabla v_h\vert},\Omega}\big(h_{\mathcal{T}}^{-1}(e_h-I_h^{\textit{av}}e_h)\big)\leq c\,\sum_{T\in \mathcal{T}_h}{\rho_{\varphi_{\vert \nabla v_h(T)\vert },\omega_T}(\nabla_{\!h} e_h)}\,.
				\end{align}
				Using \eqref{thm:best-approx.5} and \eqref{thm:best-approx.6} together with \eqref{lem:efficiency.1}, \eqref{lem:efficiency.3}, and \eqref{eq:hammera} in \eqref{thm:best-approx.4}, for every~$\varepsilon>0$,~we~arrive~at
				\begin{align}\label{thm:best-approx.6.1}
				    	I_h^2&\leq c_\varepsilon\,\big[\|F(\nabla v_h)-F(\nabla u)\|_{L^2(\Omega;\mathbb{R}^d)}^2+\mathrm{osc}_h(f,v_h)\big]
				    	+ \varepsilon\,c\,
					    	\|F(\nabla v_h)-F(\nabla_{\!h} u_h^{\textit{cr}})\|_{L^2(\Omega;\mathbb{R}^d)}^2
					    	\,.
				\end{align}
				
				\textit{ad $I_h^3$.}
				Applying element-wise the $\varepsilon$-Young inequality \eqref{ineq:young} with $\psi=\varphi_{\vert\nabla v_h\vert }$,~for~every~${\varepsilon>0}$, we obtain
				\begin{align}
					\begin{aligned}
					I_h^3&\leq c_\varepsilon\,\|F(\nabla v_h)-F(\nabla u)\|_{L^2(\Omega;\mathbb{R}^d)}^2+\varepsilon\,\rho_{\varphi_{\vert\nabla v_h\vert },\Omega}(\nabla I_h^{\textit{av}}e_h)\,.
				\end{aligned}\label{thm:best-approx.7}
				\end{align}
				Then, using element-wise the Orlicz-stability property of $I_h^{\textit{av}}\colon \smash{\mathcal{S}^{1,\textit{cr}}_D(\mathcal{T}_h)}\to \mathcal{S}^1_D(\mathcal{T}_h)$, cf. \eqref{eq:a6}, with $\psi=\varphi$ and $a=\vert \nabla v_h(T)\vert$,	where we, again, for every $T\in \mathcal{T}_h$, write $\vert\nabla v_h(T)\vert$ to indicate that the shift on the whole patch $\omega_T$ depends on the value of $\vert\nabla v_h\vert$ on the element $T$,~we~find~that
				\begin{align}\label{thm:best-approx.8}
				    \rho_{\varphi_{\vert\nabla v_h\vert },\Omega}(\nabla I_h^{\textit{av}}e_h)\leq c\,\sum_{T\in \mathcal{T}_h}{\rho_{\varphi_{\vert\nabla v_h(T)\vert },\Omega}(\nabla_{\!h} e_h)}\,.
				\end{align}
				Using \eqref{thm:best-approx.8} and \eqref{thm:best-approx.6} in conjunction with \eqref{lem:efficiency.3} and \eqref{eq:hammera} in  \eqref{thm:best-approx.7}, for every $\varepsilon>0$,~we~arrive~at
				\begin{align}\label{thm:best-approx.8.1}
				    	I_h^3&\leq c_\varepsilon\,\big[\|F(\nabla v_h)-F(\nabla u)\|_{L^2(\Omega;\mathbb{R}^d)}^2+\mathrm{osc}_h(f,v_h)\big]
				    	+ \varepsilon\,c\,
					    	\|F(\nabla v_h)-F(\nabla_{\!h} u_h^{\textit{cr}})\|_{L^2(\Omega;\mathbb{R}^d)}^2
					    	\,.
				\end{align}
				
				\textit{ad $I_h^4$.}
				Applying element-wise the $\varepsilon$-Young inequality \eqref{ineq:young}  with $\psi=\smash{\varphi_{\vert\nabla v_h\vert}}$ and the Orlicz-approximation property of $\Pi_h\colon \mathcal{L}^1(\mathcal{T}_h)\to \mathcal{L}^0(\mathcal{T}_h)$ (cf. \cite[(A.9)]{dkrt-ldg}),  for every~${\varepsilon>0}$,~we~obtain
				\begin{align}
					\begin{aligned}
					I_h^4
                &\leq c_\varepsilon\,\mathrm{osc}_h(f,v_h)+\varepsilon\,\rho_{\varphi_{\vert\nabla v_h\vert },\Omega}\big(h_{\mathcal{T}}^{-1}(e_h-\Pi_h e_h)\big)\\&\leq c_\varepsilon\,\mathrm{osc}_h(f,v_h)+\varepsilon\,\rho_{\varphi_{\vert\nabla v_h\vert },\Omega}(\nabla_{\!h}e_h)\,.
				\end{aligned}	\label{thm:best-approx.9}
				\end{align}
				Thus, using 
                \eqref{eq:hammera} in \eqref{thm:best-approx.9}, for every $\varepsilon>0$, we find that
				\begin{align}\label{thm:best-approx.9.1}
						I_h^4\leq c_\varepsilon\,\mathrm{osc}_h(f,v_h)
						+ \varepsilon\,c\,     
					    	\|F(\nabla v_h)-F(\nabla_{\!h} u_h^{\textit{cr}})\|_{L^2(\Omega;\mathbb{R}^d)}^2
					    	\,.	
				\end{align}
				Then, combining \eqref{thm:best-approx.3.1}, \eqref{thm:best-approx.6.1}, \eqref{thm:best-approx.8.1}, and \eqref{thm:best-approx.9.1} 
				in \eqref{thm:best-approx.1}, for every $\varepsilon>0$, we conclude that
				\begin{align}\label{thm:best-approx.11}
					\begin{aligned}
				( \AAA(\nabla v_h)-\AAA(\nabla_{\!h} u_h^{\textit{cr}}) ,\nabla_{\!h}e_h )_\Omega&\leq c_\varepsilon\,\big[\|F(\nabla v_h)-F(\nabla u)\|_{L^2(\Omega;\mathbb{R}^d)}^2+\mathrm{osc}_h(f,v_h)\big]\\&\quad +\varepsilon\,c\,\|F(\nabla v_h)-F(\nabla_{\!h}u_h^{\textit{cr}})\|_{L^2(\Omega;\mathbb{R}^d)}^2\,
				.
			\end{aligned}
				\end{align}
				Resorting in \eqref{thm:best-approx.11} to  \eqref{eq:hammera}, for $\varepsilon>0$ sufficiently small, for every $v_h\in \mathcal{S}^1_D(\mathcal{T}_h)$, we arrive at
				\begin{align}\label{thm:best-approx.12}
					\begin{aligned}
						\|F(\nabla v_h)-F(\nabla_{\!h}u_h^{\textit{cr}})\|_{L^2(\Omega;\mathbb{R}^d)}^2\leq c_\varepsilon\,\big[\|F(\nabla v_h)-F(\nabla u)\|_{L^2(\Omega;\mathbb{R}^d)}^2+\mathrm{osc}_h(f,v_h)\big]\,.
					\end{aligned}
				\end{align}	
				From \eqref{thm:best-approx.12}, in turn, we deduce that
				\begin{align}\label{thm:best-approx.13}
				    \begin{aligned}
				        \|F(\nabla_{\!h}u_h^{\textit{cr}})-F(\nabla u)\|_{L^2(\Omega;\mathbb{R}^d)}^2&\leq 2\, \|F(\nabla v_h)-F(\nabla_{\!h}u_h^{\textit{cr}})\|_{L^2(\Omega;\mathbb{R}^d)}^2
				       \\&\quad+2\,\|F(\nabla v_h)-F(\nabla u)\|_{L^2(\Omega;\mathbb{R}^d)}^2
				       \\&\leq 
				        c_\varepsilon\,\big[\|F(\nabla v_h)-F(\nabla u)\|_{L^2(\Omega;\mathbb{R}^d)}^2+\mathrm{osc}_h(f,v_h)\big]\,.
				    \end{aligned}
				\end{align}
				Taking in \eqref{thm:best-approx.13} the infimum with respect to $v_h\in \smash{\mathcal{S}^1_D(\mathcal{T}_h)}$, we conclude~the~assertion.
			\end{proof}

	    An immediate consequence of the medius error analysis (cf.\ Theorem \ref{thm:best-approx}) is the observation that the distance~of every $v_h\in \smash{\mathcal{S}^1_D(\mathcal{T}_h)}$ to $u_h^{\textit{cr}}\in \smash{\smash{\mathcal{S}^{1,\textit{cr}}_D(\mathcal{T}_h)}}$, up to oscillation terms, is controlled by the distance of $v_h\hspace{-0.1em}\in\hspace{-0.1em} \smash{\mathcal{S}^1_D(\mathcal{T}_h)}$ to $u\hspace{-0.1em}\in\hspace{-0.1em} \smash{W^{1,p}_D(\Omega)}$, each measured in the natural distance,~cf.~\mbox{Remark}~\ref{rem:natural_dist}. 
	    This can also be interpreted as a kind of efficiency property.
		
		\begin{corollary}\label{cor:best-approx}
		    Let $\AAA\colon\mathbb{R}^d\to \mathbb{R}^d$ satisfy  Assumption \ref{assum:extra_stress} for $p\in(1,\infty)$ and  $\delta\ge 0$, and let $F\colon\mathbb{R}^d\to \mathbb{R}^d$ be defined by \eqref{eq:def_F} for the same $p\in (1,\infty)$ and  $\delta\ge 0$. Then, there exists a constant $c>0$, depending on the characteristics of $\AAA$ and the chunkiness $\omega_0>0$, such that for every $v_h\in \smash{\mathcal{S}^1_D(\mathcal{T}_h)}$, it holds
		    \begin{align*}
		        \|F(\nabla v_h)-F(\nabla_{\!h}u_h^{\textit{cr}})\|_{L^2(\Omega;\mathbb{R}^d)}^2\leq c\,\big[\|F(\nabla v_h)-F(\nabla u)\|_{L^2(\Omega;\mathbb{R}^d)}^2+\mathrm{osc}_h(f,v_h)\big]\,.
		    \end{align*}
		\end{corollary}	
		
		\begin{proof}
		    Immediate consequence of \eqref{thm:best-approx.12}.
		\end{proof}

		It is possible to establish a best-approximation result inverse to Theorem \ref{thm:best-approx}.~For~this,~however, we need to pay by jump terms measuring the natural distance, cf. Remark \ref{rem:natural_dist},~of~\mbox{Crouzeix--Raviart} functions to $\mathcal{S}^1_D(\mathcal{T}_h)$, cf.\ Corollary \ref{cor:efficiency}.
		
		\begin{theorem}\label{thm:best-approxP1CR} 
            Let $\AAA\colon\mathbb{R}^d\to \mathbb{R}^d$ satisfy  Assumption \ref{assum:extra_stress} for $p\in (1,\infty)$ and  $\delta\ge 0$,~and~let~$F\colon\mathbb{R}^d\to \mathbb{R}^d$ be defined by \eqref{eq:def_F} for the same $p\in (1,\infty)$     and  $\delta\ge 0$. Then, there exists a
		        constant $c>0$, depending only on the characteristics of $\AAA$~and~the~chunkiness~${\omega_0>0}$, such that
		    \begin{align*}
		    \|F(\nabla u_h^{\textit{c}})&-F(\nabla u)\|_{L^2(\Omega;\mathbb{R}^d)}^2\\&\leq  c\,  \inf_{v_h\in \mathcal{S}^{1,\textit{cr}}_D(\mathcal{T}_h)}\,\big[\|F(\nabla_{\!h}  v_h)-F(\nabla     u)\|_{L^2(\Omega;\mathbb{R}^d)}^2+\big\|\smash{h_{\mathcal{S}}^{1/2}}\jump{F(\nabla_{\!h}v_h)}\big\|_{L^2(\mathcal{S}_h^{i};\mathbb{R}^d)}^2\big]
		    \\&\leq  c\,  \inf_{v_h\in \mathcal{S}^{1,\textit{cr}}_D(\mathcal{T}_h)}\,\big[\|F(\nabla_{\!h}  v_h)-F(\nabla     u)\|_{L^2(\Omega;\mathbb{R}^d)}^2+
		        \textup{dist}_F^2(v_h,\mathcal{S}^{1}_D(\mathcal{T}_h),\mathcal{T}_h)
		       +\textup{osc}_h(f,v_h) \big]\,.
		    \end{align*}
	    \end{theorem}

	    \begin{remark}
	       If \hspace{-0.1mm}the \hspace{-0.1mm}natural \hspace{-0.1mm}regularity \hspace{-0.1mm}assumption \hspace{-0.1mm}\eqref{natural_regularity} \hspace{-0.1mm}is \hspace{-0.1mm}satisfied, 
        \hspace{-0.1mm}then, \hspace{-0.1mm}using \hspace{-0.1mm}that ${\jump{F(\nabla u)}\hspace{-0.1em}=\hspace{-0.1em}0}$ in $\mathcal{S}_h^{i}$ and the trace inequality \cite[Lemma A.16, (A.17)]{kr-phi-ldg}, for every $v_h\in \smash{\mathcal{S}^{1,\textit{cr}}_D(\mathcal{T}_h)}$, we find that
	       \begin{align}\label{rem:best-approxP1CR.1} 
	            \begin{aligned}
	            \big\|\smash{h_{\mathcal{S}}^{1/2}}\jump{F(\nabla_{\!h}v_h)}\big\|_{L^2(\mathcal{S}_h^{i};\mathbb{R}^d)}^2&=\big\|\smash{h_{\mathcal{S}}^{1/2}}\jump{F(\nabla_{\!h}v_h)-F(\nabla u)}\big\|_{L^2(\mathcal{S}_h^{i};\mathbb{R}^d)}^2\\&\leq 
	           c\,\|F(\nabla_{\!h}v_h)-F(\nabla u)\|_{L^2(\Omega;\mathbb{R}^d)}^2+c\,\|h_{\mathcal{T}}\nabla F(\nabla u)\|_{L^2(\Omega;\mathbb{R}^{d\times d})}^2\,,
	           \end{aligned}
	       \end{align}
	       so that the best-approximation result in  Theorem \ref{thm:best-approxP1CR} can be refined to
	       \begin{align}\label{rem:best-approxP1CR.2} 
	        \begin{aligned}
	            \|F(\nabla u_h^{\textit{c}})-F(\nabla u)\|_{L^2(\Omega;\mathbb{R}^d)}^2&\leq c\,  \inf_{v_h\in \mathcal{S}^{1,\textit{cr}}_D(\mathcal{T}_h)}\,\|F(\nabla_{\!h}  v_h)-F(\nabla u)\|_{L^2(\Omega;\mathbb{R}^d)}^2\\&\quad+c\,\|h_{\mathcal{T}}\nabla F(\nabla u)\|_{L^2(\Omega;\mathbb{R}^{d\times d})}^2\,.
	            \end{aligned}
	       \end{align}
	       In particular, \eqref{rem:best-approxP1CR.2} in conjunction with Theorem \ref{thm:best-approx} reveals that, under the natural~\mbox{regularity} assumption \eqref{natural_regularity}, 
        the performance of the $\mathcal{S}^1_D(\mathcal{T}_h)$-approximation \eqref{eq:pDirichletS1D} and the $\mathcal{S}^{1,\textit{cr}}_D(\mathcal{T}_h)$-approximation \eqref{eq:pDirichletS1crD} of \eqref{eq:pDirichletW1p} are comparable.
	    \end{remark}
		
		The main ingredient in the proof of Theorem \ref{thm:best-approxP1CR} is the following local efficiency result for \hspace{-0.15mm}the \hspace{-0.15mm}approximation \hspace{-0.15mm}error \hspace{-0.15mm}of \hspace{-0.15mm}the \hspace{-0.15mm}node-averaging \hspace{-0.15mm}quasi-interpolation \hspace{-0.15mm}operator \hspace{-0.15mm}${I_h^{\textit{av}}\colon\hspace{-0.17em}\mathcal{S}^{1,\textit{cr}}_D(\mathcal{T}_h)\hspace{-0.17em}\to\hspace{-0.17em} \mathcal{S}^1_D(\mathcal{T}_h)}$,~\hspace{-0.15mm}cf. \eqref{eq:a6},
	     \hspace{-0.15mm}with \hspace{-0.15mm}respect \hspace{-0.15mm}to \hspace{-0.15mm}Crouzeix--Raviart \hspace{-0.15mm}functions, \hspace{-0.15mm}measured \hspace{-0.15mm}in \hspace{-0.15mm}the \hspace{-0.15mm}natural \hspace{-0.15mm}distance,~\hspace{-0.15mm}cf.~\hspace{-0.15mm}Remark~\hspace{-0.15mm}\ref{rem:natural_dist}.
	
	\begin{lemma}\label{lem:best-approx-inv}
		        Let $\AAA\colon\mathbb{R}^d\to \mathbb{R}^d$ satisfy  Assumption \ref{assum:extra_stress} for $p\in (1,\infty)$ and  $\delta\ge 0$, and let $F\colon\mathbb{R}^d\to \mathbb{R}^d$ be defined by \eqref{eq:def_F} for the same $p\in (1,\infty)$     and  $\delta\ge 0$. Then, there exists a
		        constant $c>0$, depending only on the characteristics of $\AAA$~and~the~chunkiness~${\omega_0>0}$, such that  for every $v_h\in \smash{\mathcal{S}^{1,\textit{cr}}_D(\mathcal{T}_h)}$ and $T\in \mathcal{T}_h$, it holds
			    \begin{align*}
			    \|F(\nabla_{\! h} v_h)-F(\nabla I_h^{\textit{av}} v_h)\|_{L^2(T;\mathbb{R}^d)}^2&\leq c\,\textup{dist}_F^2(v_h,W^{1,p}_D(\Omega),\omega_T)
       +c\,\big\|\smash{h_{\mathcal{S}}^{1/2}}\jump{F(\nabla_{\!h}v_h)}\big\|_{L^2(\mathcal{S}_h^{i}(T);\mathbb{R}^d)}^2\,,
			    \end{align*} 
                where $\textup{dist}_F^2(v_h,W^{1,p}_D(\Omega),\omega_T)\coloneqq \inf_{v\in W^{1,p}_D(\Omega)}{\|F(\nabla_{\! h} v_h)-F(\nabla v)\|_{L^2(\omega_T;\mathbb{R}^d)}^2}$.
		    \end{lemma}
		    
		    \begin{proof}
		        Using \eqref{eq:hammera} and element-wise for every ${T\in \mathcal{T}_h}$, the Orlicz-approximation properties of $I_h^{\textit{av}}:\smash{\mathcal{S}^{1,\textit{cr}}_D(\mathcal{T}_h)}\to \mathcal{S}^1_D(\mathcal{T}_h)$, cf. \eqref{eq:a6},  with $\psi=\varphi$ and $a=\vert \nabla_{\! h} v_h(T)\vert$,~where~we,~for~every~${T\in \mathcal{T}_h}$,~write $\vert \nabla_{\! h} v_h(T)\vert$ \hspace{-0.1mm}to \hspace{-0.1mm}indicate \hspace{-0.1mm}that \hspace{-0.1mm}the \hspace{-0.1mm}shift \hspace{-0.1mm}on \hspace{-0.1mm}the \hspace{-0.1mm}patch \hspace{-0.1mm}$\omega_T$ \hspace{-0.1mm}depends \hspace{-0.1mm}on \hspace{-0.1mm}the \hspace{-0.1mm}value \hspace{-0.1mm}of \hspace{-0.1mm}$\vert \nabla_{\! h} v_h\vert$ \hspace{-0.1mm}on \hspace{-0.1mm}the~\hspace{-0.1mm}element~\hspace{-0.1mm}$T$,  $\|\jump{v_h}_S\|_{L^\infty(S)}\hspace{-0.17em}\leq\hspace{-0.17em} c\,\vert S\vert^{-1}\|\jump{v_h}_S\|_{L^1(S)}$ \hspace{-0.2mm}(cf.\ \hspace{-0.2mm}\cite[\hspace{-0.2mm}Lemma \hspace{-0.2mm}12.1]{EG21}), and~\hspace{-0.2mm}${\vert T\vert\hspace{-0.17em} \sim\hspace{-0.17em} h_T\vert S\vert}$~\hspace{-0.2mm}for~\hspace{-0.2mm}all~\hspace{-0.2mm}${T\hspace{-0.17em}\in\hspace{-0.17em} \mathcal{T}_h}$,~\hspace{-0.1mm}${S\hspace{-0.17em}\in\hspace{-0.17em}\mathcal{S}_h(T)}$, with a constant depending only on~the~\mbox{chunkiness}~${\omega_0>0}$, 
		        for~every~${T\in \mathcal{T}_h}$, we find that
		       \begin{align}\label{lem:best-approx-inv.1}
		            \begin{aligned}
		                \|F(\nabla_{\! h} v_h)-F(\nabla I_h^{\textit{av}} v_h)\|_{L^2(T;\mathbb{R}^d)}^2&\leq c\,\rho_{\varphi_{\vert \nabla_{\! h} v_h(T)\vert     },T}(\nabla_{\! h}v_h-\nabla I_h^{\textit{av}} v_h)
		                \\&\leq 
		                c\,\sum_{S\in \mathcal{S}_h(T)\setminus\Gamma_N}{h_T\,\vert S\vert \varphi_{\vert \nabla_{\! h} v_h(T)\vert}(h_T^{-1}\vert \jump{v_h}_S\vert)\,}
		                \\&\leq 
		                c\,\sum_{S\in \mathcal{S}_h(T)\setminus\Gamma_N}{\vert T\vert\, \varphi_{\vert \nabla_{\! h} v_h(T)\vert}(\vert T\vert^{-1}\| \jump{v_h}_S\|_{L^1(S)})}
		                \,.
		            \end{aligned}
		       \end{align}\newpage
		      \hspace{-5mm}Next, for all $S\hspace{-0.16em}\in \hspace{-0.16em}\mathcal{S}_h$, we denote by $\pi_h^S\hspace{-0.16em}:\hspace{-0.16em}L^1(S)\hspace{-0.16em}\to\hspace{-0.16em} \mathbb{R}$, the side-wise (local) $L^2$-projection~operator onto constant functions, for every $w\hspace{-0.1em}\in\hspace{-0.1em} L^1(S)$ defined by ${\pi_h^Sw\coloneqq \fint_S{w\,\mathrm{d}s}}$.~Since~for~every~${w\in W^{1,1}(T)}$, where $T\in \mathcal{T}_h$ with $T\subseteq \omega_S$, by the $L^1$-stability of $\pi_h^S\colon L^1(S)\to \mathbb{R}$ and~\cite[Corollary~A.19]{kr-phi-ldg},~it~holds
		      \begin{align*}
		          \|w-\pi_h^S w\|_{L^1(S)}&\leq \|w-\Pi_hw -\pi_h^S(w-\Pi_hw)\|_{L^1(S)}
		          \\&\leq 2\,\|w-\Pi_hw \|_{L^1(S)}
		          \\&
		          \leq c\,\|\nabla w\|_{L^1(T;\mathbb{R}^d)}\,,
		      \end{align*}
		      where $c>0$ depends only on the~chunkiness $ \omega_0>0 $.
                Next, let $v\in W^{1,p}_D(\Omega)$ be fixed,~but~arbitrary. Using~that~${\pi_h^S\jump{v_h}_S=\jump{v}_S=0}$~in~$L^1(S)$ for all $S\in \mathcal{S}_h(T)\setminus\Gamma_N$ and $T\in \mathcal{T}_h$, for every $T\in \mathcal{T}_h$ and $S\in \mathcal{S}_h(T)\setminus\Gamma_N$, we find that
		      \begin{align}\label{lem:best-approx-inv.2.0}
		            \begin{aligned}
		            \| \jump{v_h}_S\|_{L^1(S)}&=\| \jump{v_h-v}_S-\pi_h^S\jump{v_h-v}_S\|_{L^1(S)}
		                \\&
                  \leq \| \nabla_{\! h} v_h-\nabla v\|_{L^1(\omega_S;\mathbb{R}^d)}\,.
		                \end{aligned}
		      \end{align}
		      Using~in~\eqref{lem:best-approx-inv.1}, \eqref{lem:best-approx-inv.2.0}, $\vert T\vert\sim \vert \omega_S\vert\sim\vert \omega_T\vert$ for all $T\in \mathcal{T}_h$, $S\in \mathcal{S}_h(T)$, where $c>0$ depends only on the chunkiness $ \omega_0>0 $, Jensen's inequality, and Lemma \ref{lem:patch_to_element}, for every $T\in \mathcal{T}_h$, we~deduce~that 
		       \begin{align}\label{lem:best-approx-inv.3}
		            \begin{aligned}
		          \|F(\nabla I_h^{\textit{av}} v_h)-F(\nabla_{\! h} v_h)\|_{L^2(T;\mathbb{R}^d)}^2
		                &
		                \leq c\,\sum_{S\in \mathcal{S}_h(T)\setminus\Gamma_N}{\vert \omega_S\vert\,\varphi_{\vert \nabla_{\! h} v_h(T)\vert}
		                (\vert \omega_S\vert^{-1}\| \nabla_{\! h} v_h-\nabla v\|_{L^1(\omega_S;\mathbb{R}^d)})}
		                \\&\leq c\,\sum_{S\in \mathcal{S}_h(T)\setminus\Gamma_N}{\rho_{\varphi_{\vert \nabla_{\! h} v_h(T)\vert},\omega_S}( \nabla_{\! h} v_h-\nabla v)}
		                \\[-1mm]&\leq c\,\rho_{\varphi_{\vert \nabla_{\! h} v_h(T)\vert},\omega_T}( \nabla_{\! h} v_h-\nabla v)
		                \\&\leq  \rho_{\varphi_{\vert \nabla_{\! h} v_h\vert},\omega_T}( \nabla_{\! h} v_h-\nabla v)+c\,\smash{\big\|\smash{h_{\mathcal{S}}^{1/2}}\jump{F(\nabla_{\!h}v_h)}\big\|_{L^2(\mathcal{S}_h^{i}(T);\mathbb{R}^d)}^2}\,.
		               \end{aligned}\hspace{-10mm}
		            \end{align}
		        Eventually, using \eqref{eq:hammera} in \eqref{lem:best-approx-inv.3} and, subsequently, taking the infimum with respect~to~${v\hspace*{-0.1em}\in\hspace*{-0.1em} W^{1,p}_D(\Omega)}$, we conclude the assertion.\enlargethispage{7mm}
		    \end{proof}
	
	\begin{proof}[Proof (of Theorem \ref{thm:best-approxP1CR}).]
	For \hspace{-0.1mm}every \hspace{-0.1mm}$v_h\hspace{-0.1em}\in\hspace{-0.1em} \smash{\mathcal{S}^{1,\textit{cr}}_D(\mathcal{T}_h)}$, \hspace{-0.1mm}using \hspace{-0.1mm}Theorem \hspace{-0.1mm}\ref{P1_best-approx} \hspace{-0.1mm}and \hspace{-0.1mm}Lemma \hspace{-0.1mm}\ref{lem:best-approx-inv},~\hspace{-0.1mm}we~\hspace{-0.1mm}find~\hspace{-0.1mm}that
	     \begin{align*}
	        \begin{aligned}\|F(\nabla u_h^{\textit{c}})\hspace{-0.1em}-\hspace{-0.1em}F(\nabla u)\|_{L^2(\Omega;\mathbb{R}^d)}^2&\leq c\,\|F(\nabla I_h^{\textit{av}} v_h)\hspace{-0.1em}-\hspace{-0.1em}F(\nabla     u)\|_{L^2(\Omega;\mathbb{R}^d)}^2
	        \\& \leq  c\,\|F(\nabla_{\!h} v_h)\hspace{-0.1em}-\hspace{-0.1em}F(\nabla     u)\|_{L^2(\Omega;\mathbb{R}^d)}^2
	        +c\,\|F(\nabla_{\! h}  v_h)\hspace{-0.1em}-\hspace{-0.1em}F(\nabla I_h^{\textit{av}}v_h)\|_{L^2(\Omega;\mathbb{R}^d)}^2
	        \\& \leq  c\,\|F(\nabla_{\!h} v_h)\hspace{-0.1em}-\hspace{-0.1em}F(\nabla     u)\|_{L^2(\Omega;\mathbb{R}^d)}^2
	        +c\,\smash{\big\|\smash{h_{\mathcal{S}}^{1/2}}\jump{F(\nabla_{\!h}v_h)}\big\|_{L^2(\mathcal{S}_h^{i};\mathbb{R}^d)}^2}\,.
	       \end{aligned}
	    \end{align*} 
	    Eventually, taking in 
	    the infimum with respect to  $v_h\in \smash{\mathcal{S}^{1,\textit{cr}}_D(\mathcal{T}_h)}$, we conclude the first claimed estimate. The second claimed estimate follows from the first  and Corollary \ref{cor:efficiency}.
	\end{proof}
	
	\begin{corollary}\label{cor:best-approxCRCR} 
            Let $\AAA\colon\mathbb{R}^d\to \mathbb{R}^d$ satisfy  Assumption \ref{assum:extra_stress} for $p\in (1,\infty)$ and  $\delta\ge 0$,~and~let~$F\colon\mathbb{R}^d\to \mathbb{R}^d$ be defined by \eqref{eq:def_F} for the same $p\in (1,\infty)$     and  $\delta\ge 0$. Then, there exists a
		        constant $c>0$, depending only on the characteristics of $\AAA$~and~the~chunkiness~${\omega_0>0}$, such that
		    \begin{align*}
		    &\|F(\nabla_{\!h} u_h^{\textit{cr}})-F(\nabla u)\|_{L^2(\Omega;\mathbb{R}^d)}^2
		    \\&\leq  c\,  \inf_{v_h\in \mathcal{S}^{1,\textit{cr}}_D(\mathcal{T}_h)}{\big[\|F(\nabla_{\!h}  v_h)-F(\nabla     u)\|_{L^2(\Omega;\mathbb{R}^d)}^2 +\smash{\big\|\smash{h_{\mathcal{S}}^{1/2}}\jump{F(\nabla_{\!h}v_h)}\big\|_{L^2(\mathcal{S}_h^{i};\mathbb{R}^d)}^2}+\mathrm{osc}_h(f,v_h)\big]}\,.
		    \\&\leq  c\,\inf_{v_h\in \mathcal{S}^{1,\textit{cr}}_D(\mathcal{T}_h)}{\big[\|F(\nabla_{\!h}  v_h)-F(\nabla     u)\|_{L^2(\Omega;\mathbb{R}^d)}^2+
		        \textup{dist}_F^2(v_h,\mathcal{S}^{1}_D(\mathcal{T}_h),\mathcal{T}_h)
		       +\textup{osc}_h(f,v_h) \big]}\,.
		    \end{align*}
    \end{corollary}
		    
    \begin{proof}
        Immediate consequence of Theorem \ref{thm:best-approx} in conjunction with Theorem \ref{thm:best-approxP1CR}.
	\end{proof}
	
	   \begin{remark}
	   If  the natural regularity assumption \eqref{natural_regularity} is satisfied, then, using \eqref{rem:best-approxP1CR.1}, 
	        the best-approximation result in  Corollary \ref{cor:best-approxCRCR} can be refined to
	       \begin{align*}
	        \begin{aligned}
	            \|F(\nabla u_h^{\textit{cr}})-F(\nabla u)\|_{L^2(\Omega;\mathbb{R}^d)}^2&\leq c\,  \inf_{v_h\in \mathcal{S}^{1,\textit{cr}}_D(\mathcal{T}_h)}{\big[\|F(\nabla_{\!h}  v_h)-F(\nabla u)\|_{L^2(\Omega;\mathbb{R}^d)}^2+\textup{osc}_h(f,v_h)\big]}\\&\quad+c\,\smash{\|h_{\mathcal{T}}\nabla F(\nabla u)\|_{L^2(\Omega;\mathbb{R}^{d\times d})}^2}\,.
	            \end{aligned}
	       \end{align*}
	    \end{remark}

    \section{A priori error analysis}\label{sec:a_priori}

        \qquad A further consequence of the medius error analysis (cf.\ Theorem \ref{thm:best-approx}) in Section \ref{sec:medius}~is~the~insight that the distance~of $u_h^{\textit{cr}}\hspace{-0.15em}\in\hspace{-0.15em} \smash{\smash{\mathcal{S}^{1,\textit{cr}}_D(\mathcal{T}_h)}}$ to $u\hspace{-0.15em}\in\hspace{-0.15em}\smash{ W^{1,p}_D(\Omega)}$, up to oscillation terms, is bounded by the distance of $u_h^c\hspace{-0.1em}\in\hspace{-0.1em} \mathcal{S}^1_D(\mathcal{T}_h)$ to $u\hspace{-0.1em}\in\hspace{-0.1em} \smash{W^{1,p}_D(\Omega)}$,  each measured in the~natural~distance, cf.\ Remark \ref{rem:natural_dist}. As a result, the approximation rate result in Theorem \ref{P1_apriori} for the $\mathcal{S}^1_D(\mathcal{T}_h)$-approximation \eqref{eq:pDirichletS1D} of \eqref{eq:pDirichletW1p} inherits to the $\smash{\mathcal{S}^{1,\textit{cr}}_D(\mathcal{T}_h)}$-approximation \eqref{eq:pDirichletS1crD}. 
        \enlargethispage{3mm}

		\begin{theorem}\label{thm:rate_u}
		   Let $\AAA\colon\mathbb{R}^d\to \mathbb{R}^d$ satisfy  Assumption \ref{assum:extra_stress} for $p\in(1,\infty)$ and  $\delta\ge 0$. Moreover, let $\varphi\colon \mathbb{R}_{\ge 0}\to \mathbb{R}_{\ge 0}$ be defined by \eqref{eq:def_phi} and let $F\colon\mathbb{R}^d\to \mathbb{R}^d$ be defined by \eqref{eq:def_F}, each for the same $p\in (1,\infty)$ and  $\delta\ge 0$. If \eqref{natural_regularity} is satisfied, then,
		    there exists a constant $c>0$, depending on the characteristics of $\AAA$ and the chunkiness $\omega_0>0$, such that, setting $h_{\max}\coloneqq \max_{T\in \mathcal{T}_h}{h_T}$, it holds
		  \begin{align*}
		    \|F(\nabla_{\!h} u_h^{\textit{cr}})-F(\nabla u)\|_{L^2(\Omega;\mathbb{R}^d)}^2&\leq  c\,\|h_{\mathcal{T}}\nabla F(\nabla u)\|_{L^2(\Omega;\mathbb{R}^{d\times d})}^2+c\,\rho_{(\varphi_{\vert \nabla u\vert})^*,\Omega}(h_{\mathcal{T}}f )
		    \\&\leq c\,\|h_{\mathcal{T}}\nabla F(\nabla u)\|_{L^2(\Omega;\mathbb{R}^{d\times d})}^2+c\, h_{\max}^2\,\big(\rho_{\varphi^*,\Omega}(f)  +\rho_{\varphi,\Omega}(\nabla u)\big) \,.
		    \end{align*}
		\end{theorem}
		
		\begin{proof}\let\qed\relax
		         Using the convexity of $(\varphi_{\vert \nabla_{\!h} v_h(x)\vert})^*\colon\mathbb{R}_{\ge 0}\to \mathbb{R}_{\ge 0}$~for~a.e.\ $x\in \Omega$, $\sup_{a\ge 0}{\Delta_2((\varphi_a)^*)}<\infty$, the Orlicz-stability of $\Pi_h$ (cf. \cite[Corollary A.8, (A.12)]{kr-phi-ldg}), and the shift change \eqref{lem:shift-change.3},~for~every~$v_h\in\mathcal{S}^1_D(\mathcal{T}_h)$, we find that
		        \begin{align}\label{cor:rate_improve.1}
		            \begin{aligned}
		            \textrm{osc}_h(f,v_h)
		           & \leq c\,\smash{\rho_{(\varphi_{\vert \nabla v_h\vert})^*,\Omega}(h_{\mathcal{T}} f )}
		            \\&\leq c\,\smash{\rho_{(\varphi_{\vert \nabla u\vert})^*,\Omega}(h_{\mathcal{T}}f )}+c\,\|F(\nabla v_h)-F(\nabla u)\|_{L^2(\Omega;\mathbb{R}^d)}^2\,.
		            \end{aligned}
		        \end{align}
    		    Using \eqref{cor:rate_improve.1} in Theorem \ref{thm:best-approx}, for every $v_h\in\mathcal{S}^1_D(\mathcal{T}_h)$, we deduce that
                \begin{align}
		             \|F(\nabla_{\!h} u_h^{\textit{cr}})-F(\nabla u)\|_{L^2(\Omega;\mathbb{R}^d)}^2&\leq  c \,\|F(\nabla v_h)-F(\nabla u)\|_{L^2(\Omega;\mathbb{R}^d)}^2+ c\,\rho_{(\varphi_{\vert \nabla u\vert})^*,\Omega}(h_{\mathcal{T}}f )\,.\label{cor:rate_improve.2}
		        \end{align}
                Choosing $v_h=u_h^{\textit{c}}\in \mathcal{S}^1_D(\mathcal{T}_h)$ in \eqref{cor:rate_improve.2} and resorting to Theorem \ref{P1_apriori}, we arrive at
                \begin{align}\label{cor:rate_improve.3}
                    \|F(\nabla_{\!h} u_h^{\textit{cr}})-F(\nabla u)\|_{L^2(\Omega;\mathbb{R}^d)}^2&\leq  c \,\|h_{\mathcal{T}}\nabla F(\nabla u)\|_{L^2(\Omega;\mathbb{R}^{d\times d})}^2+ c\,\rho_{(\varphi_{\vert \nabla u\vert})^*,\Omega}(h_{\mathcal{T}}f )\,,
                \end{align}
                which is the first claimed a priori error estimate. For the second claimed a priori~error~estimate, we need to distinguish between the cases $p\in (1,2)$ and $p\in [2,\infty)$:

                \textit{Case $p\in (1,2)$.} If $p\in (1,2)$, then, there holds the elementary inequality
		    \begin{align*}
		     \smash{(\varphi_{\vert a\vert })^*(h\,t)}&\leq \smash{c\,\smash{ h^2}\big(\varphi^*(t)+\varphi(\vert a\vert)\big)}&&\quad\textup{ for all }a\in \mathbb{R}^d\,,\; t\ge 0\,,\; h\in [0,1]\,,
		         \end{align*}
		         which follows from the definition of  shifted \mbox{$N$-functions}, cf. \eqref{eq:phi_shifted}, and  the shift change~\eqref{lem:shift-change.3}~(i.e., with $b = 0$ and using that $\vert F(a)\vert^2=\varphi(\vert a\vert )$ for all $a\in \mathbb{R}^d$), so that
		        \begin{align}\label{cor:rate_improve.4}
		           \smash{  \rho_{(\varphi_{\vert \nabla u\vert})^*,\Omega}(h_{\mathcal{T}}f )}\leq c\, \smash{h_{\max}^2}\,\big(\rho_{\varphi^*,\Omega}(f)  +\rho_{\varphi,\Omega}(\nabla u)\big)\,.
		        \end{align}

                \textit{Case $p\in [2,\infty)$.} Inasmuch as the flux $z\hspace*{-0.1em}\coloneqq \hspace*{-0.1em}\AAA(\nabla u)\hspace*{-0.1em}\in\hspace*{-0.1em} W^{p'}_N(\textup{div};\Omega)$ 
                satisfies $\textup{div}\,z\hspace*{-0.1em}=\hspace*{-0.1em}-f$~in~$L^{p'}(\Omega)$
                and    \hspace{-0.1mm}$(\varphi^*)_{\vert \AAA(a)\vert }\hspace{-0.1em}\sim\hspace{-0.1em} (\varphi_{\vert a\vert })^*$ \hspace{-0.1mm}uniformly~\hspace{-0.1mm}in~\hspace{-0.1mm}${a\hspace{-0.1em}\in \hspace{-0.1em}\mathbb{R}^d}$ \hspace{-0.1mm}(cf.\ \hspace{-0.1mm}\cite[Lemma \hspace{-0.1mm}26]{die-ett}~\hspace{-0.1mm}with~\hspace{-0.1mm}$\vert \AAA(a)\vert\hspace{-0.1em}=\hspace{-0.1em}\varphi'(\vert a\vert)$~\hspace{-0.1mm}for~\hspace{-0.1mm}all~\hspace{-0.1mm}$ {a\hspace{-0.1em}\in\hspace{-0.1em} \mathbb{R}^d }$), we have that\vspace{-1mm}
		    \begin{align}\label{cor:rate_improve.5}
		        \smash{\rho_{(\varphi_{\smash{\vert \nabla  u\vert}})^*,\Omega}(h_{\mathcal{T}}f)\leq c\, \rho_{(\varphi^*)_{\vert z\vert},\Omega}(h_{\mathcal{T}}\nabla z)}\,.
		    \end{align}
             Since $p\ge  2$, i.e., $p'\leq 2$, and, thus, $(\varphi^*)_{\vert a\vert}(h\,t)\leq c\,h^2\,(\delta^{p-1}+\vert a\vert)^{p'-2}t^2$ for all $a\in \mathbb{R}^d$~and~${t,h\ge 0}$, 
		         we have that\vspace{-1mm}
		         \begin{align}\label{cor:rate_improve.6}
		             \smash{\rho_{(\varphi^*)_{\vert z\vert},\Omega}(h_{\mathcal{T}} \nabla z)\leq c\,\|h_{\mathcal{T}}(\delta^{p-1}+\vert z\vert)^{\smash{\frac{p'-2}{2}}} \nabla z\|_{L^2(\Omega;\mathbb{R}^{d\times d})}^2\,.}
		         \end{align}
		         In addition, due to Lemma \ref{lem:reg_equiv}, Lemma \ref{lem:reg_dual_source}, and Lemma \ref{lem:reg_dual}, we have that $z\in W^{1,\smash{p'}}(\Omega;\mathbb{R}^d)$~with
		         \begin{align}\label{cor:rate_improve.7}
                    \begin{aligned}
		            \hspace*{-0.5mm} \smash{\|h_{\mathcal{T}}(\delta^{p-1}+\vert z\vert)^{\smash{\frac{p'-2}{2}}} \nabla z\|_{L^2(\Omega;\mathbb{R}^{d\times d})}}\hspace*{-0.1em}\leq\hspace*{-0.1em} c\,\|h_{\mathcal{T}}\nabla F^*(z) \|^2_{L^2(\Omega;\mathbb{R}^{d\times d})}\hspace*{-0.1em}\leq\hspace*{-0.1em} c\,\|h_{\mathcal{T}}\nabla F(\nabla u) \|^2_{L^2(\Omega;\mathbb{R}^{d\times d})}\,.
               \end{aligned}
		         \end{align}
		         Combining \eqref{cor:rate_improve.6} and \eqref{cor:rate_improve.7} in \eqref{cor:rate_improve.5}, we deduce that
           \begin{align}\label{cor:rate_improve.8}
               \smash{\rho_{(\varphi_{\smash{\vert \nabla  u\vert}})^*,\Omega}(h_{\mathcal{T}}f)\leq c\,\|h_{\mathcal{T}}\nabla F(\nabla u) \|^2_{L^2(\Omega;\mathbb{R}^{d\times d})}\,.}
           \end{align}
           As \hspace{-0.1mm}a \hspace{-0.1mm}whole, \hspace{-0.15mm}using \hspace{-0.1mm}\eqref{cor:rate_improve.4} \hspace{-0.1mm}and \hspace{-0.1mm}\eqref{cor:rate_improve.8} \hspace{-0.1mm}in \hspace{-0.1mm}\eqref{cor:rate_improve.3}, \hspace{-0.1mm}we \hspace{-0.1mm}conclude \hspace{-0.1mm}the \hspace{-0.1mm}second \hspace{-0.1mm}claimed \hspace{-0.1mm}a \hspace{-0.1mm}priori~\hspace{-0.1mm}error~\hspace{-0.1mm}estimate.~~~$\qedsymbol$
		\end{proof}

    In the case \eqref{special_case}, i.e., \eqref{eq:pDirichletW1p} and \eqref{eq:pDirichletS1crD} admit equivalent convex minimization problems and we have access to the (discrete) convex duality theory from Subsection \ref{subsec:convex_duality} and~Subsection~\ref{subsec:discrete_convex_duality}, resorting to 
    the (discrete) convex optimality relations \eqref{eq:pDirichletOptimality2} and  \eqref{eq:pDirichletOptimalityCR2}, as well as the fact that the discrete dual solution is uniquely determined by the generalized Marini formula, cf. \eqref{eq:gen_marini},  we are in the position to  derive from Corollary \ref{thm:rate_u} an a priori error estimate for the dual solution and the discrete dual solution, measured in the conjugate natural distance, cf. Remark \ref{rem:conjugate_natural_dist}.
		
		\begin{lemma}\label{lem:rate_z} 
		Let $\AAA\colon\hspace*{-0.1em}\mathbb{R}^d\hspace*{-0.1em}\to\hspace*{-0.1em} \mathbb{R}^d$ be defined by \eqref{special_case} for $p\hspace*{-0.1em}\in\hspace*{-0.1em} (1,\infty)$ and $\delta\hspace*{-0.1em}\ge\hspace*{-0.1em} 0$ and~let~${F,F^*\colon\hspace*{-0.1em}\mathbb{R}^d\hspace*{-0.1em}\to\hspace*{-0.1em} \mathbb{R}^d}$ be defined by \eqref{eq:def_F}  for the same $p\in (1,\infty)$ and $\delta\ge 0$.  
		Then, there exists~a~constant~${c>0}$, depending on $p\in (1,\infty)$, $d\in \mathbb{N}$, $\delta\ge     0$, and the chunkiness $\omega_0>0$, such that
			\begin{align*}
				\|F^*(z_h^{\textit{rt}})\hspace*{-0.1em}-\hspace*{-0.1em}F^*(z)\|^2_{L^2(\Omega;\mathbb{R}^d)}&\leq c\,\|F(\nabla_{\!h} u_h^{\textit{cr}})\hspace*{-0.1em}-\hspace*{-0.1em}F(\nabla u)\|^2_{L^2(\Omega;\mathbb{R}^d)}\hspace*{-0.1em}+\hspace*{-0.1em}
				c\,\rho_{(\varphi_{\smash{\vert \nabla  u\vert}})^*,\Omega}(h_{\mathcal{T}}f)
				\,.
			\end{align*}
		\end{lemma}
		
		\begin{proof}
		    Using the discrete convex optimality relations \eqref{eq:pDirichletOptimality2} and \eqref{eq:pDirichletOptimalityCR2}, two equivalences in \eqref{eq:hammera}, the generalized Marini formula \eqref{eq:gen_marini}, again, the discrete convex optimality relations \eqref{eq:pDirichletOptimalityCR2}, and the Orlicz-stability of $\Pi_h$ (cf. \cite[Corollary A.8, (A.12)]{kr-phi-ldg}), we find that
		    \begin{align}\label{cor:convex_rate_z.1} 
		            \begin{aligned}
		        	\|F^*(z_h^{\textit{rt}})-F^*(z)\|^2_{L^2(\Omega;\mathbb{R}^d)}&\leq 2\, \|F^*(\Pi_h z_h^{\textit{rt}})-F^*(z)\|^2_{L^2(\Omega;\mathbb{R}^d)}\\&\quad
		        	+2\,\|F^*( z_h^{\textit{rt}})-F^*(\Pi_h z_h^{\textit{rt}})\|^2_{L^2(\Omega;\mathbb{R}^d)}
		        	\\&
		        	\leq 2\, \|F^*(\AAA(\nabla_{\!h} u_h^{\textit{cr}}))-F^*(\AAA(\nabla u))\|^2_{L^2(\Omega;\mathbb{R}^d)}\\&\quad+c\,\rho_{(\varphi^*)_{\smash{\vert \Pi_h z_h^{\textit{rt}}\vert}},\Omega}(z_h^{\textit{rt}}-\Pi_h z_h^{\textit{rt}}) 
		        	\\&
		        	\leq c\, \|F(\nabla_{\!h} u_h^{\textit{cr}})-F(\nabla u)\|^2_{L^2(\Omega;\mathbb{R}^d)}+c\,\rho_{(\varphi^*)_{\smash{\vert \AAA(\nabla_{\! h }u_h^{\textit{cr}})\vert}},\Omega}(h_{\mathcal{T}}f_h)
		        	\\&
		        	\leq c\, \|F(\nabla_{\!h} u_h^{\textit{cr}})-F(\nabla u)\|^2_{L^2(\Omega;\mathbb{R}^d)}+c\,\rho_{(\varphi^*)_{\smash{\vert \AAA(\nabla_{\! h }u_h^{\textit{cr}})\vert}},\Omega}(h_{\mathcal{T}}f)\,.
		        	 \end{aligned}
		    \end{align}
		    Using that $(\varphi^*)_{\vert \AAA(a)\vert }\sim (\varphi_{\vert a\vert })^*$ uniformly in $a\in \mathbb{R}^d$ (cf. \cite[Lemma 26]{die-ett}~with~$\vert \AAA(a)\vert=\varphi'(\vert a\vert)$ for all $a\in \mathbb{R}^d$) and the shift~change~\eqref{lem:shift-change.3}, we observe that
		    \begin{align}\label{cor:convex_rate_z.2} 
		        \begin{aligned}
		            \rho_{(\varphi^*)_{\smash{\vert \AAA(\nabla_{\! h }u_h^{\textit{cr}})\vert}},\Omega}(h_{\mathcal{T}}f)&\leq c\,\rho_{(\varphi_{\smash{\vert \nabla_{\!h} u_h^{\textit{cr}}\vert}})^*,\Omega}(h_{\mathcal{T}}f)
		            \\&\leq c\,\rho_{(\varphi_{\smash{\vert \nabla  u\vert}})^*,\Omega}(h_{\mathcal{T}}f)+c\,\|F(\nabla_{\!h} u_h^{\textit{cr}})-F(\nabla u)\|^2_{L^2(\Omega;\mathbb{R}^d)}\,.
		        \end{aligned}
		    \end{align}
		    Combining \eqref{cor:convex_rate_z.1} and \eqref{cor:convex_rate_z.2}, we deduce that
		    \begin{align*}
		        \|F^*(z_h^{\textit{rt}})-F^*(z)\|^2_{L^2(\Omega;\mathbb{R}^d)}\leq c\, \|F(\nabla_{\!h} u_h^{\textit{cr}})-F(\nabla u)\|^2_{L^2(\Omega;\mathbb{R}^d)}+c\,\rho_{(\varphi_{\smash{\vert \nabla  u\vert}})^*,\Omega}(h_{\mathcal{T}}f)\,,
		   \end{align*}
            which is the  claimed a priori error estimate. 
		\end{proof}
		
		\begin{theorem}\label{thm:rate_z} Let $\AAA\colon\hspace*{-0.1em}\mathbb{R}^d\hspace*{-0.1em}\to\hspace*{-0.1em} \mathbb{R}^d$ be defined by \eqref{special_case} for $p\hspace*{-0.1em}\in\hspace*{-0.1em} (1,\infty)$ and $\delta\hspace*{-0.1em}\ge\hspace*{-0.1em} 0$ and~let~${F,F^*\colon\hspace*{-0.1em}\mathbb{R}^d\hspace*{-0.1em}\to\hspace*{-0.1em} \mathbb{R}^d}$ be defined by \eqref{eq:def_F}  for the same $p\in (1,\infty)$ and $\delta\ge 0$. 
		If  \eqref{natural_regularity} is satisfied, then, there exists a constant $c>0$, depending on $p\in (1,\infty)$, $d\in \mathbb{N}$, $\delta\ge     0$, and the chunkiness $\omega_0>0$,~such~that
		\begin{align*}
		    	\|F^*(z_h^{\textit{rt}})-F^*(z)\|^2_{L^2(\Omega;\mathbb{R}^d)}&\leq c\,\|h_{\mathcal{T}}\nabla F(\nabla u)\|_{L^2(\Omega;\mathbb{R}^{d\times d})}^2+
				c\,\rho_{(\varphi_{\vert \nabla u\vert})^*,\Omega}(h_{\mathcal{T}}f)
				\\&\leq c\,\|h_{\mathcal{T}}\nabla F(\nabla u)\|_{L^2(\Omega;\mathbb{R}^{d\times d})}^2+c\,h_{\max}^2\,\big(\rho_{\varphi^*,\Omega}(f)+\rho_{\varphi,\Omega}(\nabla u)\big)
				\,.
		\end{align*}
		\end{theorem}
		
		\begin{proof}
		    The first claimed a priori error estimate follows from Lemma \ref{lem:rate_z} and Theorem \ref{thm:rate_u}. Then, the second claimed a priori error estimate follows from the first in  conjunction with~\eqref{cor:rate_improve.4}~and~\eqref{cor:rate_improve.8}.
		\end{proof}
		
		\begin{remark}
		\begin{itemize}[noitemsep,topsep=0.0pt,labelwidth=\widthof{\textit{(ii)}},leftmargin=!]
		    \item[(i)] In the particular case $p\in  [2,\infty)$, due to \eqref{cor:rate_improve.8}, the second term~on~the~right-hand side in Theorem \ref{thm:rate_u} and Theorem \ref{thm:rate_z} can be omitted, so that we arrive at 
		        \begin{align*}
		            \|F(\nabla_{\!h} u_h^{\textit{cr}})-F(\nabla u)\|_{L^2(\Omega;\mathbb{R}^d)}^2+\|F^*(z_h^{\textit{rt}})-F^*(z)\|^2_{L^2(\Omega;\mathbb{R}^d)}\leq c\,\|h_{\mathcal{T}}\nabla F(\nabla u)\|_{L^2(\Omega;\mathbb{R}^{d\times d})}^2\,,
		        \end{align*}
		        which assumes a very similar form to the a priori error estimate in Theorem \ref{P1_apriori}.
		        
		     \item[(ii)] The a priori error estimates in Theorem \ref{thm:rate_u} and Theorem \ref{thm:rate_z} are optimal for all $p\in (1,\infty)$ and $\delta\ge 0$. This is confirmed via numerical experiments, cf. Subsection \ref{subsec:num_a_priori}.
		\end{itemize}
		        
		\end{remark}
	
    \section{A posteriori error analysis}\label{sec:a_posteriori}
	 
	 \qquad In this section, we examine a primal-dual a posteriori error estimator for the $p$-Dirichlet problem, derived in  \cite{BKAFEM}, for  reliability and efficiency. Here, reliability is an immediate~consequence of convex duality relations, while efficiency is based on Corollary \ref{cor:best-approx}.
	 
	 To begin with, in analogy with \cite[Proposition 5.1]{BKAFEM}, we introduce the
        \textit{primal-dual a posteriori error estimator} $\eta_h^2\colon\mathcal{S}^1_D(\mathcal{T}_h)\to \mathbb{R}_{\ge 0}$, for every $v_h\in \mathcal{S}^1_D(\mathcal{T}_h)$ defined by 
	 \begin{align*}
    \begin{aligned}
	     \eta_h^2(v_h)&\coloneqq \rho_{\varphi,\Omega}(\nabla v_h)-(\Pi_h \smash{z_h^{\textit{rt}}},\nabla v_h-\nabla_{\!h} u_h^{\textit{cr}})_{\Omega}-\rho_{\varphi,\Omega}(\nabla_{\!h} u_h^{\textit{cr}})
      \\&\quad +\rho_{\varphi^*,\Omega}(\smash{z_h^{\textit{rt}}})-\rho_{\varphi^*,\Omega}(\Pi_h \smash{z_h^{\textit{rt}}})\,.
      \end{aligned}
	 \end{align*}

    \begin{remark} The primal-dual a posteriori error estimator can be decomposed into two parts:
    \begin{itemize}[noitemsep,topsep=0.0pt,labelwidth=\widthof{\textit{(ii)}},leftmargin=!]
        \item[(i)] The discrete residual part $\eta_{A,h}^2\colon\mathcal{S}^1_D(\mathcal{T}_h)\to \mathbb{R}_{\ge 0}$, for every $v_h\in \mathcal{S}^1_D(\mathcal{T}_h)$ defined by 
        \begin{align*}
            \smash{\eta_{A,h}^2(v_h)\coloneqq \rho_{\varphi,\Omega}(\nabla v_h)-(\Pi_h \smash{z_h^{\textit{rt}}},\nabla v_h-\nabla_{\!h} u_h^{\textit{cr}})_{\Omega}-\rho_{\varphi,\Omega}(\nabla_{\!h} u_h^{\textit{cr}})\,,}
        \end{align*}
        measures how well $v_h\in \mathcal{S}^1_D(\mathcal{T}_h)$ satisfies the discrete convex duality relation \eqref{eq:pDirichletOptimalityCR1.2}~(or~\eqref{eq:pDirichletOptimalityCR2}).

        \item[(ii)] The data approximation part $\eta_{B,h}^2\in \mathbb{R}$, 
        defined by 
        \begin{align*}
            \smash{\eta_{B,h}^2
            \coloneqq \rho_{\varphi^*,\Omega}(\smash{z_h^{\textit{rt}}})-\rho_{\varphi^*,\Omega}(\Pi_h \smash{z_h^{\textit{rt}}})\,,}
        \end{align*}
        measures the error resulting from the replacement of $f\in L^{p'}(\Omega)$ by $f_h\in \mathcal{L}^0(\mathcal{T}_h)$ in \eqref{eq:pDirichletS1crD}.
    \end{itemize}
    \end{remark}

    The following reliability result applies.
	 
	 \begin{theorem}[Reliability]\label{thm:rel}
	 Let $\AAA\colon\mathbb{R}^d\to \mathbb{R}^d$ be defined by \eqref{special_case} for $p\in (1,\infty)$ and $\delta\ge 0$ and let $F\colon\mathbb{R}^d\to \mathbb{R}^d$ be defined by \eqref{eq:def_F} for the same $p\in (1,\infty)$ and $\delta\ge 0$.  If $f=f_h\in \mathcal{L}^0(\mathcal{T}_h)$, then, there exists a constant $c>0$, depending only on $p\in (1,\infty)$ and $\delta\ge 0$, such that  for every $v_h\in      \mathcal{S}^1_D(\mathcal{T}_h)$, it holds\vspace{-0.5mm}
	       \begin{align*}
	           \smash{\|F(\nabla v_h)-F(\nabla u)\|^2_{L^2(\Omega;\mathbb{R}^d)}\leq c\,
	               \eta^2_h(v_h)\,.}
	       \end{align*}
	\end{theorem}

 \begin{proof} In principle, the proof is already included as a special case in \cite[Proposition 3.1 \& Proposition 4.1]{BKAFEM}. For the benefit of the reader, however, the proof briefly reproduced here.

      Using the co-coercivity property of $I\colon W^{1,p}_D(\Omega)\to \mathbb{R}$ at the primal solution $u\in W^{1,p}_D(\Omega)$, i.e., that there exists a constant $c\hspace*{-0.1em}>\hspace*{-0.1em}0$, depending only on $p\hspace*{-0.1em}\in \hspace*{-0.1em}(1,\infty)$ and $\delta\hspace*{-0.1em}\ge\hspace*{-0.1em} 0$, cf. \cite[Theorem~8~(ii)]{KZ22}, such that for every $v\in W^{1,p}_D(\Omega)$, it holds
      \begin{align*}
          \smash{c^{-1}\|F(\nabla v)-F(\nabla u)\|^2_{L^2(\Omega;\mathbb{R}^d)}\leq I(v)-I(u)\leq c\,\|F(\nabla v)-F(\nabla u)\|^2_{L^2(\Omega;\mathbb{R}^d)}}\,,
      \end{align*} the strong duality relation, i.e., $I(u)=D(z)$, that $\smash{z_h^{\textit{rt}}\in W^{p'}_N(\textup{div};\Omega)}$ with $-\textup{div}\,\smash{z_h^{\textit{rt}}}=f_h=f$ in $\mathcal{L}^0(\mathcal{T}_h)$ (cf. \eqref{eq:pDirichletOptimalityCR1.1}), the discrete convex optimality relation \eqref{eq:pDirichletOptimalityCR1.2}, the discrete integration-by-parts formula \eqref{eq:pi0}, for every $v_h\in \mathcal{S}^1_D(\mathcal{T}_h)$,~we~find~that
	     \begin{align}\label{thm:eff}
	        \begin{aligned}
	            \smash{c^{-1}\|F(\nabla v_h)-F(\nabla u)\|^2_{L^2(\Omega;\mathbb{R}^d)}}&\leq I(v_h)-I(u)
	            \\&\leq I(v_h)-D(\smash{z_h^{\textit{rt}}})
	            \\&=\rho_{\varphi,\Omega}(\nabla v_h)-(f,v_h)_\Omega+\rho_{\varphi^*,\Omega}( \smash{z_h^{\textit{rt}}})
	            	         \\&=\rho_{\varphi,\Omega}(\nabla v_h )-(f_h,\Pi_hv_h)_\Omega+\rho_{\varphi^*,\Omega}(\Pi_h \smash{z_h^{\textit{rt}}})
	         \\&\quad +\rho_{\varphi^*,\Omega}(\smash{z_h^{\textit{rt}}} )-\rho_{\varphi^*,\Omega}( \Pi_h \smash{z_h^{\textit{rt}}} )
          \\&=\rho_{\varphi,\Omega}(\nabla v_h)+( \mathrm{div}\, \smash{z_h^{\textit{rt}}},\Pi_h v_h)_\Omega\\&\quad+( \Pi_h \smash{z_h^{\textit{rt}}},\nabla_{\! h} u_h^{\textit{cr}})_\Omega
          -\rho_{\varphi,\Omega}(\nabla_{\! h} u_h^{\textit{cr}})
          \\&\quad +\rho_{\varphi^*,\Omega}(\smash{z_h^{\textit{rt}}} )-\rho_{\varphi^*,\Omega}( \Pi_h \smash{z_h^{\textit{rt}}} )
	          \\&=\rho_{\varphi,\Omega}(\nabla v_h)-( \Pi_h \smash{z_h^{\textit{rt}}},\nabla v_h-\nabla_{\! h} u_h^{\textit{cr}})_\Omega-\rho_{\varphi,\Omega}(\nabla_{\! h} u_h^{\textit{cr}})
	         \\&\quad +\rho_{\varphi^*,\Omega}(\smash{z_h^{\textit{rt}}} )-\rho_{\varphi^*,\Omega}( \Pi_h \smash{z_h^{\textit{rt}}})\,,
	        \end{aligned}
    \end{align}
       which is the claimed reliability of $\eta_h^2\colon\mathcal{S}^1_D(\mathcal{T}_h)\to \mathbb{R}_{\ge 0}$.
    \end{proof}

    Key ingredient in the verification of the efficiency of the 
    primal-dual a posteriori~error~estimator $\eta_h^2\colon\mathcal{S}^1_D(\mathcal{T}_h)\to \mathbb{R}_{\ge 0}$ is the observation that it is bounded by the \textit{monotone 
    primal-dual a posteriori error estimator} $\eta_{F,h}^2\colon\mathcal{S}^1_D(\mathcal{T}_h)\to \mathbb{R}_{\ge 0}$, for every $v_h\in \mathcal{S}^1_D(\mathcal{T}_h)$ defined by 
	 \begin{align*}
     \smash{\eta_{F,h}^2(v_h)\coloneqq  \|F(\nabla v_h)-F(\nabla_{\!h} u_h^{\textit{cr}})\|^2_{L^2(\Omega;\mathbb{R}^d)}+\|F^*(\smash{z_h^{\textit{rt}}})-F^*(\Pi_h \smash{z_h^{\textit{rt}}})\|^2_{L^2(\Omega;\mathbb{R}^d)}\,.}
	 \end{align*}
    
	 \begin{lemma}
	     Let $\AAA\colon\mathbb{R}^d\to \mathbb{R}^d$ be defined by \eqref{special_case} for $p\in (1,\infty)$ and $\delta\ge 0$ and let $F\colon\mathbb{R}^d\to \mathbb{R}^d$ be defined by \eqref{eq:def_F} for the same $p\in (1,\infty)$ and $\delta\ge 0$. Then, there exists a constant $c>0$, depending only on $p\in (1,\infty)$ and $\delta\ge 0$, such that for every $v_h\in \mathcal{S}^1_D(\mathcal{T}_h)$, it holds
      \begin{align*}
          \smash{\eta_h^2(v_h)\leq c\,\eta_{F,h}^2(v_h)\,.}
      \end{align*}
	 \end{lemma}

    \begin{proof} In principle, the proof is already included as a special case in \cite[Corollary 4.2]{BKAFEM}. For the benefit of the reader, however, the proof briefly reproduced here.
	     
        The \hspace*{-0.1mm}discrete  \hspace*{-0.1mm}convex  \hspace*{-0.1mm}optimality \hspace*{-0.1mm}relation \hspace*{-0.1mm}\eqref{eq:pDirichletOptimalityCR2}, \hspace*{-0.1mm}that, \hspace*{-0.1mm}by \hspace*{-0.1mm}the \hspace*{-0.1mm}convexity \hspace*{-0.1mm}of \hspace*{-0.1mm}${\varphi\hspace*{-0.05em}\circ\hspace*{-0.05em}\vert\hspace*{-0.05em}\cdot\hspace*{-0.05em}\vert, \varphi^*\hspace*{-0.05em}\circ\hspace*{-0.05em}\vert\hspace*{-0.05em}\cdot\hspace*{-0.05em}\vert\colon\hspace*{-0.1em}\mathbb{R}^d\hspace*{-0.1em}\to\hspace*{-0.1em} \mathbb{R}}$, $D(\varphi\circ\vert\cdot\vert)=\AAA$ and $D(\varphi^*\circ\vert\cdot\vert)=\smash{\AAA^{-1}}$, it holds
	     \begin{alignat*}{2}
	         -\varphi(\vert \nabla_{\! h} u_h^{\textit{cr}}\vert)&\leq  -\varphi(\vert \nabla v_h\vert)-\AAA(\nabla v_h)\cdot(\nabla_{\! h} u_h^{\textit{cr}}-\nabla v_h)&&\quad\text{ a.e. in }\Omega\,,\\
	          -\varphi^*(\vert \Pi_h \smash{z_h^{\textit{rt}}}\vert)&\leq  -\varphi^*(\vert \smash{z_h^{\textit{rt}}}\vert)-\smash{\AAA^{-1}}(\Pi_h \smash{z_h^{\textit{rt}}})\cdot(\smash{z_h^{\textit{rt}}}-\Pi_h \smash{z_h^{\textit{rt}}})&&\quad\text{ a.e. in }\Omega\,,
	     \end{alignat*}
	     the orthogonality relation $\smash{\AAA^{-1}}(\Pi_h \smash{z_h^{\textit{rt}}})\perp \smash{z_h^{\textit{rt}}}-\Pi_h \smash{z_h^{\textit{rt}}} $ in $L^2(\Omega;\mathbb{R}^d)$, and \eqref{eq:hammera} 
        further~yield
	     \begin{align*}
	         \eta^2_h(v_h)&= \rho_{\varphi,\Omega}(\nabla v_h )-(\AAA(\nabla_{\! h} u_h^{\textit{cr}}),\nabla v_h-\nabla_{\! h} u_h^{\textit{cr}})_\Omega-\rho_{\varphi,\Omega}(\nabla_{\! h} u_h^{\textit{cr}} )
	         \\&\quad +\rho_{\varphi^*,\Omega}(\smash{z_h^{\textit{rt}}} )
	         -\rho_{\varphi^*,\Omega}( \Pi_h \smash{z_h^{\textit{rt}}} )
	         \\&\leq (\AAA(\nabla v_h)-\AAA(\nabla_{\! h} u_h^{\textit{cr}}),\nabla v_h-\nabla_{\! h} u_h^{\textit{cr}})_\Omega
	         \\&\quad+ (\smash{\AAA^{-1}}(\smash{z_h^{\textit{rt}}})-\smash{\AAA^{-1}}(\Pi_h \smash{z_h^{\textit{rt}}}),\smash{z_h^{\textit{rt}}}-\Pi_h \smash{z_h^{\textit{rt}}})_\Omega
	         \leq c\,\eta^2_{F,h}(v_h)\,,
	     \end{align*}
	     where $c>0$ depends only on $p\in (1,\infty)$ and $\delta\ge 0$.
    \end{proof}
  
   The following  efficiency result applies.
	 
	 \begin{theorem}[Efficiency]\label{thm:eff}
	 Let $\AAA\colon\mathbb{R}^d\to \mathbb{R}^d$ be defined by \eqref{special_case} for $p\in (1,\infty)$ and $\delta\ge 0$ and let $F\colon\mathbb{R}^d\to \mathbb{R}^d$ be defined by \eqref{eq:def_F} for the same $p\in (1,\infty)$ and $\delta\ge 0$. Then, there exists a
	        constant $c>0$, depending only on $p\in (1,\infty)$, $\delta\ge 0$, and the chunkiness ${\omega_0>0}$, such that  for every $v_h\in      \mathcal{S}^1_D(\mathcal{T}_h)$, it holds
	       \begin{align*}
	           \smash{\eta^2_{F,h}(v_h)\leq c\,\|F(\nabla v_h)-F(\nabla u)\|^2_{L^2(\Omega;\mathbb{R}^d)}+c\,\textup{osc}_h(f,v_h)\,.}
	       \end{align*}
	\end{theorem}
	
		\begin{proof} Appealing to Corollary \ref{cor:best-approx} and \eqref{lem:efficiency.1}, we have that
    \begin{align}\label{thm:rel_and_eff_P1.2}
        \|F(\nabla v_h)-F(\nabla_{\! h} u_h^{\textit{cr}})\|^2_{L^2(\Omega;\mathbb{R}^d)}&\leq c\,\big[\|F(\nabla v_h)-F(\nabla u)\|^2_{L^2(\Omega;\mathbb{R}^d)}+\textup{osc}_h(f,v_h)\big]\,,\\
        \smash{\rho_{(\varphi_{\vert \nabla v_h\vert})^*,\Omega}(h_{\mathcal{T}} f)}&\leq c\,\big[\|F(\nabla v_h)-F(\nabla u)\|^2_{L^2(\Omega;\mathbb{R}^d)}+\textup{osc}_h(f,v_h)\big]\,.\label{thm:rel_and_eff_P1.2.0}
    \end{align}
    On the other hand, using \eqref{eq:hammerf}, $\vert \Pi_h \smash{z_h^{\textit{rt}}}\vert=\vert \AAA(\nabla_{\! h} u_h^{\textit{cr}})\vert$,  $(\varphi^*)_{\vert \AAA(a)\vert}\sim (\varphi_{\vert a\vert})^*$ uniformly~in~${a\in \mathbb{R}^d}$,~the shift change \eqref{lem:shift-change.3}, the Orlicz-stability of $\Pi_h$ (cf. \cite[Corollary A.8, (A.12)]{kr-phi-ldg}), 
    and~\eqref{thm:rel_and_eff_P1.2.0},~we~find~that
    \begin{align}\label{thm:rel_and_eff_P1.3}
    \hspace{-2.5mm}\begin{aligned}\|F^*(\smash{z_h^{\textit{rt}}})-F^*(\Pi_h \smash{z_h^{\textit{rt}}})\|^2_{L^2(\Omega;\mathbb{R}^d)}&\leq 
        c\, \smash{\rho_{(\varphi^*)_{\vert\AAA(\nabla_{\! h} u_h^{\textit{cr}})\vert},\Omega}(h_{\mathcal{T}} f_h)}
        \\&\leq 
        c\, \smash{\rho_{(\varphi_{\vert \nabla_{\!h} u_h^{\textit{cr}}\vert})^*,\Omega}(h_{\mathcal{T}} f_h)}
        \\&\leq c_\varepsilon\, \smash{\rho_{(\varphi_{\vert \nabla v_h\vert})^*,\Omega}(h_{\mathcal{T}}f_h)}+\varepsilon\,c\,\|F(\nabla v_h)-F(\nabla_{\! h} u_h^{\textit{cr}})\|^2_{L^2(\Omega;\mathbb{R}^d)}
          \\&\leq c_\varepsilon\, \smash{\rho_{(\varphi_{\vert \nabla v_h\vert})^*,\Omega}(h_{\mathcal{T}} f)}+\varepsilon\,c\,\|F(\nabla v_h)-F(\nabla_{\! h} u_h^{\textit{cr}})\|^2_{L^2(\Omega;\mathbb{R}^d)}
           \\&\leq c_\varepsilon\,\big[ \|F(\nabla v_h)-F(\nabla u)\|^2_{L^2(\Omega;\mathbb{R}^d)}+\textup{osc}_h(f,v_h)\big]\\&\quad +\varepsilon\,c\,\|F(\nabla v_h)-F(\nabla_{\! h} u_h^{\textit{cr}})\|^2_{L^2(\Omega;\mathbb{R}^d)}\,.
            \end{aligned}\hspace{-2.5mm}
    \end{align}
    Adding \eqref{thm:rel_and_eff_P1.2} and \eqref{thm:rel_and_eff_P1.3} and, then, choosing $\varepsilon>0$ sufficiently small, we conclude the assertion.
    \end{proof}
	
	\section{Numerical experiments}\label{sec:experiments}
	
	\qquad In this section, we confirm the theoretical findings of  Section \ref{sec:a_priori} and Section \ref{sec:a_posteriori} via numerical experiments.
    
    All experiments were conducted using the finite element software package \mbox{\textsf{FEniCS}} (version 2019.1.0), cf.~\cite{LW10}. 
    All graphics are generated using the \textsf{Matplotlib} (version 3.5.1) library, cf.~\cite{Hun07}.
    
	\subsection{A priori error analysis}\label{subsec:num_a_priori}
	
	\qquad In \hspace{-0.1mm}this \hspace{-0.1mm}subsection, \hspace{-0.1mm}we \hspace{-0.1mm}confirm \hspace{-0.1mm}the \hspace{-0.1mm}theoretical \hspace{-0.1mm}findings \hspace{-0.1mm}of \hspace{-0.1mm}Section \hspace{-0.1mm}\ref{sec:a_priori}.
	\hspace{-0.1mm}More~\hspace{-0.1mm}precisely,~\hspace{-0.1mm}we~\hspace{-0.1mm}apply~\hspace{-0.1mm}the $\smash{\mathcal{S}^{1,\textit{cr}}_D(\mathcal{T}_h)}$-approximation \eqref{eq:pDirichletS1crD} of the variational 
	$p$-Dirichlet problem \eqref{eq:pDirichletW1p} with $\AAA\colon\mathbb{R}^d\to\mathbb{R}^d$, for every $a\in\mathbb{R}^d$ defined by
	\begin{align*}
	    \smash{\AAA(a) \coloneqq (\delta+\vert a\vert)^{p-2}a}\,,
	\end{align*}
	i.e., \eqref{special_case} applies, 
    where $\delta\coloneqq 1\textrm{e}{-}4$ and $p\in (1,\infty)$.~We~approximate the discrete primal solution ${u_h^{\textit{cr}}\in \smash{\mathcal{S}^{1,\textit{cr}}_D(\mathcal{T}_h)}}$  deploying the Newton line search algorithm  of \mbox{\textsf{PETSc}} (version     3.17.3), cf. \cite{LW10}, with an absolute tolerance of $\tau_{abs}= 1\textrm{e}{-}8$ and a relative tolerance of $\tau_{rel}=1\textrm{e}{-}10$. The linear~sys-tem \hspace*{-0.1mm}emerging \hspace*{-0.1mm}in \hspace*{-0.1mm}each \hspace*{-0.1mm}Newton \hspace*{-0.1mm}step \hspace*{-0.1mm}is \hspace*{-0.1mm}solved \hspace*{-0.1mm}using \hspace*{-0.1mm}a \hspace*{-0.1mm}sparse \hspace*{-0.1mm}direct \hspace*{-0.1mm}solver~\hspace*{-0.1mm}from~\hspace*{-0.1mm}\textup{\textsf{MUMPS}}~\hspace*{-0.1mm}(version~\hspace*{-0.1mm}5.5.0), cf. \cite{mumps}.
    
    For  our numerical experiments, we choose $\Omega= (-1,1)^2$, $\Gamma_D=\partial \Omega$, and as a manufactured solution of \eqref{eq:pDirichlet}, the function $u\in W^{1,p}_D(\Omega)$, for every $x\coloneqq (x_1,x_2)^\top\in \Omega$ defined by
    \begin{align*}
    	\smash{u(x)\coloneqq d(x)\,\vert x\vert^\alpha}\,,
    \end{align*}
    i.e., we set $f\coloneqq -\textup{div}\,\AAA(\nabla u)$. Here,
      $d\in C^\infty(\overline{\Omega})$, for every $x=(x_1,x_2)^\top \in\overline{\Omega}$ defined by 
    \begin{align*}
        \smash{d(x)\coloneqq (1-x_1^2)\,(1-x_2^2)}\,,
    \end{align*}
    is a smooth cut-off function enforcing the homogeneous Dirichlet boundary condition.~Moreover, we \hspace*{-0.1mm}choose \hspace*{-0.1mm}$\alpha\hspace*{-0.1em}=\hspace*{-0.1em}1.01$, \hspace*{-0.1mm}which \hspace*{-0.1mm}yields \hspace*{-0.1mm}that \hspace*{-0.1mm}$u\hspace*{-0.1em}\in W^{1,p}_D(\Omega)$ \hspace*{-0.1mm}satisfies \hspace*{-0.1mm}the \hspace*{-0.1mm}natural \hspace*{-0.1mm}regularity \hspace*{-0.1mm}assumption~\hspace*{-0.1mm}\eqref{natural_regularity}. As a result, appealing to Theorem \ref{thm:rate_u} and Theorem \ref{thm:rate_z}, we can expect the
    convergence~rate~$1$.
     
    We construct an initial triangulation $\mathcal
    T_{h_0}$, where $h_0=\smash{\frac{1}{\sqrt{2}}}$, by subdividing a~rectangular~Cartesian grid into regular triangles with different orientations.  Finer     triangulations $\mathcal T_{h_k}$, $k=1,\dots,9$, where $h_{k+1}=\frac{h_k}{2}$ for all $k=1,\dots,9$, are 
    obtained by
    regular subdivision of the previous grid: each triangle is subdivided
    into four equal triangles by connecting the midpoints of the edges, i.e., applying the red-refinement rule, cf. \cite{Car04}.
    
    Then, for the resulting series of triangulations $\mathcal T_k\coloneqq \mathcal T_{h_k}$, $k=1,\dots,9$, we apply the above Newton scheme to compute the discrete primal solution $u_k^{\textit{cr}}\coloneqq u_{h_k}^{\textit{cr}}\in \mathcal{S}^{1,\textit{cr}}_D(\mathcal{T}_k)$,~${k=1,\dots,9}$,~and,~then, resorting to the generalized Marini formula \eqref{eq:gen_marini}, the discrete dual solution ${z_k^{\textit{rt}}\coloneqq z_{h_k}^{\textit{rt}}\in\mathcal{R}T^0_N(\mathcal{T}_k)}$, $k=1,\dots,9$. Subsequently, we compute the error quantities
    \begin{align}\label{errors}
    	\left.\begin{aligned}
    		e_{F,k}&\coloneqq \|F(\nabla_{\!h_k}u_k^{\textit{cr}} )-F(\nabla u)\|_{L^2(\Omega;\mathbb{R}^2)}\,,\\
    		e_{F^*,k}&\coloneqq \|F^*(z_k^{\textit{cr}})-F^*(z)\|_{L^2(\Omega;\mathbb{R}^2)}\,,
    \end{aligned}\quad\right\}\quad k=1,\dots,9\,.
    \end{align}
    As estimation of the convergence rates,  the experimental order of convergence~(EOC)
    \begin{align*}
    	\texttt{EOC}_k(e_k)\coloneqq \frac{\log(e_k/e_{k-1})}{\log(h_k/h_{k-1})}\,, \quad k=1,\dots,9\,,
    \end{align*}
    where for every $k= 1,\dots,9$, we denote by $e_k$,
    either $e_{F,k}$ or 
    $e_{F^*,k}$, 
    ~\mbox{respectively},~is~recorded. 
    
    For different values of $p\in \{1.25,1.5,1.75, 2, 2.25, 2.5, 2.75, 3, 3.25, 3.5, 3.75,4\}$ and a
    series of triangulations~$\mathcal{T}_k$, $k = 1,\dots,9$,
    obtained by uniform mesh refinement as described~above,~the~EOC is
    computed and for $k = 4,\dots,9$ presented in Table~\ref{tab1}  and
    Table~\ref{tab3}.~In~each~case, we report a convergence ratio of about $\texttt{EOC}_k(e_k)\approx 1$, $k=4,\dots,9$, confirming the optimality of the  a priori error estimates established in Theorem~\ref{thm:rate_u} and  Theorem~\ref{thm:rate_z}.
    
\begin{table}[H]
     \setlength\tabcolsep{2pt}
 	\centering
 	\begin{tabular}{c |c|c|c|c|c|c|c|c|c|c|c|c|} \cmidrule(){2-13}
 	\hline 
		   
		    \multicolumn{1}{|c||}{\cellcolor{lightgray}\diagbox[height=1.1\line,width=0.11\dimexpr\linewidth]{\vspace{-0.6mm}$k$}{\\[-5mm] $p$}}
		    & \cellcolor{lightgray}1.25
		    & \cellcolor{lightgray}1.5 & \cellcolor{lightgray}1.75  & \multicolumn{1}{c||}{\cellcolor{lightgray}2.0}  &  \multicolumn{1}{c|}{\cellcolor{lightgray}2.25} & \cellcolor{lightgray}2.5  & \cellcolor{lightgray}2.75 &  \cellcolor{lightgray}3.0   & \cellcolor{lightgray}3.25  & \cellcolor{lightgray}3.5  & \cellcolor{lightgray}3.75  & \cellcolor{lightgray}4.0 \\ \hline\hline
			\multicolumn{1}{|c||}{\cellcolor{lightgray}$4$}  & 0.827 & 0.928 & 0.931 & \multicolumn{1}{c||}{0.936} & \multicolumn{1}{c|}{0.940} & 0.942 & 0.943 & 0.942 & 0.939 & 0.935 & 0.930 & 0.924 \\ \hline
			\multicolumn{1}{|c||}{\cellcolor{lightgray}$5$}  & 1.007 & 0.928 & 0.946 & \multicolumn{1}{c||}{0.955} & \multicolumn{1}{c|}{0.962} & 0.967 & 0.970 & 0.973 & 0.974 & 0.974 & 0.974 & 0.973 \\ \hline
			\multicolumn{1}{|c||}{\cellcolor{lightgray}$6$}  & 1.086 & 0.951 & 0.955 & \multicolumn{1}{c||}{0.962} & \multicolumn{1}{c|}{0.968} & 0.973 & 0.978 & 0.981 & 0.984 & 0.986 & 0.987 & 0.988 \\ \hline
			\multicolumn{1}{|c||}{\cellcolor{lightgray}$7$}  & 0.938 & 0.950 & 0.958 & \multicolumn{1}{c||}{0.965} & \multicolumn{1}{c|}{0.971} & 0.976 & 0.980 & 0.984 & 0.987 & 0.989 & 0.991 & 0.992 \\ \hline
			\multicolumn{1}{|c||}{\cellcolor{lightgray}$8$}  & 0.883 & 0.956 & 0.962 & \multicolumn{1}{c||}{0.967} & \multicolumn{1}{c|}{0.973} & 0.977 & 0.981 & 0.985 & 0.988 & 0.990 & 0.992 & 0.993 \\ \hline
			\multicolumn{1}{|c||}{\cellcolor{lightgray}$9$}  & 0.998 & 0.958 & 0.964 & \multicolumn{1}{c||}{0.969} & \multicolumn{1}{c|}{0.974} & 0.978 & 0.982 & 0.986 & 0.988 & 0.991 & 0.992 & 0.994 \\ \hline\hline
			\multicolumn{1}{|c||}{\cellcolor{lightgray}\small \textrm{expected}}   & 1.00  & 1.00  & 1.00  & \multicolumn{1}{c||}{1.00}  & \multicolumn{1}{c|}{1.00} & 1.00  & 1.00  & 1.00  & 1.00  & 1.00  & 1.00  & 1.00 \\ \hline
\end{tabular}\vspace{-2mm}
 	\caption{Experimental order of convergence: $\texttt{EOC}_k(e_{F,k})$,~${k=4,\dots,9}$.} 
 	\label{tab1}
 \end{table}\vspace{-5mm}

 \begin{table}[H]
     \setlength\tabcolsep{2pt}
 	\centering
 	\begin{tabular}{c |c|c|c|c|c|c|c|c|c|c|c|c|} \cmidrule(){2-13}
 	\hline 
		   
		    \multicolumn{1}{|c||}{\cellcolor{lightgray}\diagbox[height=1.1\line,width=0.11\dimexpr\linewidth]{\vspace{-0.6mm}$k$}{\\[-5mm] $p$}}
		    & \cellcolor{lightgray}1.25
		    & \cellcolor{lightgray}1.5 & \cellcolor{lightgray}1.75  & \multicolumn{1}{c||}{\cellcolor{lightgray}2.0}  &  \multicolumn{1}{c|}{\cellcolor{lightgray}2.25} & \cellcolor{lightgray}2.5  & \cellcolor{lightgray}2.75 &  \cellcolor{lightgray}3.0   & \cellcolor{lightgray}3.25  & \cellcolor{lightgray}3.5  & \cellcolor{lightgray}3.75  & \cellcolor{lightgray}4.0 \\ \hline\hline
			\multicolumn{1}{|c||}{\cellcolor{lightgray}$4$}  & 0.906 & 0.946 & 0.936 & \multicolumn{1}{c||}{0.937} & \multicolumn{1}{c|}{0.938} & 0.938 & 0.934 & 0.929 & 0.921 & 0.912 & 0.902 & 0.892 \\ \hline
			\multicolumn{1}{|c||}{\cellcolor{lightgray}$5$}  & 1.002 & 0.917 & 0.944 & \multicolumn{1}{c||}{0.958} & \multicolumn{1}{c|}{0.965} & 0.970 & 0.972 & 0.972 & 0.972 & 0.970 & 0.967 & 0.963 \\ \hline
			\multicolumn{1}{|c||}{\cellcolor{lightgray}$6$}  & 1.069 & 0.948 & 0.955 & \multicolumn{1}{c||}{0.965} & \multicolumn{1}{c|}{0.972} & 0.978 & 0.982 & 0.985 & 0.986 & 0.987 & 0.988 & 0.987 \\ \hline
			\multicolumn{1}{|c||}{\cellcolor{lightgray}$7$}  & 0.945 & 0.942 & 0.957 & \multicolumn{1}{c||}{0.968} & \multicolumn{1}{c|}{0.975} & 0.981 & 0.985 & 0.988 & 0.990 & 0.992 & 0.993 & 0.994 \\ \hline
			\multicolumn{1}{|c||}{\cellcolor{lightgray}$8$}  & 0.875 & 0.949 & 0.961 & \multicolumn{1}{c||}{0.970} & \multicolumn{1}{c|}{0.977} & 0.982 & 0.986 & 0.989 & 0.992 & 0.993 & 0.995 & 0.996 \\ \hline
			\multicolumn{1}{|c||}{\cellcolor{lightgray}$9$}  & 0.988 & 0.953 & 0.963 & \multicolumn{1}{c||}{0.971} & \multicolumn{1}{c|}{0.978} & 0.983 & 0.987 & 0.990 & 0.992 & 0.994 & 0.995 & 0.996 \\ \hline\hline
			\multicolumn{1}{|c||}{\cellcolor{lightgray}\small \textrm{expected}}   & 1.00  & 1.00  & 1.00  & \multicolumn{1}{c||}{1.00}  & \multicolumn{1}{c|}{1.00} & 1.00  & 1.00  & 1.00  & 1.00  & 1.00  & 1.00  & 1.00 \\ \hline
\end{tabular}\vspace{-2mm}
 	\caption{Experimental order of convergence: $\texttt{EOC}_k(e_{F^*,k})$,~${k=4,\dots,9}$.} 
 	\label{tab3}
 \end{table}

    \subsection{A posteriori error analysis}\label{subsec:num_a_posteriori}
    
    \qquad In \hspace{-0.1mm}this \hspace{-0.1mm}subsection, \hspace{-0.1mm}we \hspace{-0.1mm}confirm \hspace{-0.1mm}the \hspace{-0.1mm}theoretical \hspace{-0.1mm}findings \hspace{-0.1mm}of \hspace{-0.1mm}Section \hspace{-0.1mm}\ref{sec:a_posteriori}.
	\hspace{-0.1mm}More~\hspace{-0.1mm}precisely,~\hspace{-0.1mm}we~\hspace{-0.1mm}apply~\hspace{-0.1mm}the $\smash{\mathcal{S}^{1,\textit{cr}}_D(\mathcal{T}_h)}$-approximation \eqref{eq:pDirichletS1crD} of the variational 
	$p$-Dirichlet problem \eqref{eq:pDirichletW1p} with $\AAA\colon\mathbb{R}^d\to\mathbb{R}^d$, for every $a\in\mathbb{R}^d$ defined by
	\begin{align*}
	    \smash{\AAA(a) \coloneqq (\delta+\vert a\vert)^{p-2}a}\,,
	\end{align*}
    where $\delta\coloneqq 1\mathrm{e}{-}5$ and $p\in (1,\infty)$, in an adaptive mesh refinement algorithm based~on~local~refine-ment indicators $(\eta^2_{h,T}(v_h))_{T\in \mathcal{T}_h}$ associated with   the primal-dual a posteriori~error~estimator~$\eta^2_{h}(v_h)$. 
    More precisely, for every $v_h\in \smash{\mathcal{S}^1_D(\mathcal{T}_h)}$ and $T\in \mathcal{T}_h$, we define
    \begin{align*}
        \eta_{A,h,T}^2(v_h)&\coloneqq \rho_{\varphi,T}(\nabla v_h)-(\Pi_h \smash{z_h^{\textit{rt}}},\nabla v_h-\nabla_{\!h} u_h^{\textit{cr}})_T-\rho_{\varphi,T}(\nabla_{\!h} u_h^{\textit{cr}})\,,\\
        \eta_{B,h,T}^2
        &\coloneqq \rho_{\varphi^*,T}(\smash{z_h^{\textit{rt}}})-\rho_{\varphi^*,T}(\Pi_h \smash{z_h^{\textit{rt}}})\,,\\
        \eta_{h,T}^2(v_h)&\coloneqq \eta_{A,h,T}^2(v_h)+\eta_{B,h,T}^2
        \,.
    \end{align*}
    
    Before we present numerical experiments, we briefly outline the  details of the implementations. 
    In general, we follow the adaptive algorithm, cf. \cite{DK08}:\enlargethispage{5mm}
    
	\begin{algorithm}[AFEM]\label{alg:afem}
		Let $\varepsilon_{\textup{STOP}}\!>\!0$, $\theta\!\in\! (0,1)$ and  $\mathcal{T}_0$ a conforming initial  triangulation~of~$\Omega$. Then, for every $k\in \mathbb{N}\cup \{0\}$:
	\begin{description}[noitemsep,topsep=1pt,labelwidth=\widthof{\textit{('Estimate')}},leftmargin=!,font=\normalfont\itshape]
		\item[('Solve')]\hypertarget{Solve}{}
		Compute the discrete primal solution $u_k^{\textit{cr}}\in \smash{\mathcal{S}^{1,\textit{cr}}_D(\mathcal{T}_k)}$. Post-process $u_k^{\textit{cr}}\in \smash{\mathcal{S}^{1,\textit{cr}}_D(\mathcal{T}_k)}$ to obtain the discrete dual solution $z_k^{\textit{rt}}\in \smash{\mathcal{R}T^0_N(\mathcal{T}_k)}$ and a conforming approximation $v_k\in \mathcal{S}^1_D(\mathcal{T}_k)$ of the primal solution $u\in W^{1,p}_D(\Omega)$.
		\item[('Estimate')]\hypertarget{Estimate}{} Compute the local refinement indicators $\smash{(\eta^2_{k,T}(v_k))_{T\in \mathcal{T}_k}\coloneqq (\eta^2_{h_k,T}(v_k))_{T\in \mathcal{T}_k}}$. If $\smash{\eta^2_k(v_k)\coloneqq \eta^2_{h_k}(v_k)\leq \varepsilon_{\textup{STOP}}}$, then \textup{STOP}; otherwise, continue with step (\hyperlink{Mark}{'Mark'}).
		\item[('Mark')]\hypertarget{Mark}{}  Choose a minimal (in terms of cardinality) subset $\mathcal{M}_k\subseteq\mathcal{T}_k$ such that
		\begin{align*}
			\sum_{T\in \mathcal{M}_k}{\eta_{k,T}^2(v_k)}\ge \theta^2\sum_{T\in \mathcal{T}_k}{\eta_{k,T}^2(v_k)}\,.
		\end{align*}
		\item[('Refine')]\hypertarget{Refine}{} Perform a conforming refinement of $\mathcal{T}_k$ to obtain $\mathcal{T}_{k+1}$~such~that~each $T\in \mathcal{M}_k$  is refined in $\mathcal{T}_{k+1}$. 
		Increase~$k\mapsto k+1$~and~continue~with~('Solve').
	\end{description}
	\end{algorithm}

	\begin{remark}
			\begin{description}[noitemsep,topsep=1pt,labelwidth=\widthof{\textit{(iii)}},leftmargin=!,font=\normalfont\itshape]
				\item[(i)] \hspace{-2.5mm}The discrete primal solution $u_k^{\textit{cr}}\hspace{-0.1em}\in \hspace{-0.1em}\smash{\mathcal{S}^{1,\textit{cr}}_D(\mathcal{T}_k)}$  in step (\hyperlink{Solve}{'Solve'}) is computed~deploy-ing the Newton~line~search algorithm of \textsf{\textup{PETSc}}, cf. \cite{PETSc19}, with an absolute tolerance of about $\tau_{\textit{abs}}=1\mathrm{e}{-}8$ and a relative tolerance~of~about~${\tau_{\textit{rel}}=1\mathrm{e}{-}10}$. The linear system emerging in each Newton step is solved using a sparse direct solver~from~\textup{\textsf{MUMPS}} (version~5.5.0),~cf.~\cite{mumps}. 
				\item[(ii)] The reconstruction of the discrete dual solution $z_k^{\textit{rt}}\in \smash{\mathcal{R}T^0_N(\mathcal{T}_k)}$ in step (\hyperlink{Solve}{'Solve'}) is based on the generalized Marini formula \eqref{eq:gen_marini} and does not entail further computational costs.
                \item[(iii)] As a conforming approximation, we employ $v_k=I^{\textit{av}}_{h_k} u_k^{\textit{cr}}\in \mathcal{S}^1_D(\mathcal{T}_k)$.
				\item[(iv)] If not otherwise specified, we employ the parameter $\theta=\smash{\frac{1}{2}}$ in step (\hyperlink{Estimate}{'Mark'}).
				\item[(v)] To find the  set $\mathcal{M}_k\subseteq \mathcal{T}_k$ in step (\hyperlink{Mark}{'Mark'}), we~deploy~the~D\"orfler marking strategy,~cf.~\cite{Doe96}.
				\item[(vi)] The \hspace*{-0.1mm}(minimal) \hspace*{-0.1mm}conforming \hspace*{-0.1mm}refinement \hspace*{-0.1mm}of \hspace*{-0.1mm}$\mathcal{T}_k$ \hspace*{-0.1mm}with \hspace*{-0.1mm}respect \hspace*{-0.1mm}to \hspace*{-0.1mm}$\mathcal{M}_k$~\hspace*{-0.1mm}in~\hspace*{-0.1mm}step~\hspace*{-0.1mm}(\hyperlink{Refine}{'Refine'})~\hspace*{-0.1mm}is~\hspace*{-0.1mm}\mbox{obtained} by deploying the \textit{red}-\textit{green}-\textit{blue}-refinement algorithm, cf.~\cite{Car04}.\vspace{-1mm}
			\end{description}
	\end{remark}

    For our numerical experiments, we choose $\Omega\coloneqq \left(-1,1\right)^2 \setminus ([0,1]\times [-1,0])$, $ \Gamma_D \coloneqq \partial\Omega$, and as a manufactured solution of \eqref{eq:pDirichlet}, the function $u\in W^{1,p}_D(\Omega)$, in polar coordinates, for every $(r,\theta)^\top\in (0,\infty)\times (0,2\pi)$ defined by
	\begin{align*}
		\smash{u(r, \theta) \coloneqq d(r,\theta)\, r^\sigma \sin(\tfrac{2}{3}\, \theta)\,.}
	\end{align*}
    Here $d\in C^\infty(\overline{\Omega})$, in polar coordinates, for every $(r,\theta)^\top\in (0,\infty)\times (0,2\pi)$ defined by $ d(r,\theta)\coloneqq (1-r^2\,\cos^2(\theta))\,(1-r^2\,\sin^2(\theta))$, 
    enforces the homogeneous Dirichlet boundary condition. 
    Moreover, for every $p\in (1,\infty)$,~we~choose~$\sigma \coloneqq 1.01-\smash{\frac{1}{p}}$, which just yields that $F(\nabla u) \in W^{\frac{1}{2},2}(\Omega;\mathbb{R}^2)$, so \hspace{-0.1mm}that \hspace{-0.1mm}uniform \hspace{-0.1mm}mesh \hspace{-0.1mm}refinement \hspace{-0.1mm}is \hspace{-0.1mm}expected \hspace{-0.1mm}to \hspace{-0.1mm}yield \hspace{-0.1mm}an \hspace{-0.1mm}error \hspace{-0.1mm}decay \hspace{-0.1mm}for \hspace{-0.1mm}$e_{F,k}$, \hspace{-0.1mm}cf.~\hspace{-0.1mm}\eqref{errors},~\hspace{-0.1mm}with~\hspace{-0.1mm}rate~\hspace{-0.1mm}$\frac{1}{2}$.

 The initial triangulation $\mathcal{T}_0$ in Algorithm \ref{alg:afem} consists of 96 elements and 65 vertices.~In~Figure~\ref{fig:pDirichletPD}, for $p\in \{1.5,2,2.5,3\}$,  $ k=0,\dots,19$, if using adaptive mesh refinements, $ k=0,\dots,4$, if using uniform mesh refinement, and $v_k\coloneqq I_{h_k}^{\textit{av}}u_k^{\textit{cr}}\in\smash{\mathcal{S}^1_D(\mathcal{T}_k)}$, the primal-dual a posteriori error estimator $\eta^2_k(v_k)\hspace{-0.1em}\coloneqq \hspace{-0.1em}\eta^2_{h_k}(v_k)$
	 as well as the error quantity $\rho^2(v_k)\coloneqq \|F (\nabla v_k)-F (\nabla u) \|_{L^2(\Omega;\mathbb{R}^2)}^2$
	are plotted versus the number of degrees of freedom  $N_k \coloneqq  \textup{card}(\mathcal{S}_{h_k}^{i})$ (i.e., $\smash{h_k \sim N_k^{\smash{-\frac{1}{2}}}}$), in a \mbox{$\log\log$-plot}.~In~it, one clearly observes that uniform mesh refinement yields the expected reduced rate $\smash{h_k \sim N_k^{\smash{-\frac{1}{2}}}}$, while \hspace{-0.1mm}adaptive \hspace{-0.1mm}mesh \hspace{-0.1mm}refinement \hspace{-0.1mm}yields \hspace{-0.1mm}the \hspace{-0.1mm}improved \hspace{-0.1mm}quasi-optimal~\hspace{-0.1mm}rate~\hspace{-0.1mm}$\smash{h_k^2\hspace{-0.1em} \sim \hspace{-0.1em}N_k^{-1}}$.~\hspace{-0.1mm}In~\hspace{-0.1mm}\mbox{particular}, for every $p\in \{1.5,2,2.5,3\}$ and $k\hspace{-0.1em}=\hspace{-0.1em}0,\dots,19$, \hspace{-0.1mm}if \hspace{-0.1mm}using \hspace{-0.1mm}adaptive \hspace{-0.1mm}mesh \hspace{-0.1mm}refinement, \hspace{-0.1mm}and \hspace{-0.1mm}$ k\hspace{-0.1em}=\hspace{-0.1em}0,\dots,4$, \hspace{-0.1mm}if \hspace{-0.1mm}using \hspace{-0.1mm}uniform \hspace{-0.1mm}mesh \hspace{-0.1mm}\mbox{refinement}, the primal-dual a posteriori error estimator $\eta^2_k(v_k)$ defines a reliable and efficient upper \hspace{-0.1mm}bound \hspace{-0.1mm}for \hspace{-0.1mm}$\rho^2(v_k)$, \hspace{-0.1mm}confirming \hspace{-0.1mm}the \hspace{-0.1mm}findings~\hspace{-0.1mm}of~\hspace{-0.1mm}Theorem~\hspace{-0.1mm}\ref{thm:rel}~\hspace{-0.1mm}and~\hspace{-0.1mm}\mbox{Theorem}~\hspace{-0.1mm}\ref{thm:eff}.\vspace{-0.25cm}

    \begin{figure}[H]
        \centering
        \includegraphics[width=14cm]{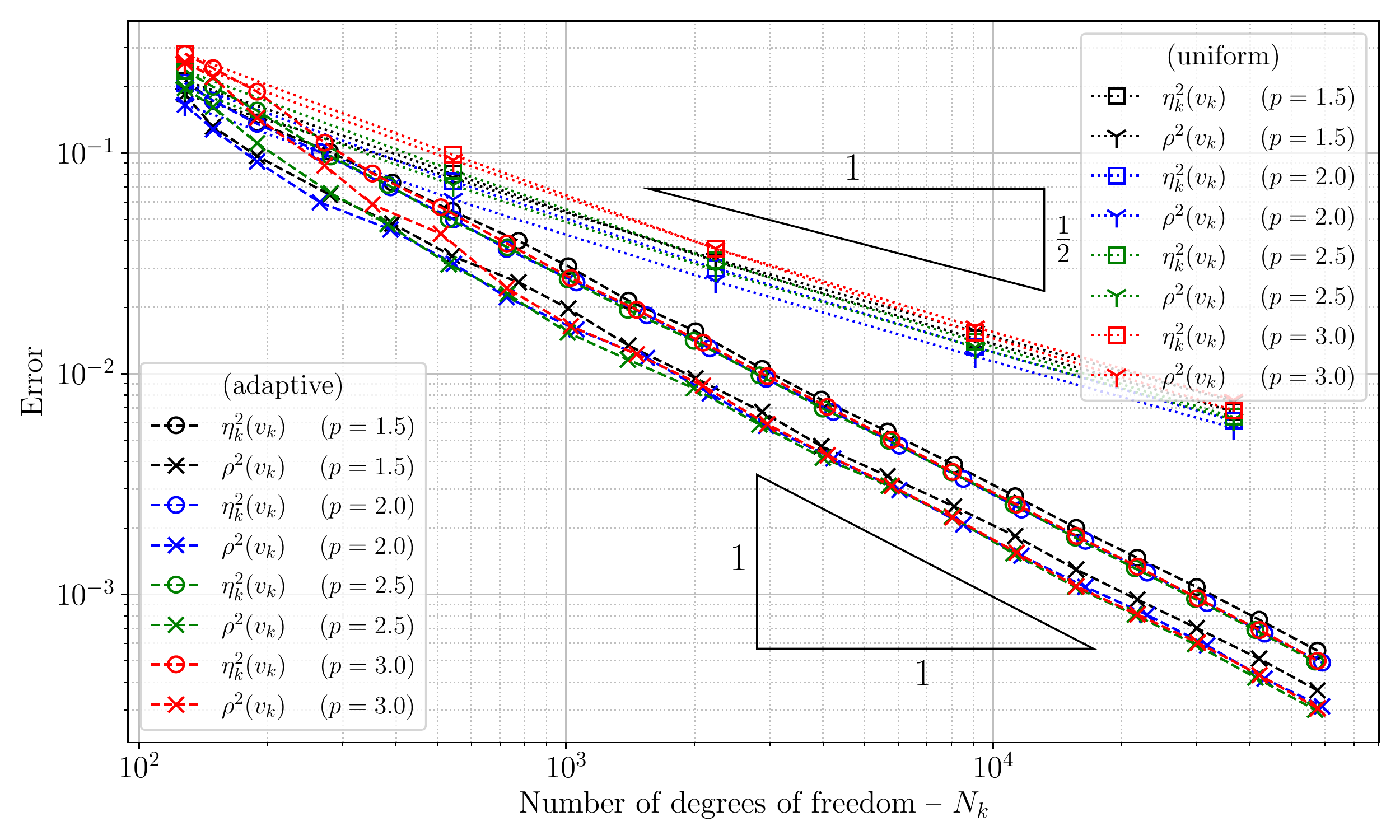}\vspace{-0.5cm}
        \caption{Plots of $\eta_k^2(v_k)$ and  $\rho^2(v_k)$ for $p\in \{1.5,2,2.5,3\}$ and $v_k\coloneqq I_{h_k}^{\textit{av}}u_k^{\textit{cr}}\in\smash{\mathcal{S}^1_D(\mathcal{T}_k)}$, using adaptive mesh refinement for $ k=0,\dots,19$ and using uniform mesh refinement for $ k=0,\dots,4$.\vspace{-1.25cm}}
        \label{fig:pDirichletPD}
    \end{figure}

    \appendix
    \section{Appendix}\label{sec:appendix}

    \hspace{5mm}In this appendix, we give a proof of the inequalities \eqref{eq:a6}.

    \begin{proposition}\label{prop:n-function}
		Let $\psi\colon \mathbb{R}_{\ge 0}\to \mathbb{R}_{\ge 0}$ be an $N$-function such that $\psi\in\Delta_2\cap \nabla_2$. Then,~for~every $v_h\in \mathcal{S}^{1,\textit{\textrm{cr}}}_D(\mathcal{T}_h)$, $m\in \{0,1\}$, $a\ge 0$, and $T\in \mathcal{T}_h$, we have that
		\begin{align*}
			\fint_T{\psi_a(h_T^m\vert \nabla_h^m(v_h-\textcolor{black}{I_h^{\textit{av}}}v_h)\vert)\,\textup{d}x}\leq  c_{\textit{av}}\sum_{S\in \mathcal{S}_h(T)\setminus\Gamma_N}{\fint_S{\psi_a(\vert\jump{v_h}_S\vert)\,\textup{d}s}}\,,
		\end{align*}
            where $c_{\textit{av}}>0$  depends only on $\Delta_2(\psi),\Delta_2(\psi^*)>0$ and the chunkiness $\omega_0>0$.
	\end{proposition}

	\begin{proof}
		Owing to \cite[Lemma A.2]{BKObstacle} together with \cite[Lemma 12.1]{EG21}, there exists a constant~$\overline c_{\textit{av}}>0$, depending only on the chunkiness $\omega_0>0$, such that 
		\begin{align}
			h_T^m\| \nabla_h^m(v_h-\textcolor{black}{I_h^{\textit{av}}}v_h)\|_{L^\infty(T;\mathbb{R}^{d^m})}\leq \overline c_{\textit{av}}\sum_{\smash{S\in \mathcal{S}_h(T)\setminus\Gamma_N}}{\fint_S{\vert \jump{v_h}_S\vert\,\textup{d}s}}\,.\label{prop:n-function1}
		\end{align}
	    Using in \eqref{prop:n-function1} the $\Delta_2$-condition and
		convexity of $\psi_a\colon \mathbb{R}_{\ge 0}\to \mathbb{R}_{\ge 0}$, $ a\ge 0 $, in particular, Jensen's inequality, and that $\sup_{h>0}{\sup_{T\in \mathcal{T}_h}{\textup{card}(\mathcal{S}_h(T)\setminus \Gamma_N)}}\leq c_{\mathcal{T}}$, where $c_{\mathcal{T}}>0$ depends only on the chunkiness $\omega_0>0$,	
  we find that
		\begin{align*}
		\fint_T{\psi_a(h_T^m\vert \nabla_h^m(v_h-\textcolor{black}{I_h^{\textit{av}}}v_h)\vert)\,\textup{d}x}
			&\leq 
		\Delta_2(\psi_a)^{\lceil \overline c_{\textit{av}}c_{\mathcal{T}}\rceil }
			\psi_a\bigg(\frac{1}{\textup{card}(\mathcal{S}_h(T)\setminus\Gamma_N)}\sum_{S\in\mathcal{S}_h(T)\setminus\Gamma_N}{\fint_S{\vert \jump{v_h}_S\vert\,\textup{d}s}}\bigg)
		\\&\leq  \Delta_2(\psi_a)^{\lceil \overline c_{\textit{av}}c_{\mathcal{T}}\rceil }\frac{1}{\textup{card}(\mathcal{S}_h(T)\setminus\Gamma_N)}\sum_{S\in\mathcal{S}_h(T)\setminus\Gamma_N}{\fint_S{\psi_a(\vert \jump{v_h}_S\vert)\,\textup{d}s}}\,.
		\end{align*}
        Eventually, using that $\sup_{a\ge 0}{\Delta_2(\psi_a)}<\infty$, cf. \cite[Lemma 22]{DK08}, we conclude the assertion.
		\end{proof}
	
		\begin{corollary}\label{cor:n-function}
			Let $\psi\colon \mathbb{R}_{\ge 0}\to \mathbb{R}_{\ge 0}$ be an $N$-function  such that $\psi\in\Delta_2\cap \nabla_2$. 
			Then,  for every $v_h\in  \mathcal{S}^{1,\textit{cr}}_D(\mathcal{T}_h)$, $m\in\{0,1\}$, $a\ge 0$ and $T\in \mathcal{T}_h$, we have that
		\begin{align*}
			\fint_T{\psi_a(h_T^m\vert \nabla_h^m(v_h-\textcolor{black}{I_h^{\textit{av}}}v_h)\vert)\,\textup{d}x}&\leq  c_{\textit{av}}\sum_{S\in\mathcal{S}_h(T)\setminus\Gamma_N}{\fint_S{\psi_a(h_S\vert\jump{\nabla_h v_h}_S\vert)\,\textup{d}s}}\\&\leq 
			 \tilde{c}_{\textit{av}}\fint_{\omega_T}{\psi_a(h_T\vert\nabla_h v_h\vert)\,\textup{d}x}\,,
		\end{align*}
		where $\tilde{c}_{\textit{av}}>0$ depends only  on $\Delta_2(\psi),\Delta_2(\psi^*)>0$ and the chunkiness $\omega_0>0$.
		\end{corollary}
		
		\begin{proof}
			Follows from Proposition \ref{prop:n-function}, if we exploit that $\jump{v_h}_S=\jump{\nabla_h v_h}_S\cdot(\textup{id}_{\mathbb{R}^d}-x_S)$~on~$S$~for~all $S\in \mathcal{S}_h$ and $v_h\in \mathcal{S}^{1,\textit{cr}}_D(\mathcal{T}_h)$ and the discrete trace inequality~\cite[Lemma~12.8]{EG21}.
		\end{proof}

    \textit{Acknowledgement.} The author is grateful for the stimulating discussions with S. Bartels.
 
    {\setlength{\bibsep}{0pt plus 0.0ex}\small


\begin{thebibliography}{10}

\bibitem{acerbi-fusco}
\bgroup\scshape{}E.~Acerbi\egroup{} and \bgroup\scshape{}N.~Fusco\egroup{},
  Regularity for minimizers of nonquadratic functionals: the case {$1<p<2$},
  \emph{J. Math. Anal. Appl.} \textbf{140} no.~1 (1989), 115--135.

\bibitem{mumps}
\bgroup\scshape{}P.~R. Amestoy\egroup{}, \bgroup\scshape{}I.~S. Duff\egroup{},
  \bgroup\scshape{}J.~Koster\egroup{}, and \bgroup\scshape{}J.-Y.
  L'Excellent\egroup{}, A fully asynchronous multifrontal solver using
  distributed dynamic scheduling,  \emph{SIAM Journal on Matrix Analysis and
  Applications} \textbf{23} no.~1 (2001), 15--41.

\bibitem{PETSc19}
\bgroup\scshape{}S.~Balay et al.\egroup{}, and \bgroup\scshape{}H.~Zhang\egroup{},
  {PETS}c {W}eb page, \url{https://www.mcs.anl.gov/petsc}, 2019.

\bibitem{BL1993a}
\bgroup\scshape{}J.~W. Barrett\egroup{} and \bgroup\scshape{}W.~B.
  Liu\egroup{}, Finite element approximation of the {$p$}-{L}aplacian,
  \emph{Math. Comp.} \textbf{61} no.~204 (1993), 523--537.

\bibitem{baliu94}
\bgroup\scshape{}J.~W. Barrett\egroup{} and \bgroup\scshape{}W.~B.
  Liu\egroup{}, Finite element approximation of degenerate quasilinear elliptic
  and parabolic problems,  in \emph{Numerical analysis 1993 (Dundee, 1993)},
  \emph{Pitman Res. Notes Math. Ser.} \textbf{303}, Longman Sci. Tech., Harlow,
  1994, pp.~1--16.

\bibitem{barliu}
\bgroup\scshape{}J.~W. Barrett\egroup{} and \bgroup\scshape{}W.~B.
  Liu\egroup{}, Quasi-norm error bounds for the finite element approximation of
  a non-{N}ewtonian flow,  \emph{Numer. Math.} \textbf{68} no.~4 (1994),
  437--456.

\bibitem{Bar21}
\bgroup\scshape{}S.~Bartels\egroup{}, Nonconforming discretizations of convex
  minimization problems and precise relations to mixed methods,  \emph{Comput.
  Math. Appl.} \textbf{93} (2021), 214--229. 
  \doi{10.1016/j.camwa.2021.04.014}.\enlargethispage{10mm}
  
\bibitem{BKAFEM}
\bgroup\scshape{}S.~Bartels\egroup{} and
  \bgroup\scshape{}A.~Kaltenbach\egroup{}, Explicit and efficient error
  estimation for convex minimization problems,  \emph{Math. Comp.} (2023),
  accepted. \doi{10.48550/ARXIV.2204.10745}.

  \bibitem{BKObstacle}
\bgroup\scshape{}S.~Bartels\egroup{} and
  \bgroup\scshape{}A.~Kaltenbach\egroup{}, Error analysis for a
  {C}rouzeix-{R}aviart approximation of the obstacle problem, 2023.
  \doi{10.48550/ARXIV.2302.01646}.
  
	\bibitem{BM20}
		\bgroup\scshape{}S.~Bartels\egroup{} and
		\bgroup\scshape{}M.~Milicevic\egroup{}, Primal-dual gap estimators for {\it a
			posteriori} error analysis of nonsmooth minimization problems,  \emph{ESAIM
			Math. Model. Numer. Anal.} \textbf{54} (2020), 1635--1660. 
		\doi{10.1051/m2an/2019074}.  

\bibitem{BDK12}
\bgroup\scshape{}L.~Belenki\egroup{}, \bgroup\scshape{}L.~Diening\egroup{}, and
  \bgroup\scshape{}C.~Kreuzer\egroup{}, Optimality of an adaptive finite
  element method for the {$p$}-{L}aplacian equation,  \emph{IMA J. Numer.
  Anal.} \textbf{32} no.~2 (2012), 484--510.
  \doi{10.1093/imanum/drr016}.

\bibitem{BDR10}
\bgroup\scshape{}L.~Berselli\egroup{}, \bgroup\scshape{}L.~Diening\egroup{},
  and \bgroup\scshape{}M.~R\r{u}\v{z}i\v{c}ka\egroup{}, Existence of strong
  solutions for incompressible fluids with shear dependent viscosities,
  \emph{Journal of Mathematical Fluid Mechanics} \textbf{12} (2010), 101--132.
  \doi{10.1007/s00021-008-0277-y}.

\bibitem{br-plasticity}
\bgroup\scshape{}L.~C. Berselli\egroup{} and \bgroup\scshape{}M.~R{\r u}{\v
  z}i{\v c}ka\egroup{}, Global regularity for systems with {$p$}-structure
  depending on the symmetric gradient,  \emph{Adv. Nonlinear Anal.} \textbf{9}
  no.~1 (2020), 176--192. \doi{10.1515/anona-2018-0090}.

\bibitem{BH70}
\bgroup\scshape{}J.~H. Bramble\egroup{} and \bgroup\scshape{}S.~R.
  Hilbert\egroup{}, Estimation of linear functionals on sobolev spaces with
  application to fourier transforms and spline interpolation,  \emph{SIAM J.
  Numer. Anal.} \textbf{7} no.~1 (1970), 112--124.
  \doi{10.1137/0707006}.

\bibitem{Sus96}
\bgroup\scshape{}S.~C. Brenner\egroup{}, Two-level additive schwarz
  preconditioners for nonconforming finite element methods,  \emph{Mathematics
  of Computation} \textbf{65} no.~215 (1996), 897--921.

\bibitem{Bre15}
\bgroup\scshape{}S.~C. Brenner\egroup{}, Forty years of the
  {C}rouzeix-{R}aviart element,  \emph{Numer. Methods Partial Differential
  Equations} \textbf{31} no.~2 (2015), 367--396.
  \doi{10.1002/num.21892}.

\bibitem{Car04}
\bgroup\scshape{}C.~Carstensen\egroup{}, An adaptive mesh-refining algorithm
  allowing for an {$H^1$} stable {$L^2$} projection onto {C}ourant finite
  element spaces,  \emph{Constr. Approx.} \textbf{20} no.~4 (2004), 549--564.
   \doi{10.1007/s00365-003-0550-5}.
   
\bibitem{CLY06}
\bgroup\scshape{}C.~Carstensen\egroup{}, \bgroup\scshape{}W.~Liu\egroup{}, and
  \bgroup\scshape{}N.~Yan\egroup{}, A posteriori {FE} error control for
  {$p$}-{L}aplacian by gradient recovery in quasi-norm,  \emph{Math.\ Comp.}
  \textbf{75} no.~256 (2006), 1599--1616. 
  \doi{10.1090/S0025-5718-06-01819-9}.

\bibitem{Cia78}
\bgroup\scshape{}P.~Ciarlet\egroup{}, \emph{The Finite Element Method for
  Elliptic Problems}, Society for Industrial and Applied Mathematics, 2002.
  \doi{10.1137/1.9780898719208}.

\bibitem{CR73}
\bgroup\scshape{}M.~Crouzeix\egroup{} and \bgroup\scshape{}P.-A.
  Raviart\egroup{}, Conforming and nonconforming finite element methods for
  solving the stationary {S}tokes equations. {I},  \emph{Rev. Fran\c{c}aise
  Automat. Informat. Recherche Op\'{e}rationnelle S\'{e}r. Rouge} \textbf{7}
  no.~{\rm R}-3 (1973), 33--75.

\bibitem{Dac08}
\bgroup\scshape{}B.~Dacorogna\egroup{}, \emph{Direct methods in the calculus of
  variations}, second ed., \emph{Applied Mathematical Sciences} \textbf{78},
  Springer, New York, 2008.

\bibitem{die-ett}
\bgroup\scshape{}L.~Diening\egroup{} and \bgroup\scshape{}F.~Ettwein\egroup{},
  Fractional estimates for non-differentiable elliptic systems with general
  growth,  \emph{Forum Math.} \textbf{20} no.~3 (2008), 523--556.

\bibitem{DK08}
\bgroup\scshape{}L.~Diening\egroup{} and \bgroup\scshape{}C.~Kreuzer\egroup{},
  Linear convergence of an adaptive finite element method for the
  {$p$}-{L}aplacian equation,  \emph{SIAM J. Numer. Anal.} \textbf{46} no.~2
  (2008), 614--638. \doi{10.1137/070681508}.

\bibitem{dkrt-ldg}
\bgroup\scshape{}L.~Diening\egroup{}, \bgroup\scshape{}D.~Kr\"oner\egroup{},
  \bgroup\scshape{}M.~R{\r u}{\v z}i{\v c}ka\egroup{}, and
  \bgroup\scshape{}I.~Toulopoulos\egroup{}, A {L}ocal {D}iscontinuous
  {G}alerkin approximation for systems with $p$-structure,  \emph{IMA J. Num.
  Anal.} \textbf{34} no.~4 (2014), 1447--1488. \doi{doi:
  10.1093/imanum/drt040}.

\bibitem{DR07}
\bgroup\scshape{}L.~Diening\egroup{} and
  \bgroup\scshape{}M.~R\r{u}\v{z}i\v{c}ka\egroup{}, Interpolation operators in
  {O}rlicz-{S}obolev spaces,  \emph{Numer. Math.} \textbf{107} no.~1 (2007),
  107--129. \doi{10.1007/s00211-007-0079-9}.

\bibitem{Doe96}
\bgroup\scshape{}W.~D\"{o}rfler\egroup{}, A convergent adaptive algorithm for
  {P}oisson's equation,  \emph{SIAM J. Numer. Anal.} \textbf{33} no.~3 (1996),
  1106--1124. \doi{10.1137/0733054}.

\bibitem{Ebmeyer05}
\bgroup\scshape{}C.~Ebmeyer\egroup{}, Global regularity in nikolskij spaces for
  elliptic equations with p-structure on polyhedral domains,  \emph{Nonlinear
  Analysis} \textbf{63} no.~6-7 (2005), e1--e9.

\bibitem{eb-liu}
\bgroup\scshape{}C.~Ebmeyer\egroup{} and \bgroup\scshape{}W.~B. Liu\egroup{},
  Quasi-norm interpolation error estimates for finite element approximations of
  problems with $p$--structure,  \emph{Numer. Math.} \textbf{100} (2005),
  233--258.

\bibitem{ELS04}
\bgroup\scshape{}C.~Ebmeyer\egroup{}, \bgroup\scshape{}W.~Liu\egroup{}, and
  \bgroup\scshape{}M.~Steinhauer\egroup{}, Global regularity in fractional
  order {S}obolev spaces for the {$p$}-{L}aplace equation on polyhedral
  domains,  \emph{Z. Anal. Anwendungen} \textbf{24} no.~2 (2005), 353--374.

\bibitem{ET99}
\bgroup\scshape{}I.~Ekeland\egroup{} and
  \bgroup\scshape{}R.~T\'{e}mam\egroup{}, \emph{Convex analysis and variational
  problems}, english ed., \emph{Classics in Applied Mathematics} \textbf{28},
  Society for Industrial and Applied Mathematics (SIAM), Philadelphia, PA,
  1999, 
  \doi{10.1137/1.9781611971088}.

\bibitem{EG21}
\bgroup\scshape{}A.~Ern\egroup{} and \bgroup\scshape{}J.~L. Guermond\egroup{},
  \emph{Finite Elements I: Approximation and Interpolation}, \emph{Texts in
  Applied Mathematics} no.~1, Springer International Publishing, 2021.
  \doi{10.1007/978-3-030-56341-7}.

\bibitem{giu1}
\bgroup\scshape{}E.~Giusti\egroup{}, \emph{Direct methods in the calculus of
  variations}, World Scientific Publishing Co. Inc., River Edge, NJ, 2003.

\bibitem{HA18}
\bgroup\scshape{}C.~Helanow\egroup{} and \bgroup\scshape{}J.~Ahlkrona\egroup{},
  Stabilized equal low-order finite elements in ice sheet modeling---accuracy
  and robustness,  \emph{Comput. Geosci.} \textbf{22} no.~4 (2018), 951--974.
  \doi{10.1007/s10596-017-9713-5}.

\bibitem{Hun07}
\bgroup\scshape{}J.~D. Hunter\egroup{}, Matplotlib: A 2d graphics environment,
  \emph{Computing in Science \& Engineering} \textbf{9} no.~3 (2007), 90--95.
  \doi{10.1109/MCSE.2007.55}.



\bibitem{kr-phi-ldg}
\bgroup\scshape{}A.~Kaltenbach\egroup{} and
  \bgroup\scshape{}M.~R{\r{u}}{\v{z}}i{\v{c}}ka\egroup{}, \emph{Convergence
  analysis of a {L}ocal {D}iscontinuous {G}alerkin approximation for nonlinear
  systems with {O}rlicz-structure}, 2022. \doi{10.48550/arXiv.2204.09984}

\bibitem{KZ22}
\bgroup\scshape{}A.~Kaltenbach\egroup{} and
  \bgroup\scshape{}M.~Zeinhofer\egroup{}, \emph{The Deep Ritz Method for
  Parametric $p$-Dirichlet Problems}, 2022. \doi{10.48550/arXiv.2207.01894}.

\bibitem{LLC18}
\bgroup\scshape{}D.~J. Liu\egroup{}, \bgroup\scshape{}A.~Q. Li\egroup{}, and
  \bgroup\scshape{}Z.~R. Chen\egroup{}, Nonconforming {FEM}s for the
  {$p$}-{L}aplace problem,  \emph{Adv. Appl. Math. Mech.} \textbf{10} no.~6
  (2018), 1365--1383.\doi{10.4208/aamm}.

\bibitem{LY01A}
\bgroup\scshape{}W.~Liu\egroup{} and \bgroup\scshape{}N.~Yan\egroup{},
  Quasi-norm a priori and a posteriori error estimates for the nonconforming
  approximation of {$p$}-{L}aplacian,  \emph{Numer. Math.} \textbf{89} no.~2
  (2001), 341--378.  \doi{10.1007/PL00005470}.

\bibitem{LY01B}
\bgroup\scshape{}W.~Liu\egroup{} and \bgroup\scshape{}N.~Yan\egroup{},
  Quasi-norm local error estimators for {$p$}-{L}aplacian,  \emph{SIAM J.
  Numer. Anal.} \textbf{39} no.~1 (2001), 100--127.
  \doi{10.1137/S0036142999351613}.

  \bibitem{Liu99}
\bgroup\scshape{}W.~B. Liu\egroup{}, Degenerate quasilinear elliptic equations
  arising from bimaterial problems in elastic-plastic mechanics,
  \emph{Nonlinear Anal.} \textbf{35} no.~4, Ser. A: Theory Methods (1999),
  517--529. \doi{10.1016/S0362-546X(98)00014-5}.

\bibitem{LW10}
\bgroup\scshape{}A.~Logg\egroup{} and \bgroup\scshape{}G.~N. Wells\egroup{},
  Dolfin: Automated finite element computing,  \emph{ACM Transactions on
  Mathematical Software} \textbf{37} no.~2 (2010). 
  \doi{10.1145/1731022.1731030}.

  \bibitem{MNRR96}
\bgroup\scshape{}J.~M\'{a}lek\egroup{}, \bgroup\scshape{}J.~Ne\v{c}as\egroup{},
  \bgroup\scshape{}M.~Rokyta\egroup{}, and
  \bgroup\scshape{}M.~R\r{u}\v{z}i\v{c}ka\egroup{}, \emph{Weak and
  measure-valued solutions to evolutionary {PDE}s}, \emph{Applied Mathematics
  and Mathematical Computation} \textbf{13}, Chapman \& Hall, London, 1996. \doi{10.1007/978-1-4899-6824-1}.

\bibitem{mret18}
\bgroup\scshape{}T.~Malkmus\egroup{},
  \bgroup\scshape{}M.~R\r{u}\v{z}i\v{c}ka\egroup{},
  \bgroup\scshape{}S.~Eckstein\egroup{}, and
  \bgroup\scshape{}I.~Toulopoulos\egroup{}, Generalizations of {SIP} methods to
  systems with {$p$}-structure,  \emph{IMA J. Numer. Anal.} \textbf{38} no.~3
  (2018), 1420--1451.   \doi{10.1093/imanum/drx040}.

\bibitem{Osw93}
\bgroup\scshape{}P.~Oswald\egroup{}, On the robustness of the
  bpx-preconditioner with respect to jumps in the coefficients,
  \emph{Math.\ Comp.} \textbf{68} no.~226 (1999), 633--650.

\bibitem{Pad97}
\bgroup\scshape{}C.~Padra\egroup{}, A posteriori error estimators for
  nonconforming approximation of some quasi-newtonian flows,  \emph{SIAM J.
  Numer. Anal.} \textbf{34} no.~4 (1997), 1600--1615.
  \doi{10.1137/S0036142994278322}.

\bibitem{RT75}
\bgroup\scshape{}P.-A. Raviart\egroup{} and \bgroup\scshape{}J.~M.
  Thomas\egroup{}, A mixed finite element method for 2nd order elliptic
  problems,  in \emph{Mathematical aspects of finite element methods ({P}roc.
  {C}onf., {C}onsiglio {N}az. delle {R}icerche ({C}.{N}.{R}.), {R}ome, 1975)},
  1977, pp.~292--315. Lecture Notes in Math., Vol. 606.\enlargethispage{10mm}

\bibitem{dr-nafsa}
\bgroup\scshape{}M.~R{\r u}{\v z}i{\v c}ka\egroup{} and
  \bgroup\scshape{}L.~Diening\egroup{}, Non--{N}ewtonian fluids and function
  spaces,  in \emph{Nonlinear Analysis, Function Spaces and Applications,
  Proceedings of {NAFSA} 2006 {P}rague}, \textbf{8}, 2007, pp.~{95--144}.

\bibitem{SZ90}
\bgroup\scshape{}L.~Scott\egroup{} and \bgroup\scshape{}S.~Zhang\egroup{},
  Finite element interpolation of nonsmooth functions satisfying boundary
  conditions,  \emph{Math.\ Comp.} \textbf{54}
  (1990), 483--493. \doi{10.1090/S0025-5718-1990-1011446-7}.

\bibitem{Ver95}
\bgroup\scshape{}R.~Verf{\"u}rth\egroup{}, A posteriori error estimates for
  nonlinear problems,  \emph{Math.\ Comp.} (1994), 445--475.

\bibitem{Zei90B}
\bgroup\scshape{}E.~Zeidler\egroup{}, \emph{Nonlinear functional analysis and
  its applications. {II}/{B} {N}onlinear monotone operators}, Springer-Verlag,
  New York, 1990.  \doi{10.1007/978-1-4612-0985-0}.

\end{thebibliography}

	\providecommand{\MR}[1]{}
\providecommand{\bysame}{\leavevmode\hbox to3em{\hrulefill}\thinspace}
\providecommand{\noopsort}[1]{}
\providecommand{\mr}[1]{\href{http://www.ams.org/mathscinet-getitem?mr=#1}{MR~#1}}
\providecommand{\zbl}[1]{\href{http://www.zentralblatt-math.org/zmath/en/search/?q=an:#1}{Zbl~#1}}
\providecommand{\jfm}[1]{\href{http://www.emis.de/cgi-bin/JFM-item?#1}{JFM~#1}}
\providecommand{\arxiv}[1]{\href{http://www.arxiv.org/abs/#1}{arXiv~#1}}
\providecommand{\doi}[1]{\url{https://doi.org/#1}}
\providecommand{\MR}{\relax\ifhmode\unskip\space\fi MR }
\providecommand{\MRhref}[2]{%
  \href{http://www.ams.org/mathscinet-getitem?mr=#1}{#2}
}
\providecommand{\href}[2]{#2}

}
\end{document}